\documentclass[12pt,a4paper]{article}

\usepackage{amsfonts,amsthm,amssymb,latexsym}
\usepackage{amsmath,setspace,amscd,verbatim}
\usepackage[alphabetic]{amsrefs}
\usepackage[pdftex]{graphicx}
\usepackage{graphics}
\usepackage[matrix,arrow]{xy}
\usepackage[active]{srcltx}
\usepackage[colorlinks=true]{hyperref}
\usepackage{mathrsfs}
\usepackage{breqn}
\usepackage{tikz-cd}
\usepackage{thmtools,thm-restate}
\usepackage{xcolor}
\usepackage[utf8]{inputenc}
\usepackage[english]{babel}

\title{Quantitative conditions for right-handedness~of~flows \\
	\vspace{1cm}
Conditions quantitatives \\pour les flots l\'evogyres}
\date{\today}
\author{Anna Florio\thanks{Partially supported by Fondation Sciences Math\'ematiques de Paris.}\ \thanks{Partially supported by ANR AAPG 2021 PRC CoSyDy: Conformally symplectic dynamics, beyond symplectic dynamics, ANR-CE40-0014.} \and Umberto Hryniewicz\thanks{Partially supported by the DFG SFB/TRR 191 `Symplectic Structures in Geometry, Algebra and Dynamics', Projektnummer 281071066-TRR 191.}}

\newcommand{\C}{\mathbb{C}}
\newcommand{\R}{\mathbb{R}}
\newcommand{\Z}{\mathbb{Z}}
\newcommand{\N}{\mathbb{N}}
\newcommand{\D}{\mathbb{D}}

\newcommand{\link}{\mathrm{link}}

\newcommand{\T}{\mathbb{T}}

\newcommand{\supp}{{\rm supp}}
\newcommand{\wind}{{\rm wind}}
\newcommand{\PP}{\mathbb{P}}
\newcommand{\hor}{\mathrm{hor}}
\newcommand{\ver}{\mathrm{vert}}

\newcommand\blfootnote[1]{%
  \begingroup
  \renewcommand\thefootnote{}\footnote{#1}%
  \addtocounter{footnote}{-1}%
  \endgroup
}

\theoremstyle{plain}
\newtheorem{theorem}{Theorem}[section]

\newtheorem{proposition}[theorem]{Proposition}
\newtheorem{lemma}[theorem]{Lemma}
\newtheorem{corollary}[theorem]{Corollary}

\theoremstyle{definition}
\newtheorem{definition}[theorem]{Definition}

\theoremstyle{remark}
\newtheorem{remark}[theorem]{Remark}
\newtheorem{notation}[theorem]{Notation}

\definecolor{anna}{rgb}{0,.6,0}

\definecolor{umberto}{rgb}{0.773,0.294,0.549}

\begin{document}

\maketitle

\begin{abstract}
We give a numerical condition for right-handedness of a dynamically convex Reeb flow on the $3$-sphere. Our condition is stated in terms of an asymptotic ratio between the amount of rotation of the linearised flow and the linking number of trajectories with a periodic orbit that spans a disk-like global surface of section. As an application, we find an explicit constant $\delta_* < 0.7225$ such that if a Riemannian metric on the $2$-sphere is $\delta$-pinched with $\delta > \delta_*$, then its geodesic flow lifts to a right-handed flow on the $3$-sphere. In particular, all finite non-empty collections of periodic orbits of such a geodesic flow bind open books whose pages are global surfaces of section.\\

Nous donnons une condition numérique garantissant qu'un flot de Reeb dynamiquement convexe sur la sphère $S^3$ soit lévogyre. Notre condition fait intervenir un certain rapport asymptotique entre le taux de rotation du flot linéarisé et le nombre d'enlacement entre les trajectoires et une orbite périodique qui borde une surface de section globale de type disque. En application, nous trouvons une constante explicite $\delta_*<0.7225$ telle que si une métrique riemannienne sur la sphère $S^2$ est $\delta$-pincée avec $\delta>\delta_*$, alors son flot géodésique se relève en un flot lévogyre sur la sphère $S^3$. En particulier, toute collection non vide finie d'orbites périodiques d'un tel flot géodésique borde un livre ouvert dont les pages sont des surfaces de section globale.
\end{abstract}

\blfootnote{\textup{2010} \textit{Mathematics Subject Classification}: \textup{37Exx, 53D25}}

\setcounter{tocdepth}{2}

\tableofcontents

%\newpage

\section{Introduction}

Right-handed vector fields, introduced by Ghys in~\cite{ghys}, form a special class of non-singular vector fields on homology $3$-spheres. Their flows will also be called right-handed, for simplicity. Roughly speaking, all pairs of trajectories of such a flow link positively. Ghys formalised this definition in terms of positivity of the quadratic linking form, which assigns some kind of linking number to any pair of invariant Borel probability measures; see also~\cite{Arn86}. Right-handedness has interesting dynamical implications as the following statement demonstrates.

\begin{theorem}[Ghys~\cite{ghys}]
\label{thm_ghys}
Every {non-empty} finite collection of periodic orbits of a right-handed flow binds an open book whose pages are global surfaces of section.
\end{theorem}

In particular, the above statement imposes strong restrictions on the periodic orbits that can arise in right-handed flows: knots or links that are not fibred cannot be realised. 
%For the sake of completeness, and for convenience of the reader, a full proof of Theorem~\ref{thm_ghys} is presented. We shall see that 
{Both the definition of right-handedness (Definition~\ref{def_right_handedness}) and the proof of Theorem~\ref{thm_ghys} (\cite[appendix~B]{FHArxiv_v1}) 
can actually be given independently of the quadratic linking form.}

Right-handedness is difficult to check. It can be checked for certain special flows, like the Hopf flow or other integrable flows. Since in~\cite{ghys} it is stated that right-handedness is {a $C^1$-open condition, one knows that there exist}
%\marginpar{referee asks why ``in theory'': how to explain that it is due to the fact that Ghys statement is not fully justified...?} 
more interesting {right-handed flows near these simple examples, but this abstract reasoning does not give them explicitly.} It is precisely this general lack of explicit examples that motivates our work. One exception is provided by the work of Dehornoy in~\cite{Deh}. Sometimes it is more natural to talk about left-handedness, which is equivalent to right-handedness when the ambient orientation is reversed. Dehornoy showed that the geodesic flow on the unit tangent bundle of a hyperbolic $n$-conic $2$-sphere is left-handed if, and only if, $n=3$. Then, using a version of Gromov's geodesic rigidity~\cite{gromov}, the case of arbitrary negative sectional curvature is reduced to the hyperbolic case. 
%The Lorenz vector field plays a central role in the discussion of~\cite{ghys} although, strictly speaking, it is not right-handed.

To state our main application, consider the polynomial $P(x)=4x^3-2x^2-1$, which has a unique real root $x_*$. It satisfies $0.84 < x_* <0.85$. Set $\delta_* := x_*^2$. { Given $\delta \in (0,1]$, a Riemannian metric on $S^2$ is said \textit{$\delta$-pinched} if $\delta \leq K_{\min}/K_{\max}$, where $K_{\min},K_{\max}$ denote the minimum and the maximum of the Gaussian curvature.}

\begin{theorem}\label{thm_pinched}
If $\delta > \delta_*$, in particular if $\delta \geq 0.7225$, 
then the geodesic flow of a $\delta$-pinched Riemannian metric on~$S^2$ lifts to a right-handed flow on $S^3$.
\end{theorem}

As mentioned before, the main motivation for such a statement is that it is not a perturbative result, hence it can be used to check right-handedness in an explicit set of flows that are far from integrable.

Other implications of right-handedness besides Theorem~\ref{thm_ghys} have been obtained. For instance, Dehornoy and Rechtman showed in~\cite{DehRec} that if an invariant measure of a right-handed flow can be approximated by long periodic orbits, then their Seifert genera grow proportionally to the square of the period, the proportionality constant being the helicity of the given invariant measure.

More generally, we look for numerical conditions for right-handedness within the class of Reeb flows of dynamically convex contact forms on $S^3$. This class was introduced by Hofer, Wysocki and Zehnder in~\cite{convex}. They showed that this class is rich enough to provide applications: by~\cite[Theorem~3.4]{convex} the Hamiltonian flow on a strictly convex energy level in a $4$-dimensional symplectic vector space is the Reeb flow of a dynamically convex contact form. Harris and Paternain showed in~\cite{HP} that the geodesic flow of a Finsler metric on the two-sphere with reversibility~$r$ is dynamically convex if the flag curvatures are pinched by more than $(r/(r+1))^2$. Theorem~\ref{thm_pinched} will be deduced from Theorem~\ref{thm_main_0} below, which provides an abstract condition for right-handedness of a dynamically convex Reeb flow on $S^3$ in terms of a dynamical pinching condition. Theorem~\ref{thm_main_0} also implies Theorem~\ref{thm_main_1} which gives right-handedness on strictly convex energy levels in terms of a relation between the curvatures and the return times of a disk-like global surface of section.

It was explained to us by Dehornoy in private communication that there are Riemannian metrics on $S^2$ with strict positive curvature whose geodesic flows do not lift to right-handed flows on $S^3$. Consider the infimum $\delta_0$ of all $\delta \in (0,1]$ with the following property: if a Riemannian metric on $S^2$ has sectional curvature pinched by at least $\delta$ then its geodesic flow lifts to a right-handed flow on $S^3$. Dehornoy's examples show that $\delta_0>0$. Together with Theorem~\ref{thm_pinched} we get $0<\delta_0 \leq \delta_* <0.7225$. We are led to ask:

\medskip

\noindent {\bf Question.} What is the value of $\delta_0$? 

\medskip

\subsection{Global surfaces of section}

Let $X$ be a smooth vector field on a closed and oriented $3$-manifold $M$. Its flow is denoted by $\phi^t$.

\begin{definition}\label{defn_GSS}
A \emph{global surface of section} 
%(GSS) 
for $\phi^t$, or for $X$, is a smooth, embedded and compact surface $\Sigma \hookrightarrow M$, such that $\partial \Sigma$ consists of periodic orbits or is empty, $X$~is transverse to $\Sigma\setminus\partial\Sigma$, and for every $x\in M$ one finds $t_-<0<t_+$ such that $\phi^{t_\pm}(x) \in \Sigma$.
\end{definition}

\begin{remark}\label{rmk_orientation_gss}
The orientation of $M$ and the co-orientation of $\Sigma\setminus\partial\Sigma$ induced by $X$ together orient $\Sigma$. In this paper we always assume that global surfaces of section are oriented in this way. 
\end{remark}

%\begin{definition}
%When $\Sigma$ is a global surface of section and its boundary orientation coincides with the flow orientation on $\partial\Sigma$, then $\Sigma$ is said to be \emph{positive}.
%\end{definition}

\begin{remark}
From a dynamical perspective a global surface of section is a valuable tool since one can deduce dynamical properties of the flow from those of the associated \emph{first return map}. To define it we need first to consider the \emph{return time function} $\tau: \Sigma\setminus\partial\Sigma \to (0,+\infty)$, $\tau(x) = \min \{ t>0 \mid \phi^t(x) \in \Sigma \}$. One can use $\tau$ to define the return map $\psi: \Sigma\setminus\partial\Sigma \to \Sigma\setminus\partial\Sigma$ by $\psi(x) = \phi^{\tau(x)}(x)$. It follows from Definition~\ref{defn_GSS} that $\psi$ is a smooth diffeomorphism.
\end{remark}

Consider the {normal bundle to the flow $\xi = TM/\R X \to M$, and denote by $\PP_+\xi$ the circle bundle $(\xi\setminus0)/\R_+ \to M$. The latter is isomorphic to the unit normal bundle in~$\xi$ once a choice of metric is fixed.} { The equivalence class in~$\PP_+\xi$ of a non-zero vector $\nu \in \xi$ will be denoted by $\R_+\nu$.} The linearised flow $D\phi^t$ induces flows on $\xi$ and on $\PP_+\xi$, both denoted by $D\phi^t$ with no fear of ambiguity. These flows cover $\phi^t$. 
%The generating vector field of $D\phi^t$ on $\PP_+\xi$ is denoted by~{\color{red} $X_{\PP}$}. 
%$\tilde X$ $X_{\PP}$. 
Both $\xi$ and $\PP_+\xi$ get oriented as bundles by the flow and the ambient orientation. Now let $\gamma:\R/T\Z \to M$ be a periodic orbit of $\phi^t$, where $T>0$ is the primitive period. The total space $\T_\gamma$ of the trivial circle bundle $\gamma(T\cdot)^*\PP_+\xi \to \R/\Z$, which we see as a submanifold of $\PP_+\xi$, is a $D\phi^t$-invariant torus. The dynamics of $D\phi^t$ on $\T_\gamma$ will be referred to as \textit{linearised polar dynamics} along $\gamma$. For each boundary orbit $\gamma$ of a global surface of section $\Sigma$, consider 
$$ 
\nu^\Sigma\gamma = \{ \R_+\nu \mid \nu\in T\Sigma|_\gamma \ \text{is outward pointing} \} \, . %{\color{red} \subset (\xi\setminus0)/\R_+}. 
$$ 
It follows that $\nu^\Sigma\gamma / \R X$ {defines} the graph of a section of $\gamma(T\cdot)^*\PP_+\xi$ and, as such, {determines} a smooth submanifold of~$\T_\gamma$.

\begin{definition}
The global surface of section $\Sigma$ is called \emph{strong} if the associated return time function is bounded away from zero and bounded from above. It is called \emph{$\partial$-strong} if $\nu^\Sigma\gamma / \R X$ is a global surface of section for the linearised polar dynamics along every $\gamma \subset \partial\Sigma$. 
%\marginpar{To discuss here the modification proposed by referee}
\end{definition}

{The reader will easily check that $\partial$-strong implies strong.

\begin{remark}
As suggested by the referee, the notion of $\partial$-strong global surface of section could be more succinctly and geometrically described as follows.
Consider $z \in \partial\Sigma$ and $v \in T_zM \setminus \R X_z$. 
%Then $t \mapsto D\phi^t\cdot v$ is a curve inside the $4$-dimensional space $TM|_{\partial \Sigma}$. 
It follows that $\Sigma$ is $\partial$-strong if, and only if, for any~$z$ and~$v$ as above the curve $t \mapsto D\phi^t\cdot v$ intersects the $3$-dimensional manifold $T\Sigma|_{\partial\Sigma}$ infinitely many often in the future and in the past, and all such intersections are transverse in the $4$-dimensional space $TM|_{\partial \Sigma}$.
\end{remark}
}

In~\cite{Po} Poincar\'e described annular global surfaces of section for certain regimes of the planar circular restricted three-body problem (PCR3BP). In the same paper one finds his \textit{last geometric theorem}, proved by Birkhoff in~\cite{birkhoff_PB} and nowadays known as the Poincar\'e-Birkhoff theorem. Poincar\'e applied his statement to the return map of the sections he found for the PCR3BP to obtain infinitely many periodic orbits. In~\cite{birkhoff} Birkhoff explained that positively curved Riemannian geodesic flows on~$S^2$ always have annulus-like global surfaces of section. Birkhoff's result admits a generalisation to Reeb flows in dimension three, see~\cite{HSW}. Birkhoff's section plays an important role in the proof of Theorem~\ref{thm_pinched}.

\subsection{Reeb flows}

A \emph{contact form} $\lambda$ on a $3$-manifold $M$ is a $1$-form such that $\lambda \wedge d\lambda$ defines a volume form. We always consider $M$ equipped with the orientation induced by $\lambda \wedge d\lambda$. The associated \emph{Reeb vector field} $X$ is implicitly determined by the equations
\begin{equation}
i_Xd\lambda=0 \qquad\qquad i_X\lambda=1
\end{equation}
and its flow $\phi^t$ is referred to as the Reeb flow. The plane field $\xi = \ker\lambda$ is called a \emph{contact structure} and can be seen as a representation of $TM/\R X$. It becomes a symplectic vector bundle with $d\lambda$. Note that $\phi^t$ preserves $\lambda$, hence also $\xi$, $d\lambda$ and~$\lambda\wedge d\lambda$.

A periodic orbit $\gamma$ of $\phi^t$ has a \emph{Conley-Zehnder} index relative to a symplectic trivialization of $(\xi,d\lambda)|_\gamma$. This index is an integer that can be described in terms of transverse rotation numbers, see Remark~\ref{rmk_CZ_rot} for details.

\begin{definition}[Hofer, Wysocki and Zehnder~\cite{char2}]
A contact form $\lambda$ on a closed $3$-manifold $M$ is \emph{dynamically convex} if the first Chern class $c_1(\xi,d\lambda)$ vanishes on $\pi_2(M)$, and every contractible periodic Reeb orbit has Conley-Zehnder index at least equal to three in a symplectic frame that extends to a capping disk.
\end{definition}

One of the motivations for the above definition is the following remarkable result.

\begin{theorem}[Hofer, Wysocki and Zehnder~\cite{convex}]\label{thm_HWZ}
The Reeb flow of every dynamically convex contact form on $S^3$ admits a disk-like global surface of section.
\end{theorem}

Using the methods from~\cite{convex} one can prove a characterisation result for periodic Reeb orbits bounding a disk-like 
%GSS.
global surface of section.
{Let us recall the self-linking number of a knot $K$ transverse to a contact structure on $S^3$. 
Consider a knot $K'$ obtained from $K$ by pushing it in the direction of a global trivialization of the contact structure.
Choose an orientation for $K$ and orient $K'$ accordingly.
The self-linking number of $K$ is the linking number of $K$ and $K'$.}

\begin{theorem}[\cite{openbook}]\label{thm_JSG}
A periodic Reeb orbit of a dynamically convex contact form on $S^3$ bounds a disk-like global surface of section if, and only if, it is unknotted and has self-linking number $-1$.
\end{theorem}

\begin{remark}\label{rmk_JSG_strong}
The 
%GSSs 
global surfaces of section obtained from~\cite{convex,openbook} are closures of projections of pseudo-holomorphic planes in symplectisations, and as such are smooth in the interior but only $C^1$ up to the boundary. As shown in the proof of Proposition~\ref{prop_main}, 
%we will discussIn Appendix~\ref{app_coordinates} it will be proved that 
they can be $C^1$-perturbed to smooth $\partial$-strong global surfaces of section.
\end{remark}

If $M = S^3$ then the contact structure $(\xi,d\lambda)$ is trivial as a symplectic vector bundle. 
%The linearised flow on $\xi$ descends to a flow on the $S^1$-bundle $(\xi\setminus 0)/\R_+$ still denoted by $D\phi^t$, whose generating vector field is denoted by~$\tilde X$. 
A global symplectic trivialisation $\sigma$ of $(\xi,d\lambda)$ induces a trivialisation $\PP_+\xi \simeq S^3\times\R/2\pi\Z$. The map obtained by composing this diffeomorphism with the projection onto the second factor will be denoted by
\begin{equation}
\label{angular_coord_sigma}
\Theta_\sigma : \PP_+\xi \to \R/2\pi\Z.
\end{equation}
Let the periodic Reeb orbit $\gamma_0$ bound a $\partial$-strong disk-like global surface of section~$\Sigma$. Stokes theorem implies that $\Sigma$ orients $\gamma_0=\partial\Sigma$ along the flow. For every $x\in S^3\setminus\gamma_0$ we denote
\begin{equation}
t^\Sigma_+(x) = \inf \{t\geq0 \mid \phi^t(x) \in\Sigma \} \qquad t^\Sigma_-(x) = \sup \{t\leq0 \mid \phi^t(x) \in\Sigma \}
\end{equation}
which are uniformly bounded functions that vanish and are discontinuous on $\Sigma \setminus\gamma_0$, and are non-vanishing and smooth on $S^3\setminus\Sigma$. Given $T>0$ consider the interval 
\begin{equation}\label{interval_(T,x)}
I(T,x;\Sigma) = [t^\Sigma_-(x),T+t^\Sigma_+(\phi^T(x))]
\end{equation}
and denote by $k(T,x;\Sigma) \subset S^3\setminus\gamma_0$ any loop obtained by concatenating to $\phi^{I(T,x;\Sigma)}(x)$ a path in the interior of $\Sigma$ from $\phi^{T+t^\Sigma_+(\phi^T(x))}(x)$ to $\phi^{t^\Sigma_-(x)}(x)$. For $u \in (\xi\setminus 0)/\R_+$ denote by $t \in \R \mapsto \widetilde\Theta_\sigma(t,u) \in \R$ a continuous lift of $t \mapsto \Theta_\sigma(D\phi^t \cdot u)$. Finally we define
\begin{equation}\label{def_number_kappa}
\kappa(\gamma_0) \ = \ \liminf_{T\to+\infty} \ \left( \ \inf_{x,u} \ \frac{\widetilde\Theta_\sigma(T,u)-\widetilde\Theta_\sigma(0,u)}{\link(k(T,x;\Sigma),\gamma_0)} \right)
\end{equation}
where the infimum is taken over all pairs $(x,u)$, where $x\in S^3\setminus\gamma_0$ and $u \in (\xi_x\setminus0)/\R_+$.

\begin{remark}
%In Sections~\ref{section comparing link} and~\ref{section comparing frame}, we will prove that $\kappa(\gamma_0)$ does not depend on~$\Sigma$ or on~$\sigma$; see Corollary~\ref{independance delta strong} and Lemma~\ref{independance frame}.
The quantity $\kappa(\gamma_0)$ does not depend on the choice of $\partial$-strong disk-like global surface of section spanned by $\gamma_0$. Moreover, since we work on $S^3$, it also does not depend on the chosen global symplectic trivialisation $\sigma$.
See~\cite[appendix~A]{FHArxiv_v1} for detailed proofs of these statements.
%: since we are in $S^3$, all global symplectic trivialisations are homotopic and this is enough to imply that the corresponding quantities are the same. 
%We refer to~\cite[appendix~A]{FHArxiv_v1} for detailed proofs of these statements.}
\end{remark}

Our abstract result reads as follows.

\begin{theorem}\label{thm_main_0}
Let $\lambda$ be a dynamically convex contact form on $S^3$, and $\gamma_0$ be an unknotted periodic Reeb orbit with self-linking number $-1$. If $\kappa(\gamma_0)>2\pi$ then the Reeb flow of $\lambda$ is right-handed.
\end{theorem}

As mentioned before, a rich source of dynamically convex contact forms are the strictly convex energy levels in $\R^4$. Consider $\R^4$ equipped with coordinates $(q_1,p_1,q_2,p_2)$ and symplectic form $\omega_0 = d\lambda_0$, where $\lambda_0$ is the $1$-form 
\begin{equation*}
\lambda_0 = \frac{1}{2} ( p_1dq_1 - q_1dp_1 + p_2dq_2 - q_2dp_2 )\,.
\end{equation*}
It follows that $\lambda_0$ is a contact form on hypersurfaces that are star-shaped with respect to the origin.

Let $C \subset \R^4$ be a smooth convex body with the origin in the interior. Denote by $\nu_C$ the {\it gauge function} of~$C$, which is defined to be the unique $1$-homogeneous function $\nu_C$ satisfying $\partial C = \nu_C^{-1}(1)$. Note that $\nu_C$ is continuous on $\R^4$ and smooth on $\R^4\setminus\{0\}$. The Hessian of $\nu_C^2$, denoted by $D^2\nu_C^2$, defines a $0$-homogeneous matrix-valued function on $\R^4\setminus\{0\}$. Consider 
\begin{equation}
K_{\rm min}^C = \inf_{z\in\R^4\setminus\{0\}} \ \min \{ \mu \mid \text{$\mu$ is an eigenvalue of $D^2\nu_C^2(z)$} \} \, .
\end{equation}
Convexity implies $K_{\rm min}^C \geq0$. We call $\partial C$ strictly convex if $K_{\rm min}^C >0$. The Hamiltonian vector field $X_H$ of $H = \nu_C^2$, defined on $\R^4 \setminus \{0\}$ by $i_{X_H}\omega_0 = -dH$, satisfies $i_{X_H}\lambda_0 = H$. Hence $X_H$ restricts to $\partial C$ as the Reeb vector field of the contact form induced by $\lambda_0$. In~\cite{convex} Hofer, Wysocki and Zehnder proved that this contact form is dynamically convex provided $\partial C$ is strictly convex. In particular, there are $\partial$-strong disk-like %GSSs 
global surfaces of section for the Hamiltonian dynamics on $\partial C$ under the assumption of strict convexity.

\begin{theorem}
\label{thm_main_1}
Assume that $\partial C$ is strictly convex, and let $D\subset\partial C$ be a disk-like $\partial$-strong global surface of section. Denote by $\tau_\min(D)>0$ the infimum of the first return time on $D$. {Then $\kappa(\partial D) > 2 K_{\rm min}^C \, \tau_{\rm min}(D)$. In particular,} if the inequality
%\marginpar{Referee suggests to change here. To discuss with Umberto.}
\begin{equation}
\label{dynamical_pinching}
K_{\rm min}^C \, \tau_{\rm min}(D) > \pi 
%\frac{\pi}{\tau_{\rm min}(D)}
\end{equation}
holds, 
% for some disk-like global surface of section $D$ on $\partial C$ 
then the Hamiltonian flow on $\partial C$ is right-handed.
\end{theorem}

\begin{remark}
If $C$ is the unit Euclidean ball in $\R^4$, then the gauge function is the Euclidean norm, we have $K^C_{\min} = 2$, and the Reeb flow on $\partial C$ is $\pi$-periodic and equal to the Hopf flow. In particular, $\tau_{\rm min}(D) = \pi$ for any disk-like global surface of section $D \subset \partial C$.
%Let us normalise the Hopf flow in such a way that every point is periodic with primitive period equal to $\pi$. 
If $C'$ is $C^2$-close to $C$ then it can be proved that the Reeb flow  on $\partial C'$ admits a $\partial$-strong disk-like global surface of section $D'$ with return time uniformly close to~$\pi$. Since the curvatures of $C'$ are close to those of $C$ we get $$ K^{C'}_{\min}  \, \tau_{\min}(D') \sim K^{C}_{\min} \, \tau_{\min}(D) = 2\pi \qquad \Rightarrow \qquad K^{C'}_{\min}  \, \tau_{\min}(D') > \pi \, . $$ By Theorem~\ref{thm_main_1} the Reeb flow on $\partial C'$ is right-handed, as claimed in~\cite{ghys}.
\end{remark}

{
\begin{remark}
In the case of Reeb flows, one gets restrictions of a contact topological nature 
%from the argument given in Appendix~\ref{app_abundance}. 
from Ghys' theorem \cite{ghys}. For instance, every periodic Reeb orbit defines a transverse knot that satisfies equality in Bennequin's inequality. 
Hence, we get restrictions on the transverse knot types defined by periodic Reeb orbits. 
%In the case of Reeb flows, the proof shows that 
One also deduces that all finite collections of periodic orbits bind open book decompositions that support the contact structure in the sense of Giroux, with pages that are global surfaces of section.
\end{remark}
}

\medskip

\noindent {\it Organisation of the paper.} In Section~\ref{sec_proof_thm_main} we give the definition of right-handedness and prove Theorem~\ref{thm_main_0}. We also show how Theorem~\ref{thm_main_1} follows from Theorem~\ref{thm_main_0}. Section~\ref{sec_proof_pinched} is devoted to the study of geodesic flows on $S^2$: we show how the claimed pinching condition enables us to verify the quantitative condition of Theorem~\ref{thm_main_0}. The proof uses comparison theorems from Riemannian geometry. %In Appendix~\ref{app_gen_position}, we study the properties of global surfaces of section in generic position and deduce that the quantity $\kappa(\gamma_0)$ is independent of the choice of the $\partial$-strong disk-like global surface of section spanned by $\gamma_0$, and of the choice of a global symplectic trivialization. %For the sake of completeness, a proof of Theorem~\ref{thm_ghys} is given in Appendix~\ref{app_abundance}. 
%Finally, in Appendix~\ref{app_coordinates}, we discuss technical properties of $\partial$-strong disk-like global surfaces of section that are used in the proof of Theorem~\ref{thm_main_0}.

\medskip

\noindent {\it Acknowledgements.} We thank Alberto Abbondandolo, Marie-Claude Arnaud, Pierre Dehornoy, Ana Rechtman and Pedro Salom\~ao for helpful discussions. We would like to thank the referee for carefully reading our paper. This project initiated when UH visited Avignon Universit\'e in 2018, invited by Marie-Claude Arnaud and Andrea Venturelli, while AF was a doctoral student there. We are especially grateful to them for this opportunity.

\section{Right-handedness from dynamical pinching}\label{sec_proof_thm_main}

\subsection{Transverse rotation numbers}\label{ssec_transv_rot_numbers}

Let $\gamma$ be a non-constant periodic orbit of a smooth flow $\phi^t$ defined on an oriented $3$-manifold~$M$. Denote by $T>0$ its primitive period, and think of $\gamma$ as a map $\gamma:\R/T\Z\to M$. Consider coordinates $(t,z=x+iy=re^{i\theta}) \in \R/T\Z \times \C$ defined on a small tubular neighbourhood $N$ of $\gamma$ such that $dt\wedge dx \wedge dy > 0$ and $\phi^t(\gamma(0)) = (t,0)$. We shall refer to such coordinates as tubular coordinates around $\gamma$. For every $\theta_0 \in \R$ consider the continuous real valued function $\theta(t)$ defined by
\begin{equation*}
D\phi^t(0,0) \cdot (0,e^{i\theta_0}) \in \R(1,0) + \R_+ (0,e^{i\theta(t)}) \qquad \theta(0) = \theta_0
\end{equation*}
If $y \in H^1(N\setminus\gamma,\R)$ is homologous to $p \ dt + q \ d\theta$ then we define
\begin{equation}
\rho^y(\gamma) = \frac{T}{2\pi} \left( p + q \lim_{t\to+\infty} \frac{\theta(t)}{t} \right).
\end{equation}
This number is called the transverse rotation number of $\gamma$ with respect to $y$. 
{The number $\rho^y(\gamma)$ might also be interpreted as a Poincar\'e translation number of the linearized flow on the unit normal bundle.} It turns out that $\rho^y(\gamma)$ does not depend on the choice of tubular coordinates or on the initial condition $\theta_0$; see~\cite[section~2]{SFS}.

\begin{remark}\label{rmk_CZ_rot}
In the notation above, suppose that $\phi^t$ is the Reeb flow of a contact form $\lambda$, and let $\sigma$ be a {local} symplectic trivialization of $(\xi,d\lambda)$ along $\gamma$. Denote by $\gamma_\sigma$ an oriented loop in $N\setminus\gamma$ obtained by pushing $\gamma$ in the direction of $\sigma$. If $N$ is small enough then $d\lambda$ defines an area form on any meridional disk $D$. Orient $D$ by $d\lambda$. There is a unique class $y_\sigma\in H^1(N\setminus\gamma,\Z)$ determined by $\left< y_\sigma,\partial D \right>=1$, $\left< y_\sigma,\gamma_\sigma \right>=0$. If  no transverse Floquet multiplier of $\gamma$ is a root of unit of order $n\geq1$ then one says that the $n$-th iterate $\gamma^n$ of $\gamma$ is non-degenerate, and defines {
\begin{equation}
\label{eq_def_CZ_frame}
\mu_{\rm CZ}^\sigma(\gamma^n) =\lfloor n2\pi \rho^{y_\sigma}(\gamma) \rfloor + \lceil n2\pi \rho^{y_\sigma}(\gamma) \rceil.%2 \lceil n2\pi \rho^{y_\sigma}(\gamma) \rceil\, . %\begin{cases} 2 \lfloor n2\pi\rho^{y_\sigma}(\gamma) \rfloor + 1 & \text{if $n2\pi \rho^{y_\sigma}(\gamma) \not \in\Z$} \\ 2 n2\pi \rho^{y_\sigma}(\gamma) & \text{if $n2\pi\rho^{y_\sigma}(\gamma)  \in\Z$} \end{cases}
\end{equation}
Note that when $\gamma^n$ is non-degenerate then $n2\pi\rho^{y_\sigma}(\gamma) \in \Z$ precisely when $\gamma^n$ is positive hyperbolic.
When $\gamma^n$ is degenerate} then $\mu_{\rm CZ}^\sigma(\gamma^n)$ is defined to be the lowest possible value of the right-hand side above, obtained from small $C^2$-perturbations of $\lambda$ that keep $\gamma$ as a periodic Reeb orbit with~$\gamma^n$ non-degenerate. The inequality $\mu_{\rm CZ}^\sigma(\gamma) \geq 3$ is equivalent to $2\pi\rho^{y_\sigma}(\gamma) > 1$. {This follows from~\eqref{eq_def_CZ_frame} in the non-degenerate case, and is also true in general.}
\end{remark}

Suppose further that $M$ is a homology $3$-sphere. Consider any oriented Seifert surface $S$ spanned by $\gamma$. We require that $\partial S = \gamma$ including orientations, when $\gamma$ is oriented by the flow. As before, orient the meridional disk $D\subset N$ by $d\lambda$. Let $S^* \in H^1(M\setminus\gamma,\Z)$ denote the class dual to $S$. Since $M$ is a homology $3$-sphere, the class $S^*$ is independent of $S$. In fact, $\left<S^*,\beta\right> = \link(\gamma,\beta)$ for any oriented loop $\beta$ in $M\setminus\gamma$. After restricting to $N\setminus \gamma$ we can view it as a class in $H^1(N\setminus\gamma,\Z)$.

\begin{definition}\label{def_rot_Seifert}
Assume that $M$ is a homology $3$-sphere. 
Let $y$ be the cohomology class dual to some (hence any) oriented Seifert surface $S$ spanned by $\gamma$.
We call $\rho^{y}(\gamma)$ the transverse rotation number of $\gamma$ in a Seifert framing.
\end{definition}

From now on assume that $M$ is a homology three-sphere and that $\phi^t$ has no rest points. Let $\gamma_1,\dots,\gamma_n$ be a collection of periodic orbits oriented along the flow, and consider the oriented link $L = \gamma_1 \cup \cdots \cup \gamma_n$. Let $\Sigma_i$ be an oriented Seifert surface satisfying $\partial\Sigma_i = \gamma_i$, orientations included. Let $\Sigma$ be an oriented Seifert surface satisfying $\partial\Sigma = L$, orientations included. Consider small tubular neighbourhoods $N_i$ of $\gamma_i$ with tubular coordinates as above. Denote by $\Sigma^* \in H^1(M\setminus L,\R)$ and $\Sigma^*_i \in H^1(M\setminus\gamma_i,\R)$ the classes dual to $\Sigma$ and $\Sigma_i$ respectively. Each $\Sigma^*_i$ restricts to a class in $H^1(N_i\setminus\gamma_i,\R)$ still denoted by $\Sigma^*_i$ with no fear of ambiguity. Similarly $\Sigma^*$ restricts to a class in $H^1(N_i\setminus\gamma_i,\R)$ for every $i$, all of which are denoted by~$\Sigma^*$ with no fear of ambiguity.

\begin{lemma}\label{lemma_rotation_numbers}
For every $i$ we have $\rho^{\Sigma^*}(\gamma_i) = \rho^{\Sigma^*_i}(\gamma_i) + \frac{1}{2\pi} \sum_{j\neq i} \link(\gamma_i,\gamma_j)$.
\end{lemma}

\begin{proof}
Let $(t,z=|z|e^{i\theta}) \in \R/T_i\Z \times \C$ be tubular coordinates on a small tubular neighbourhood $N_i$ of $\gamma_i$. With $\varepsilon>0$ small we consider the loops $e_i(t)=(t,\varepsilon)$, $f_i(\theta)=(0,\varepsilon e^{i\theta})$. Then $\{e_i,f_i\}$ is a basis for $H_1(N_i\setminus\gamma_i,\Z)$ whose dual basis is $\{dt/T_i,d\theta/2\pi\}$. Then on $N_i\setminus\gamma_i$ one has $$ \begin{aligned} \Sigma_i^* &\equiv \left< \Sigma_i^*,e_i \right> \frac{dt}{T_i} + \left< \Sigma_i^*,f_i \right> \frac{d\theta}{2\pi} = \left< \Sigma_i^*,e_i \right> \frac{dt}{T_i} + \frac{d\theta}{2\pi} \\ \Sigma^* &\equiv \left< \Sigma^*,e_i \right> \frac{dt}{T_i} + \left< \Sigma^*,f_i \right> \frac{d\theta}{2\pi} = \left< \Sigma^*,e_i \right> \frac{dt}{T_i} + \frac{d\theta}{2\pi} \end{aligned} $$ Denote by $\nu_i$ a section of the normal bundle of $\Sigma_i$ along $\gamma_i$. Let $\gamma_i'$ be the loop obtained by pushing $\gamma_i$ in the direction of $\nu_i$. Note that $\gamma_i'$ gets an orientation from $\gamma_i$. Then $$ \gamma_i' \equiv \left< \frac{dt}{T_i},\gamma_i' \right> e_i + \left< \frac{d\theta}{2\pi},\gamma_i' \right> f_i = e_i + \left< \frac{d\theta}{2\pi},\gamma_i' \right> f_i $$ From these formulas it follows that
\begin{equation*}
{\rm int}(\gamma_i',\Sigma) - {\rm int}(\gamma_i',\Sigma_i) = \left< \Sigma^*,\gamma'_i \right> - \left< \Sigma^*_i,\gamma'_i \right> = \left< \Sigma^*,e_i \right> - \left< \Sigma^*_i,e_i \right>
\end{equation*}
Consider the $2$-cycle $S=\Sigma - \Sigma_1 - \dots - \Sigma_n$. Since the ambient space is a homology sphere (over $\Z$), $S$ is a boundary and we get
\begin{equation}
\begin{aligned}
0 &= {\rm int}(\gamma'_i,S) = {\rm int}(\gamma_i',\Sigma) - {\rm int}(\gamma_i',\Sigma_i) - \sum_{j\neq i} \link(\gamma_i',\gamma_j) \\
&= \left< \Sigma^*,e_i \right> - \left< \Sigma^*_i,e_i \right> - \sum_{j\neq i} \link(\gamma_i,\gamma_j) \\
\end{aligned}
\end{equation}
Hence, if $\theta(t)$ denotes a lift of the polar angle of the linearised flow along $\gamma_i$ in the given tubular coordinates we get
\begin{equation}
\begin{aligned}
\rho^{\Sigma^*}(\gamma_i) &= \frac{T_i}{2\pi} \left( \frac{\left<\Sigma^*,e_i\right>}{T_i} + \frac{1}{2\pi} \lim_{t\to+\infty} \frac{\theta(t)}{t} \right) \\
&= \frac{T_i}{2\pi} \left( \frac{\left<\Sigma^*_i,e_i\right>}{T_i} + \frac{1}{T_i} \sum_{j\neq i} \link(\gamma_i,\gamma_j) + \frac{1}{2\pi} \lim_{t\to+\infty} \frac{\theta(t)}{t} \right) \\
&= \rho^{\Sigma_i}(\gamma_i) + \frac{1}{2\pi} \sum_{j\neq i} \link(\gamma_i,\gamma_j)
\end{aligned}
\end{equation}
as desired. 
\end{proof}

\begin{corollary}\label{cor_positive_transv_number}
If $n\geq2$ and $\link(\gamma_i,\gamma_j)\geq1$ for all $i\neq j$ then $\rho^{\Sigma^*}(\gamma_i) > \rho^{\Sigma_i^*}(\gamma_i)$ for every~$i$.
\end{corollary}

\begin{remark}
Let the periodic orbit $\gamma$ be oriented by the flow, and let $\Sigma,\hat\Sigma$ be oriented Seifert surfaces such that $\partial\Sigma=\gamma$, $\partial\hat\Sigma=\gamma$ (including orientations). Lemma~\ref{lemma_rotation_numbers} proves the previously claimed fact that $\rho^{\Sigma}(\gamma)=\rho^{\hat\Sigma}(\gamma)$.
\end{remark}

\subsection{Right-handedness}\label{ssec_right_handedness}

Fix a smooth %non-vanishing 
{nowhere vanishing} vector field $X$ on an oriented homology $3$-sphere, and denote its flow by $\phi^t$. Let $\mathscr{P}$ be the set of $\phi^t$-invariant Borel probability measures. Denote by $\mathscr{R}$ the set of recurrent points, and consider the following measurable set: 
%with respect to the product Borel $\sigma$-algebra:
\begin{equation}
R = \{ (x,y) \in \mathscr{R} \times \mathscr{R} \mid \phi^\R(x) \cap \phi^\R(y) = \emptyset \} \, .
\end{equation}
Let $\mu_1,\mu_2 \in \mathscr{P}$ be ergodic, and denote by $\mu_1\times\mu_2$ the product measure. There are two cases:
\begin{itemize}
\item[{\bf (A)}] $\mu_1\times\mu_2(R)=1$.
\item[{\bf (B)}] $\mu_1\times\mu_2(R)=0$ and %$\supp(\mu_1) \cup \supp(\mu_2) \subset \gamma$ for some periodic orbit $\gamma$.
{$\supp(\mu_1)=\supp(\mu_2)=\gamma$ for some periodic orbit $\gamma$ (in particular $\mu_1=\mu_2$).}
\end{itemize}
Each case needs to be treated separately. Fix an auxiliary Riemannian metric $g$ { that near $p$ and $q$ realises pieces of trajectories of $X$ as geodesic arcs}.

\medskip

\noindent {\bf Case A.} Consider $(p,q)\in R$ and let $\mathcal{S}(p,q)$ denote the set of ordered pairs of sequences $(\{T_n\},\{S_n\})$ satisfying $T_n,S_n\to+\infty$, $\phi^{T_n}(p) \to p$ and $\phi^{S_n}(q) \to q$. For $n$ large enough let $\alpha_n$ and $\beta_n$ be the (unique) shortest geodesic arcs from $\phi^{T_n}(p)$ to $p$ and from $\phi^{S_n}(q)$ to $q$, respectively. Consider $C^1$-small %\footnote{The assumption that $\hat\alpha,\hat\beta$ are $C^1$-close to $\alpha_n,\beta_n$ is important. Being $C^0$-close is not enough.}
 perturbations $\hat\alpha,\hat\beta$ of $\alpha_n,\beta_n$, keeping end points fixed, such that the closed loops $k(T_n,p)$ and $k(S_n,q)$ obtained by concatenating $\hat\alpha$ to $\phi^{[0,T_n]}(p)$ and $\hat\beta$ to $\phi^{[0,S_n]}(q)$, respectively, do not intersect each other. The assumption that $\hat\alpha,\hat\beta$ are $C^1$-close to $\alpha_n,\beta_n$ is important: being $C^0$-close is not enough. Define
\begin{equation}\label{def_link_minus}
\begin{aligned}
\link_-(\phi^{[0,T_n]}(p),\phi^{[0,S_n]}(q)) &= \liminf_{\tiny \hat\alpha\stackrel{C^1}{\to}\alpha_n \ \hat\beta\stackrel{C^1}{\to}\beta_n} \ \link(k(T_n,p),k(S_n,q))
\end{aligned}
\end{equation}
and
\begin{equation}\label{funny_ell}
\ell(p,q) \ = \ \inf_{(\{T_n\},\{S_n\}) \in \mathcal{S}(p,q)} \ \liminf_{n\to\infty} \ \frac{1}{T_nS_n} \link_-(\phi^{[0,T_n]}(p),\phi^{[0,S_n]}(q)).
\end{equation}
Note that~\eqref{def_link_minus} and~\eqref{funny_ell} belong to $[-\infty,+\infty]$.
One says that $\mu_1,\mu_2$ are positively linked if for $\mu_1\times\mu_2$-almost all points $(p,q)$ in $R$ the inequality $\ell(p,q) > 0$ holds.

\medskip

\noindent {\bf Case B.} One says that $\mu_1,\mu_2$ are positively linked if the transverse rotation number of the periodic orbit $\gamma$ %containing
{corresponding to the supports} of $\mu_1,\mu_2$ computed in a Seifert framing is strictly positive. 

\begin{definition}\label{def_right_handedness}
The flow $\phi^t$ is said to be right-handed if all pairs of ergodic measures in $\mathscr{P}$
%$\mu_1,\mu_2 \in \mathscr{P}$ 
are positively linked. 
\end{definition}

\begin{remark}
The above definition is equivalent to the one explained in~\cite{ghys}. Of course, Ghys defines right-handedness in a much more elegant way by explaining that in case {\bf A} the ergodicity assumption can be used to prove that for $\mu_1\times\mu_2$-almost all $(p,q) \in R$, all possible sequences $$ \frac{\link(k(T_n,p),k(S_n,q))}{T_nS_n} $$ as above will converge to a common limit. This limit is defined to be the value of the quadratic linking form evaluated at the pair $(\mu_1,\mu_2)$. There is also a way of assigning a number in case {\bf B}. The advantage of the definition explained here is that it avoids dealing with details on the existence of the quadratic linking form. 
\end{remark}

\subsection{Proof of Theorem~\ref{thm_main_0}}\label{ssec_proof_abstract_thm}

We fix a dynamically convex contact form $\lambda$ on $S^3$, denote by $X$ the associated Reeb vector field and by $\phi^t$ the Reeb flow. %The following technical statement, which is based on the main result from~\cite{convex}, is proved in appendix~\ref{app_coordinates}. 
{The following technical statement is based on the main result from~\cite{convex}.}

\begin{proposition}
\label{prop_main}
	Let $\lambda$ be a dynamically convex contact form on $S^3$, and denote by $X$ the Reeb vector field of $\lambda$. Let $\gamma_0$ be any unknotted periodic Reeb orbit with self-linking number $-1$. Denote the primitive period by $T_0>0$. There exists a map of class $C^\infty$
	\begin{equation}\label{map_Psi_prop_main}
	\Psi:\R/\Z\times\D \to S^3
	\end{equation}
	with the following properties:
	\begin{itemize}
		\item[(a)] $\Psi(0,e^{is}) = \gamma_0(T_0s/2\pi)$ for all $s \in \R/2\pi\Z$, and $\Psi(0,\cdot):\D\to S^3$ is an embedding that defines a $\partial$-strong global surface of section for the flow of $X$.
		\item[(b)] $\Psi$ defines an orientation preserving diffeomorphism $\R/\Z \times \mathring{\D} \to S^3 \setminus \gamma_0$.
		\item[(c)] There exists a smooth vector field $W$ on $\R/\Z \times \D$ that coincides with the pull-back of $X$ by $\Psi|_{\R/\Z \times \mathring{\D}}$ on $\R/\Z \times \mathring{\D}$, and is tangent to $\R/\Z\times\partial\D$.
		\item[(d)] For every $t\in\R/\Z$ the disk $\{t\}\times\D$ is transverse to $W$ up to the boundary, and defines a global section for the flow of $W$.
		\item[(e)] There is a non-vanishing vector field $Z$ on $S^3$ satisfying $i_Z\lambda=0$, in particular $Z$ is Legendrian, with the following property. Let $Z_0$ be the unique smooth vector field on $\mathring{\D}$ defined by $$ Z(\Psi(0,z)) - d\Psi(0,z) \cdot Z_0(z) \in \R X(\Psi(0,z)). $$ If $\phi:\mathring{\D} \to \R$ is continuous and satisfies $Z_0 = |Z_0|e^{i\phi}$, then $\phi \in L^\infty(\mathring{\D})$.
	\end{itemize}
\end{proposition}

{
\begin{proof}
Following \cite{convex,char2,openbook}, there exists an embedded closed disk $\mathcal{D}\hookrightarrow S^3$ of class $C^1$ such that $\partial\mathcal{D}=\gamma_0$. It is obtained by projecting to $M$ a special  finite energy plane in the symplectisation of $(M,\xi)$ asymptotic to~$\gamma_0$. 
By a careful analysis of the aforementioned embedding, it is also possible to show that there exists a $C^\infty$ family of $C^\infty$ disks $\mathcal{D}_\delta$, parametrised by $0<\delta\ll 1$, such that:
\begin{enumerate}
\item Each $\mathcal{D}_\delta$ is a $\partial$-strong global surface of section satisfying $\partial\mathcal{D}_\delta=\gamma_0$.
\item $\mathcal{D}_\delta\to \mathcal{D}$ in $C^1$ as $\delta\to 0$.
\end{enumerate}
For detailed proofs of the claims above we refer to~\cite[appendix~C]{FHArxiv_v1}.

We want now to project vector fields on $\mathcal{D}_\delta\setminus\gamma_0$.	
Consider a \textit{Martinet tube} $F:\R/\Z\times\text{int}(\mathbb{D})\subset \R/\Z\times\C\to U$ around $\gamma_0$. This means that $U$ is an open neighborhood of $\gamma_0$, and $F$ is a diffeomorphism satisfying
\begin{itemize}
\item[(MT1)] $F(t,0) = \gamma_0(T_0t)$.
\item[(MT2)] There exists a smooth function $f:\R/\Z \times {\rm int}(\D)\to(0,+\infty)$ such that $f(\vartheta,0) \equiv T_0$, $df(\vartheta,0)\equiv0$, and $F^*\lambda = f(\vartheta,x+iy) (d\vartheta + xdy)$.
\end{itemize}
The vector fields $\partial_x,\partial_y$ along $\gamma_0$ define a $d\lambda$-positive frame of $\xi$ along $\gamma_0$.
Fix $\delta$ small. Near $\gamma_0$ we can find a non-vanishing section $Y_0$ of $\xi$ that satisfies $\R Y_0 = T\mathcal{D}_\delta \cap \xi$ on points of $\mathcal{D}_\delta$ near $\gamma_0$. Using the coordinates $(\vartheta,x+iy)$ in the Martinet tube near~$\gamma_0$, we can write
	\begin{equation*}
	Y_0 = a_0 \partial_x + b_0 (\partial_y-x\partial_\vartheta)
	\end{equation*}
	where the $\C$-valued function $a_0+ib_0$ does not vanish. Consider the vector field
	\begin{equation}
	\label{vf_Y_1}
	Y_1 = a_1 \partial_x + b_1 (\partial_y-x\partial_\vartheta) \qquad \begin{pmatrix} a_1 \\ b_1 \end{pmatrix} = \begin{pmatrix} \cos(2\pi\vartheta) & \sin(2\pi\vartheta) \\ -\sin(2\pi\vartheta) & \cos(2\pi\vartheta) \end{pmatrix} \begin{pmatrix} a_0 \\ b_0 \end{pmatrix} \, .
	\end{equation}
	Since $\gamma_0$ has self-linking number $-1$, and $Y_1$ winds $-1$ with respect to $Y_0$ along $\gamma_0$, there exists a smooth extension of $Y_1$ as a non-vanishing section of $\xi \to S^3$. 
We continue to denote this extension by~$Y_1$, with no fear of ambiguity.

	Recall that $\mathcal{D}_\delta$ is a global surface of section spanned by $\gamma_0$ for small enough $\delta$, and $\mathcal{D}_\delta \to \mathcal{D}$ in $C^1$ as $\delta \to 0$. From now on we will denote by $\dot{\mathcal{D}}_\delta = \mathcal{D}_\delta \setminus \partial \mathcal{D}_\delta = \mathcal{D}_\delta \setminus \gamma_0$ the interior of $\mathcal{D}_\delta$. The projection along the Reeb direction is a smooth vector bundle map $$ P : TS^3|_{\dot{\mathcal{D}}_\delta} \to T\dot{\mathcal{D}}_\delta $$ characterised by $$ P^2=P, \qquad\qquad \ker P = \R X. $$ Note that $P$ becomes singular on the boundary since $X$ is tangent to $\partial \mathcal{D}_\delta = \gamma_0$. Note also that $P$ defines a vector bundle isomorphism between $\xi|_{\dot{\mathcal{D}}_\delta}$ and $T\dot{\mathcal{D}}_\delta$, this follows from the transversality between the Reeb vector field and $\dot{\mathcal{D}}_\delta$.
	
	Let $A_\delta \subset \mathcal{D}_\delta$ be a small compact neighborhood of $\gamma_0 = \partial \mathcal{D}_\delta$ in $\mathcal{D}_\delta$, small enough so that it is contained in the domain of definition of $Y_0$. Denote $\dot A_\delta = A_\delta \setminus \gamma_0$. We can equip $A_\delta$ with polar coordinates $(\rho,\vartheta)$, coeherent with the Martinet tube in the sense that $\vartheta$ is the same $\R/\Z$-coordinate in the domain of $F$, so that $A_\delta$ corresponds to $\{(\rho,\vartheta) \in [1-\varepsilon,1]\times\R/\Z\}$ for $\varepsilon>0$ small enough, and, moreover, $\dot A_\delta$ corresponds to $\{(\rho,\vartheta) \in [1-\varepsilon,1)\times\R/\Z\}$. The local slice disk $\{\vartheta \equiv 0 \mod \Z\}$ intersects $A_\delta$ in a smooth arc $\eta$ transverse to $\partial A_\delta$ defining a generator of $H_1(A_\delta,\partial A_\delta)$. %In the coordinates $(\rho,\vartheta)$, $\eta$ corresponds to $\{\rho \in [1-\varepsilon,1], \ \theta \equiv 0 \mod \Z\}$.
	
	By~\eqref{vf_Y_1}, $Y_0$ and $Y_1$ are positively collinear only at $\{\vartheta \equiv 0 \mod \Z\}$. Hence on $\dot{A}_\delta$ the vector fields $P(Y_1)$ and $Y_0 = P(Y_0)$ are positively collinear only at $\eta \setminus \gamma_0$. The universal covering of $A_\delta$ can be given coordinates $(\rho,\tilde\vartheta) \in [1-\varepsilon,1] \times \R$, where $\vartheta = \tilde\vartheta \mod \Z$. Hence the universal covering of $\dot A_\delta$ is $[1-\varepsilon,1) \times \R$. The disk $\mathcal{D}_\delta$ can be equipped with global coordinates $u+iv \in \D$ such that $u+iv = \rho e^{i2\pi\vartheta}$ near the boundary. We can write
	\begin{equation*}
	\begin{cases}
	Y_0 = R_0 (\cos\varphi_0 \ \partial_u + \sin\varphi_0 \ \partial_v) & \text{on $A_\delta$} \\
	P(Y_1) = R_1 (\cos\varphi_1 \ \partial_u + \sin\varphi_1 \ \partial_v) & \text{on $\dot{\mathcal{D}}_\delta$}
	\end{cases}
	\end{equation*}
	in polar coordinates, where
	\begin{equation*}
	\varphi_0 : [1-\varepsilon,1] \times \R/\Z \to \R/2\pi\Z \qquad \varphi_1 : \dot{\mathcal{D}}_\delta \to \R/2\pi\Z
	\end{equation*}
	are smooth. Choose smooth lifts 
	\begin{equation*}
	\tilde \varphi_0 : [1-\varepsilon,1] \times \R \to \R \qquad \text{and} \qquad \tilde \varphi_1 : \dot{\mathcal{D}}_\delta \to \R
	\end{equation*}
	of $\varphi_0$ and $\varphi_1$, respectively. We will also write $\tilde \varphi_1(\rho,\tilde\vartheta)$ to denote the corresponding lift of the restriction of $\varphi_1$ to~$\dot A_\delta$, with no fear of ambiguity. Note that
	\begin{equation}
	\label{periodicity_varphis}
	\begin{cases}
	\tilde\varphi_0(\rho,\tilde\vartheta+1) = \tilde\varphi_0(\rho,\tilde\vartheta) + 2\pi \, ,\\
	\tilde\varphi_1(\rho,\tilde\vartheta+1) = \tilde\varphi_1(\rho,\tilde\vartheta) \, .
	\end{cases}
	\end{equation}
	Note also that $\tilde\varphi_0$ is bounded on any compact subset of $[1-\varepsilon,1] \times \R$ since $Y_0$ is smooth on $A_\delta$. As observed before,
	\begin{equation}
	\label{collinear_varphis}
	\tilde\varphi_1(\rho,\tilde\vartheta) - \tilde\varphi_0(\rho,\tilde\vartheta) \in 2\pi\Z \qquad \Leftrightarrow \qquad \tilde\vartheta \in \Z.
	\end{equation}
	The lifts of $\eta$ divide the universal covering of $\dot A_\delta$ in fundamental domains. It follows from~\eqref{collinear_varphis} that on each such fundamental domain $[1-\varepsilon,1) \times [k,k+1]$ the function $\tilde\varphi_1 - \tilde\varphi_0$ takes values on $[2\pi(m-1),2\pi m]$ for some $m \in \Z$, in particular it is bounded there. Hence the function $\tilde\varphi_1  = \tilde\varphi_0 + \tilde\varphi_1 - \tilde\varphi_0$ is bounded on each fundamental domain $[1-\varepsilon,1) \times [k,k+1]$. In view of the second equation in~\eqref{periodicity_varphis} we can conclude that $\tilde\varphi_1$ is bounded on $\dot A_\delta$. By compactness of the closure of $\dot{\mathcal{D}}_\delta \setminus \dot A_\delta$, we get that $\tilde\varphi_1 \in L^\infty(\dot{\mathcal{D}}_\delta)$. 
	
	This shows that the desired vector field $Z$ satisfying property (e) in Proposition~\ref{prop_main} can be taken as $Z=Y_1$.

	We can now conclude our proof. Using the Martinet tube $F$ chosen above,
%	:\R/\Z\times{\rm int}(\D) \to U$, 
	we can consider the space
	\begin{equation}
	M = \left. \left( S^3\setminus \gamma_0 \ \sqcup \ \R/\Z \times [0,1) \times \R/2\pi\Z \right) \right/ \sim
	\end{equation}
	where a point $F(\vartheta,re^{i\theta}) \in U \setminus \gamma_0$ is identified with $(\vartheta,r,\theta) \in \R/\Z \times (0,1) \times \R/2\pi\Z$. One gets a differentiable structure on $M$ which makes $M$ diffeomorphic to $\R/\Z\times \D$ in such a way that $re^{i2\pi\vartheta}$ are polar coordinates near the boundary of the $\D$-factor. %We have claimed that, from~\cite{convex,openbook}, % Starting from a special finite-energy plane $\util$ obtained from~\cite{convex}, or from~\cite{openbook}, we showed in {\it Step~2} that 
%	there exists a 
The smooth embedded disk $\mathcal{D}_\delta \subset S^3$ 
%satisfying $\partial \mathcal{D}_\delta = \gamma_0$, and 
is such that $\dot{\mathcal{D}}_\delta = \mathcal{D}_\delta \setminus \partial \mathcal{D}_\delta$ is the interior of a smoothly embedded meridional disk $D\subset M$ intersecting $\partial M$ cleanly. %This claim follows from~\eqref{strip_in_polar_coords_delta} since $\Gamma_\delta(\rho,\vartheta) = (1-\rho)e(\vartheta)$ when $\rho$ is close to $1$; here $e(\vartheta)$ is the eigenvector of the asymptotic operator~\eqref{asump_op_formula} governing the asymptotic behaviour of $\util$ near $\infty$. 
	Moreover, it is possible to prove that $X|_{S^3\setminus\gamma_0}$ extends to a smooth vector field $W$ on $M$ and such that, %represented as~\eqref{vf_W_boundary_torus} in the coordinates $(\vartheta,\theta)$ on $\partial M$. I
	in particular, $W$ is tangent to $\partial M$. %Actually
	Since $\mathcal{D}_\delta$ is $\partial$-strong, the embedded disk %Finally, in {\it Step~2} we showed with~\eqref{conclusion1}-\eqref{conclusion2} that 
	$D$ is transverse to $W$ up to the boundary and %. From~\eqref{vf_W_boundary_torus} we see that 
	$D\cap \partial M = \partial D$ is a global surface of section for the flow of $W$ on $\partial M$. Since $D \setminus \partial M = \dot{\mathcal{D}}_\delta$ is also a global surface of section for the flow of $W$ on $M\setminus \partial M$, one gets that $D$ is a global section for the flow of $W$ on $M$. Using the flow of $W$ to deform $D$ one constructs a smooth foliation $\{D_t\}_{t\in\R/\Z}$ of $M$ by embedded disks in such that $D=D_0$ and all $D_t$ are transverse to $W$. It follows that all $D_t$ are global surfaces of section for the flow of $W$ on $M$ since $D=D_0$ is. Hence we get a diffeomorphism $$ M \simeq \R/\Z \times \D \qquad\qquad D_t \simeq \{t\}\times \D $$ Properties (a)-(d) follow from this construction.

	We fix coordinates $u+iv$ on $D_0 \simeq \{0\} \times \D$. %In {\it Step 3} we proved that there is 
	The aforementioned discussion has shown the existence of a non-vanishing vector field $Z$ on $S^3$ tangent to the contact structure $\xi = \ker\lambda$ that projects to $\dot{\mathcal{D}}_\delta \simeq D_0 \setminus \partial D_0 \simeq \{0\} \times \mathring{\D}$ along the direction of $X \simeq W$ to a non-vanishing vector field on $\mathring{\D}$ for which any continuous choice of argument (in the coordinates $u+iv$) defines a function in $L^\infty(\mathring{\D})$. This is property~(e).
\end{proof}

}

%\begin{remark}
%In appendix~\ref{app_coordinates} the reader will see that the set $\Psi(\{0\}\times\mathring{\D})$ is a small perturbation of the projection of the finite-energy plane obtained by Hofer, Wysocki and Zehnder in~\cite{convex}.
%\end{remark}

Denote by $\varphi^t$ the flow of $W$, so that on $\R/\Z\times\mathring{\D}$ we have $$ \varphi^t = \Psi^{-1} \circ \phi^t \circ \Psi. $$ The first return time back to $\{0\}\times\D$ is
\begin{equation}
\tau:\D\to(0,+\infty) \qquad \tau(z)=\inf\{ t>0 \mid \varphi^t(0,z)\in \{0\} \times \D \}.
\end{equation}
It follows from (c) and (d) that $\tau$ is bounded away from zero and bounded from above. Denote 
\begin{equation*}
0 < \tau_{\min} = \min \tau \qquad \qquad \tau_{\max} = \max \tau < +\infty
\end{equation*}
The return map is denoted by
\begin{equation}\label{return_map_notation}
h:	\D \to \D \qquad\qquad (0,h(z)) = \varphi^{\tau(z)}(0,z)
\end{equation}

\begin{lemma}\label{lemma_flipping_pages_time}
	For every $c\in(0,1)$ there exists $t:[0,1]\times \D \to [0,+\infty)$ smooth such that: 
	\begin{itemize}
		\item[(i)] $t(0,z)=0$ and $t(1,z)=\tau(z)$ for all $z\in\D$.
		\item[(ii)] $D_1t(s,z) \geq c\tau_{\min}$ for all $(s,z) \in [0,1] \times \D$.
		\item[(iii)] $\exists 0<\varepsilon\ll1$ such that $D_1t(s,z) = \tau_\max$ if $s\in[0,\varepsilon)\cup(1-\varepsilon,1]$, for all $z\in\D$.
	\end{itemize} 
\end{lemma}

\begin{proof}
Fix $c\in(0,1)$ and let $\delta:=c\tau_\min$. Consider the following set
\[
\Omega := \{ (t,z) \mid z\in {\D}, \ 0\leq t\leq \tau(z) \} \subset \R \times \D,
\]
and define the 
piecewise smooth function $s:\Omega\rightarrow[0,1]$ by 
$$
s(t,z):=
\begin{cases}
\dfrac{ t}{\tau_{\max}} & 0\leq  t\leq \tau(z)-\delta \\ \\
\dfrac{\tau(z)-\delta}{\tau_\max} + \dfrac{\tau_{\max}-\tau(z)+\delta}{\delta\tau_\max} \left(t-\tau(z)+\delta\right) & \text{for }\tau(z)-\delta\leq  t\leq\tau(z).
\end{cases}
$$
Let us now define $\hat t : [0,1]\times{\mathbb{D}}\to [0,+\infty)$ by the identity $ \hat t(s(t,z),z) = t$. Observe that $\hat t(0,z)=0$ and $\hat t(1,z)=\tau(z)$. For every $(s,z) \in [0,1] \times \mathring{\mathbb{D}}$ it holds, {denoting by $D^{\pm}_1$ one-sided derivatives in the first variable,$$ \min\left\{ D^+_1\hat t(s,z),D^-_1\hat t(s,z) \right\} \geq \min\left\{ \tau_{\max}, \dfrac{\delta\tau_{\max}}{\tau_{\max}-\tau(z)+\delta}  \right\} >\delta=c\tau_\min\,. $$}
%$$ \min\left\{ D^+_1\hat t(s,z),D^-_1\hat t(s,z) \right\} \geq \min\left\{ \tau_{\max}, \dfrac{\delta\tau_{\max}}{\tau_{\max}-\tau(z)+\delta}  \right\} = \dfrac{\delta\tau_{\max}}{\tau_{\max}-\tau(z)+\delta} $$} where $D^\pm_1$ denote one-sided derivatives in the first variable. In particular
%$$
%\dfrac{\delta\tau_{\max}}{\tau_{\max}-\tau(z)+\delta} > \delta = c\tau_{\min}.
%$$
Using a parametrised splining method we get the a smooth $t:[0,1] \times {\mathbb{D}} \to [0,+\infty)$ satisfying {$D_1t(s,z) \in [D^+_1\hat t(s,z),D^-_1\hat t(s,z)]\geq c\tau_\min$}, $t(0,z)=0$ and $t(1,z)=\tau(z)$. % In particular $D_1t(s,z) \geq c\tau_\min$ for all $(s,z) \in [0,1] \times {\mathbb{D}}$. 
Extend $t$ to $[0,+\infty)\times {\D}$ by
\begin{equation}\label{extension of that}
t(s,z) = \sum_{j=0}^{\lfloor s\rfloor -1} \tau \circ h^j(z) \ + \ t \left( s-\lfloor s\rfloor,h^{\lfloor s\rfloor}(z) \right) \qquad\qquad (s>1)
\end{equation}
It follows that $t$ is continuous on $[0,+\infty) \times {\D}$ and smooth on $[n,n+1]\times {\D}$ for each $n\geq 0$, that %$D_1t$ has only discontinuities possibly 
the discontinuities of $D_1t$ could eventually appear only at $\N\times {\D}$, and that $D_1t\geq c\tau_\min$ on $[n,n+1]\times\mathring{\D}$ for every $n\geq 0$.
Again by a parametrized splining method, we can ask also that the extended function $t$ is smooth at $\N\times {\mathbb{D}}$ and that there exists $0<\varepsilon\ll 1$ such that $D_1t(s,z)=\tau_\max$ if $s\in[0,\varepsilon)\cup(1-\varepsilon,1]$.
\end{proof}

% 	From the standing assumption that $\kappa(\gamma_0)>2\pi$ we can choose $c\in(0,1)$ such that 
% 	\begin{equation}\label{choice_of_c}
% 	c\,\kappa(\gamma_0)>2\pi.
% 	\end{equation}
% 	From now on we work with this fixed choice of $c$.

Consider disks $D_s$ defined by
\begin{equation}
D_s = \{ \varphi^{t(s,z)}(0,z) \mid z\in\D \}
\end{equation}
where $s$ varies on $[0,1]$. Note that $D_1=D_0=\{0\}\times\D$ in view of~\eqref{return_map_notation}. Hence we get a continuous $\R/\Z$-family $\{D_s\}_{s\in\R/\Z}$ of smooth disks.
{Actually, the family $\{D_s\}_{s\in\R/\Z}$ defines a smooth foliation of $\R/\Z\times\D$.}

%
%\begin{lemma}\label{lemma_smooth_foliation}
%The family $\{D_s\}_{s\in\R/\Z}$ defines a smooth foliation of $\R/\Z\times\D$.
%\end{lemma}
%
%\begin{proof}
%It suffices to study smoothness of the foliation near the disk $D_0$. We claim that there exists $\varepsilon>0$ small enough such that $D_{1+u} = \varphi^{\tau_\max u}(D_0)$ for all $u \in (-\varepsilon,0]$, and $D_{u} = \varphi^{\tau_\max u}(D_0)$ for all $u \in [0,\varepsilon)$. In fact, let $\varepsilon>0$ be small enough so that (iii) in Lemma~\ref{lemma_flipping_pages_time} holds. It follows from (i) and (iii) in Lemma~\ref{lemma_flipping_pages_time} that $t(s,z)=\tau_\max s$ for all $s\in[0,\varepsilon)$ and that $t(s,z) = \tau(z) + \tau_\max (s-1)$ for all $s\in(1-\varepsilon,1]$, for all $z\in\D$. In particular, the $D_{u} = \varphi^{\tau_\max u}(D_0)$ holds for all $u \in [0,\varepsilon)$. With $u\in(-\varepsilon,0]$ we have
%\begin{equation}
%\begin{aligned}
%p\in D_{1+u} \qquad \Leftrightarrow \qquad p &= \varphi^{t(1+u,z)}(0,z) \qquad \text{(for some $z\in\D$)} \\
%&= \varphi^{\tau(z) + \tau_\max(1+u-1)}(0,z) \\
%&= \varphi^{\tau_\max u}(\varphi^{\tau(z)}(0,z)) \\
%&= \varphi^{\tau_\max u}(0,h(z)) \in \varphi^{\tau_\max u}(D_0)
%\end{aligned}
%\end{equation}
%It follows that the disks fit around $D_0$ as the smooth foliation $\{\varphi^{\tau_\max u}(D_0)\}_{|u|<\varepsilon}$.
%\end{proof}

\begin{lemma}\label{lemma_property_f}
In addition to properties (a)-(e) from Proposition~\ref{prop_main} one can assume that the map $\Psi$ in~\eqref{map_Psi_prop_main} satisfies the following property.
\begin{itemize}
\item[(f)] If $\varphi^t$ denotes the flow of $W$ then for every $c\in(0,1)$ there exists a smooth function $t:[0,1]\times \D \to [0,+\infty)$ satisfying properties (i)-(iii) in Lemma~\ref{lemma_flipping_pages_time}, and a smooth isotopy $\{h_s:\D\to\D\}_{s\in[0,1]}$ such that
\begin{equation}\label{isotopy_h_tube}
\varphi^{t(s,z)}(0,z)=(s,h_s(z))
\end{equation}
for all $(s,z) \in [0,1]\times\D$.
\end{itemize}
\end{lemma}

\begin{proof}
As a consequence of Lemma~\ref{lemma_flipping_pages_time} and {since $\{D_s\}_{s\in\R/\Z}$ is a smooth foliation of $\R/\Z\times\mathbb{D}$, }% ~\ref{lemma_smooth_foliation} 
one finds an orientation preserving diffeomorphism $\Phi : \R/\Z\times\D \to \R/\Z\times\D$ satisfying $\Phi(\{s\}\times\D) = D_s$ for all $s\in\R/\Z$. If $\widetilde W = \Phi^*W$ and $\tilde\varphi^t$ denotes the flow of $\widetilde W$ then we find a smooth isotopy $\{h_s\}_{s\in[0,1]}$ of self-diffeomorphisms of $\D$ satisfying $h_0=id$, $h_1=h$, uniquely determined by $\tilde\varphi^{t(s,z)}(0,z) = (s,h_s(z))$. Here $h$ is the return map~\eqref{return_map_notation}. The proof is concluded if we revert the notation from $\tilde\varphi^t$, $\widetilde W$ and $\Phi \circ \Psi$ back to $\varphi^t$, $W$ and $\Psi$.
\end{proof}

\begin{notation}
If $z:[a,b] \to \C\setminus\{0\}$ is continuous then we denote
\begin{equation}\label{def winding}
\wind_{[a,b]}(z) = \frac{\theta(b)-\theta(a)}{2\pi}
\end{equation}
where $\theta:[a,b]\to\R$ is continuous and satisfies $z(s)=|z(s)|e^{i\theta(s)}$.
\end{notation}

Up to rotating the coordinate system on $S^3 \setminus \gamma_0$ obtained via the map~$\Psi$, i.e., up to changing $\Psi(t,z)$ by $\Psi(t,e^{2\pi kt}z)$ for some $k\in\Z$, we may assume that the following property holds: for any pair of loops $\alpha,\beta:\R/\Z \to \mathring{\D}$ satisfying $\alpha(t) \neq \beta(t) \ \forall t$, the linking number in $S^3$ between  the loops $t\mapsto \Psi(t,\alpha(t))$ and $t\mapsto \Psi(t,\beta(t))$ is equal to $\wind_{[0,1]}(\beta(t)-\alpha(t))$. From now on we assume that $\Psi$ satisfies this property.

Let $I=\{f_t\}_{t\in[0,1]} \subset \text{Diff}^1(\mathring\D)$ be a $C^1$-isotopy joining $f_0=\mathrm{id}_{\mathring\D}$ to $f_1=f$. Extend the isotopy on $\R_+=[0,+\infty)$ by setting $f_{1+t}=f_t\circ f$.

\begin{theorem}[\cite{Flo}]\label{Florio}
	Let $x,y \in \mathring{\D}$, $x\neq y$, and $T\geq0$ be fixed arbitrarily. Then 
	$$
	\wind_{[0,T]}(f_{t}(y)-f_{t}(x)) = \wind_{[0,T]}\left( Df_{t}(z)(y-x) \right)
	$$
	holds for some $z$ in the segment $[x,y]$ joining $x$ to $y$.
\end{theorem}

From the standing assumption that $\kappa(\gamma_0)>2\pi$ we can choose $0<\varepsilon\ll1$ such that 
\begin{equation}\label{choice_of_eps}
\kappa(\gamma_0)-\varepsilon>2\pi.
\end{equation}
Our main technical statement reads as follows.

\begin{proposition}\label{prop_main_main}
%There exists $\eta>0$ and $\bar T>0$ such that for every pair of recurrent points $p,q \in \Psi(\{0\}\times \mathring{\D})$ of $\phi^t$ in distinct trajectories and every ...
%Assume that $\kappa_\sigma>\tau_{\min}/2\pi$ for some symplectic trivialization~$\sigma$ of $(\xi,d\lambda)$. 
Let $p,q \in \Psi(\{0\}\times \mathring{\D})$ be recurrent points for $\phi^t$ in distinct trajectories. Then 
\begin{equation} 
\ell(p,q) \geq\frac{\kappa(\gamma_0)-\varepsilon-2\pi}{2\pi\tau^2_\max} > 0
\end{equation}
where $\ell(p,q)$ is the number~\eqref{funny_ell}.
\end{proposition}

{

\subsubsection*{Key ideas of the proof of Proposition~\ref{prop_main_main}}

We resume here the main ideas and steps of the proof of Proposition~\ref{prop_main_main}; the proof can be found in the next subsection.
For simplicity, consider periodic points $p,q\in\Psi(\{0\}\times\D)$ of the flow $\{\phi^t\}$ with period $T_p,T_q$ respectively.
The quantity we want to estimate is $\ell(p,q) = \frac{\mathrm{link}(\gamma_p,\gamma_q)}{T_pT_q}$ where $\gamma_p=\phi^{[0,T_p]}(p), \gamma_q=\phi^{[0,T_q]}(q)$.
We now explain why these periodic orbits are positively linked when $\kappa(\gamma_0) > 2\pi$ further assuming that both have the same linking number~$n$ with $\gamma_0$, and that $n$ is large enough.
This means that $\Psi^{-1}(p)$ and $\Psi^{-1}(q)$ have the same period $n\in\Z$ with respect to the flow $\phi^t=\varphi^s$ parametrised by the ``flipping pages'' parameter~$s$. 
%This means that each periodic orbit makes $n$ turns around the solid torus $\R/\Z\times\D$ before closing up or, equivalently, that both orbits $\gamma_p,\gamma_q$ have the same linking number~$n$ with $\gamma_0$. 
%Thus, the quantity we want to calculate becomes $\ell(p,q) = \frac{\mathrm{link}(\gamma_p,\gamma_q)}{T_pT_q}$, where $\gamma_p=\phi^{[0,T_p]}(p), \gamma_q=\phi^{[0,T_q]}(q)$. 

Recall that $\{\varphi^s\}$ induces a smooth isotopy $\{h_s\}$ of self-diffeomorphisms of $\D$. Because of the transversality of the flow to each page $\{s\}\times\D$, the linking number can be computed in terms of winding numbers with respect to the isotopy $\{h_s\}$, if we assume that $\Psi$ satisfies the following normalisation condition: given loops $c_1,c_2:\R/\Z \to \mathring{\D}$ satisfying $c_1(s) \neq c_2(s)$~$\forall s$, the loops $s \mapsto \Psi(s,c_1(s))$ and $s \mapsto \Psi(s,c_2(s))$ in $S^3$ have linking number equal to $\wind_{s\in[0,1]} (c_2(s)-c_1(s))$.
Consider the sets $\mathcal{P} := \Psi^{-1}(\gamma_p) \cap \{0\}\times\D$ and $\mathcal{Q} := \Psi^{-1}(\gamma_q) \cap \{0\}\times\D$. 
Then one chooses $(0,y_0)\in\mathcal{Q}$ arbitrarily to get
\begin{equation}
\label{eq_summing_up}
\begin{aligned}
\mathrm{link}(\gamma_p,\gamma_q) &= \sum_{(0,x)\in\mathcal{P}}\sum_{(0,y)\in\mathcal{Q}}\mathrm{wind}_{s\in[0,1]}(h_s(y)-h_s(x)) \\
&= \sum_{(0,x)\in\mathcal{P}}\mathrm{wind}_{s\in[0,n]}(h_s(y_0)-h_s(x)) \, .
\end{aligned}
\end{equation}
%where $(0,y_0)\in\mathcal{Q}$ is chosen arbitrarily. 
%Observe that we are calculating the winding number with respect to a Seifert frame. 
By Theorem~\ref{Florio}, we can replace each winding number with a linearised one, that is, for each $(0,x)\in\mathcal{P}$ there exists $z(x)\in\D$ such that
\[
\mathrm{wind}_{s\in[0,n]}(h_s(y_0)-h_s(x))=\mathrm{wind}_{s\in[0,n]}(Dh_s(z(x))(y_0-x))\,.
\]
Consider the (trivial) rank-$2$ vector bundle over the $3$-manifold $\R/\Z \times \mathring{\D}$ whose fiber over the point $(s,z)$ is $T_z\D = \R^2$.
We can choose a special frame of this vector bundle with the following property: its push forward by $\Psi$ projects along the Reeb vector field to a frame of $\xi$ on $S^3 \setminus\gamma_0$ that extends to a global frame of $\xi$ on $S^3$.
By definition
%, we have information on $\kappa(\gamma_0)$, that is the angle variation of the linearised dynamics in a global frame:
\[
\mathrm{wind}_{s\in[0,n]}^{\mathrm{Global}}(Dh_s(z(x))(y_0-x))\geq \dfrac{\kappa(\gamma_0)-\varepsilon}{2\pi}n\,,
\]
where $\mathrm{wind}_{s\in[0,n]}^{\mathrm{Global}}$ denotes the winding number with respect to a special frame mentioned above, provided that $0<\varepsilon < \kappa(\gamma_0)-2\pi$ and that $n$ is large enough.
%$n$ is seen as the linking number of the orbit of $\Psi(0,z(x))$ over the time interval corresponding to $[0,n]$ with $\gamma_0$.
%Here  of the rank-$2$ vector bundle over $\R/\Z \times \mathring{\D}$ whose fiber over $(s,z)$ is $T_z\D = \R^2$.
%This frame has the property that 
Denoting then by $X=\{X_s\}$ a non-autonomous vector field on $\mathring{\D}$ that is constant in the above mentioned special frame, we have
\[
\begin{aligned}
& \mathrm{wind}_{s\in[0,n]}(Dh_s(z(x))(y_0-x)) \\
& \quad = \mathrm{wind}_{s\in[0,n]}^{\mathrm{Global}}(Dh_s(z(x))(y_0-x))+\mathrm{wind}_{s\in[0,n]}(X(h_s(z(x)))) \,.
\end{aligned}
\]
Informally speaking, at each turn of the solid torus the above mentioned special frame rotates roughly once in the negative sense with respect to the Seifert frame. 
This is not quite precise, but is effectively precise when $n$ is large enough; later, in the actual proof of Proposition~\ref{prop_main_main}, we need to be careful about this point.
Hence
\[
\mathrm{wind}_{s\in[0,n]}(Dh_s(z(x))(y_0-x))\geq \dfrac{n}{2\pi}(\kappa(\gamma_0)-\varepsilon-2\pi)\,.
\]
Coming back to~\eqref{eq_summing_up}, and since the cardinality of $\mathcal{P}$ is $n$, we obtain that
\[
\dfrac{\mathrm{link}(\gamma_p,\gamma_q)}{T_pT_q}\geq \dfrac{n^2}{2\pi T_pT_q}(\kappa(\gamma_0)-\varepsilon-2\pi)\,.
\]
By the control of the flipping pages parameter $s$ with respect to time parameter $t$ of the original Reeb flow, it holds $n^2/T_pT_q\geq 1/\tau_\max^2$. One gets $\ell(p,q)\geq \frac{\kappa(\gamma_0)-\varepsilon-2\pi}{2\pi\tau_\max^2}$, which is strictly positive. 
When $n$ is not large enough, or when the linking numbers $n_p$ and $n_q$ of $\gamma_p$ and $\gamma_q$ with $\gamma_0$, respectively, do not coincide, then one applies the above arguments with a very large common multiple of $n_p$ and $n_q$ in the place of $n$.
In the sequel, we formalise these ideas and prove Proposition~\ref{prop_main_main}.

}

\subsubsection*{Proof of Proposition~\ref{prop_main_main}\label{detailed}}

From now on we are concerned with the proof of Proposition~\ref{prop_main_main}.
{We define loops $k(T_n,p),k(S_n,q)$ as in case {\bf A} of subsection~\ref{ssec_right_handedness}, but make here a more general construction.
We denote by $(s,z=x+iy)$ the coordinates on the domain $\R/\Z\times\D$ of the map~$\Psi$.
The points $p,q$ are assumed to belong to $\Psi(\{0\}\times \mathring{\D})$.
Choose an auxiliary Riemannian metric $g$ on $S^3$ that is Euclidean in local coordinates near~$p$ and~$q$ with respect to which the vector field~$X$ is constant.
It can also be assumed that $g$ induces the metric $dx^2+dy^2$ tangentially to $\Psi(\{0\}\times \mathring{\D})$.
The metric $g$ can easily be constructed by considering flow boxes for $X$ near $p$ and $q$ obtained by flowing small disks in $\Psi(\{0\}\times \mathring{\D})$ centred at these points.
Let $\mathcal{W}_p$ and $\mathcal{W}_q$ be small disjoint $g$-convex neighbourhoods of $p$ and $q$, respectively, where the metric $g$ has the above properties. We can assume that $\mathcal{W} = \mathcal{W}_p\cup\mathcal{W}_q \subset \Psi([-1/4,1/4]\times\D)$.
Consider shortest geodesic paths $\alpha_n \subset \mathcal{W}_p$ and $\beta_n \subset\mathcal{W}_q$ from $\phi^{T_n}(p)$ to $p$ and from $\phi^{S_n}(q)$ to $q$, respectively. 
The sequences $T_n,S_n$ are assumed to satisfy $$ T_n,S_n \to +\infty \quad \phi^{T_n}(p) \in \mathcal{W}_p \quad \phi^{S_n}(q) \in \mathcal{W}_q \, . $$
Note that such sequences exist if $p$ and $q$ are recurrent points, but here we only assume that $p$ and $q$ lie in distinct orbits, and that $T_n,S_n$ as above exist.
We find real numbers $s_n,\hat s_n$ and $z_n,\hat z_n \in \mathring{\D}$ such that $\phi^{T_n}(p)=\Psi(s_n,z_n)$, $\phi^{S_n}(q) = \Psi(\hat s_n,\hat z_n)$, $\max(|s_n|,|\hat s_n|) \leq 1/4$.}
%When $n$ is large enough 

%Moreover, $\alpha_n \subset \mathcal{W}_p$ and~$\beta_n \subset\mathcal{W}_q$.} 
%Up to selection of a bigger $n\geq1$, we find $\varepsilon>0$ small and convex $\varepsilon$-balls such that $B_{\varepsilon}(p)\cap B_{\varepsilon}(q)=\emptyset$, $\phi^{T_n}(p)\in B_{\varepsilon}(p)$ and $\phi^{S_n}(q)\in B_{\varepsilon}(q)$. 
There exist $m_n(p),m_n(q)\geq 1$ and 
\begin{equation*}
t^p_0=0<t^p_1<\dots<t^p_{m_n(p)}\leq T_n \qquad \qquad t^q_0=0<t^q_1<\dots<t^q_{m_n(q)}\leq S_n
\end{equation*}
characterized by
\begin{equation*}
\begin{aligned}
\{ t^p_0,\dots,t^p_{m_n(p)}\} &= \{t\in[0,T_n] \mid \phi^t(p) \in \Psi(\{0\}\times\mathring{\D}) \} \\
\{ t^q_0,\dots,t^q_{m_n(q)}\} &= \{t\in[0,S_n] \mid \phi^t(q) \in \Psi(\{0\}\times\mathring{\D}) \} \, .
\end{aligned}
\end{equation*}
{ Let $t^p_{m_n(p)+1}$ denote the next hitting time following $t^p_{m_n(p)}$.
We have estimates}
\begin{equation*}
m_n(p) \tau_\min \leq t^p_{m_n(p)} \leq T_n < t^p_{m_n(p)+1} \leq t^p_{m_n(p)} + \tau_\max \leq (m_n(p)+1)\tau_\max 
\end{equation*}
%where $t^p_{m_n(p)+1}$ denotes the hitting time following $t^p_{m_n(p)}$. 
and it follows that
\begin{equation}\label{estimates_closing_times_p}
\frac{m_n(p)}{T_n} \leq \frac{1}{\tau_\min}, \qquad\qquad \frac{T_n}{\tau_\max}-1 \leq m_n(p) \, . 
%\to +\infty.
\end{equation}
Analogously
\begin{equation}\label{estimates_closing_times_q}
\frac{m_n(q)}{S_n} \leq \frac{1}{\tau_\min}, \qquad\qquad \frac{S_n}{\tau_\max}-1 \leq m_n(q) \, . 
%\to +\infty.
\end{equation}

The loop $k(T_n,p)=\phi^{[0,T_n]}(p) + \hat\alpha_n$ is obtained by concatenating to $\phi^{[0,T_n]}(p)$ a $C^1$-small perturbation $\hat{\alpha}_n$ of $\alpha_n$ with fixed endpoints, and $k(S_n,q) = \phi^{[0,S_n]}(q) + \hat\beta_n$ is obtained by concatenating to $\phi^{[0,S_n]}(q)$ a $C^1$-small perturbation $\hat{\beta}_n$ of $\beta_n$ with fixed endpoints. 
The $\hat{\alpha}_n,\hat{\beta}_n$ are chosen so that $k(T_n,p)$ and $k(S_n,q)$ do not intersect. 
Observe that different choices of $\hat{\alpha}_n,\hat{\beta}_n$ determine an ambiguity by an additive integer in $[-m_n(p)-m_n(q),m_n(p)+m_n(q)]$ of the value of $\link(k(T_n,p),k(S_n,q))$. 
This is so because $\hat{\alpha}_n,\hat{\beta}_n$ are $C^1$-close enough to $\alpha_n,\beta_n$.
As in subsection~\ref{ssec_right_handedness}, we continue to denote
$$
\link_-(\phi^{[0,T_n]}(p),\phi^{[0,S_n]}(q)) = \liminf_{\tiny \hat\alpha\stackrel{C^1}{\to}\alpha_n \ \hat\beta\stackrel{C^1}{\to}\beta_n} \ \link(k(T_n,p),k(S_n,q)) \, .
$$
In view of~\eqref{estimates_closing_times_p}-\eqref{estimates_closing_times_q} {and of the fact that $\hat{\alpha}_n,\hat{\beta}_n$ are $C^1$-close enough to $\alpha_n,\beta_n$, we deduce the estimate
% the ambiguity of the value of 
%\begin{equation*}
%\frac{\link(k(T_n,p),k(S_n,q))}{T_nS_n}
%\end{equation*}
%is at most
%\begin{equation*}
%\frac{m_n(p)+m_n(q)}{T_nS_n} \leq \frac{1}{\tau_\min}\left(\frac{1}{S_n}+\frac{1}{T_n} \right) \to 0
%\end{equation*}
%provided  $\hat{\alpha}_n,\hat{\beta}_n$ are $C^1$-close enough to $\alpha_n,\beta_n$. In particular
%\begin{equation}\label{auxi_link_link_minus}
%\begin{aligned}
%& \frac{\link_-(\phi^{[0,T_n]}(p),\phi^{[0,S_n]}(q))}{T_nS_n} \\
%& \qquad \leq \frac{\link(k(T_n,p),k(S_n,q))}{T_nS_n} \\
%& \qquad \leq \frac{\link_-(\phi^{[0,T_n]}(p),\phi^{[0,S_n]}(q))}{T_nS_n} + \frac{1}{\tau_\min}\left(\frac{1}{S_n}+\frac{1}{T_n} \right)
%\end{aligned}
%\end{equation}
\begin{equation}\label{auxi_link_link_minus}
\begin{aligned}
& \left| \frac{\link_-(\phi^{[0,T_n]}(p),\phi^{[0,S_n]}(q))}{T_nS_n} -\frac{\link(k(T_n,p),k(S_n,q))}{T_nS_n} \right| %\\
%& \qquad 
\leq \frac{1}{\tau_\min}\left(\frac{1}{S_n}+\frac{1}{T_n} \right) \, . 
%\to 0\,.
\end{aligned}
\end{equation}}
%provided  $\hat{\alpha}_n,\hat{\beta}_n$ are $C^1$-close enough to $\alpha_n,\beta_n$.

{ We now study the effect of replacing $T_n$ by a different sequence $T'_n\to+\infty$ such that $\{t \in \R \mid \phi^t(p) \in \mathcal{W}_p \}$ contains the closed interval whose endpoints are $T_n,T'_n$.
%$|T_n-T'_n|\to0$. 
%Assume $T'_n\leq T_n$ without loss of generality. 
Let $\alpha_n'$ be the shortest geodesic arc from $\phi^{T'_n}(p)$ to $p$, which must again satisfy~$\alpha_n' \subset \mathcal{W}$.
Let $\hat\alpha_n'$ be a small $C^1$-perturbation of $\alpha'_n$ with fixed endpoints, so that the corresponding loop $k(T'_n,p)$ as described before does not intersect $k(S_n,q)$.}
Construct the loop $c=(-\hat\alpha'_n) + \phi^{T'_n \to T_n}(p) + \hat\alpha_n$.
For $n$ large we estimate
\begin{equation}
\label{estimate_little_loop_c}
|\link(c,k(S_n,q))| \leq 2m_n(q)
\end{equation}
provided $\hat\alpha_n',\hat\alpha_n$ are $C^1$-close enough to $\alpha'_n,\alpha_n$. The identity $k(T'_n,p) + c=k(T_n,p)$ implies that
\begin{equation*}
e_n = \frac{\link(k(T_n,p),k(S_n,q))}{T_nS_n} - \frac{\link(k(T_n',p),k(S_n,q))}{T_n'S_n}
\end{equation*}
can be estimated by
\begin{equation*}
\begin{aligned}
|e_n| &= \left| \frac{\link(k(T_n,p),k(S_n,q))}{T_nS_n} - \frac{\link(k(T_n',p),k(S_n,q))}{T_n'S_n} \right| \\
&= \left| \frac{\link(k(T_n,p),k(S_n,q))}{T_nS_n} - \frac{\link(k(T_n,p),k(S_n,q))-\link(c,k(S_n,q))}{T_nS_n}\frac{T_n}{T'_n} \right| \\
&\leq \left| 1-\frac{T_n}{T'_n} \right|\cdot \left| \frac{\link(k(T_n,p),k(S_n,q))}{T_nS_n} \right| + \left| \frac{\link(c,k(S_n,q))}{S_nT'_n} \right| \\
&\leq \left| 1-\frac{T_n}{T'_n} \right| \cdot\left| \frac{\link(k(T_n,p),k(S_n,q))}{T_nS_n} \right| + \frac{2m_n(q)}{S_nT'_n}
\end{aligned}
\end{equation*}
provided $\hat\alpha_n',\hat\alpha_n,\hat\beta_n$ are $C^1$-close enough to $\alpha'_n,\alpha_n,\beta_n$. Together with~\eqref{estimates_closing_times_p} and~\eqref{auxi_link_link_minus} this implies
\begin{equation}\label{long_estimate_link_minus}
\begin{aligned}
& \frac{\link_-(\phi^{[0,T'_n]}(p),\phi^{[0,S_n]}(q))}{T'_nS_n} \\
& \geq \frac{\link(k(T_n',p),k(S_n,q))}{T_n'S_n} - \frac{1}{\tau_\min}\left(\frac{1}{S_n}+\frac{1}{T_n'} \right) \\
&\geq \frac{\link(k(T_n,p),k(S_n,q))}{T_nS_n} - |e_n| \ - \ \frac{1}{\tau_\min}\left(\frac{1}{S_n}+\frac{1}{T_n'} \right) \\
%&\geq \frac{\link(k(T_n,p),k(S_n,q))}{T_nS_n} -\left|1-\frac{T_n}{T'_n} \right| \left| \frac{\link(k(T_n,p),k(S_n,q))}{T_nS_n} \right| \\
%& \qquad - \frac{2m_n(p)}{T_nS_n}\frac{T_n}{T'_n} - \frac{1}{\tau_\min}\left(\frac{1}{S_n}+\frac{1}{T_n'} \right) \\
&\geq \frac{\link(k(T_n,p),k(S_n,q))}{T_nS_n} - \left|1-\frac{T_n}{T'_n} \right| \cdot\left| \frac{\link(k(T_n,p),k(S_n,q))}{T_nS_n} \right| \\
& \qquad - \frac{2}{\tau_\min }\frac{1}{T'_n} - \frac{1}{\tau_\min}\left(\frac{1}{S_n}+\frac{1}{T_n'} \right) \\
& \geq \frac{\link_-(\phi^{[0,T_n]}(p),\phi^{[0,S_n]}(q))}{T_nS_n} - \left|1-\frac{T_n}{T'_n} \right|\cdot \left| \frac{\link_-(\phi^{[0,T_n]}(p),\phi^{[0,S_n]}(q))}{T_nS_n} \right| \\
& \qquad - \frac{2}{\tau_\min }\frac{1}{T'_n} - \frac{1}{\tau_\min}\left(\frac{1}{S_n}+\frac{1}{T_n'} \right) - \frac{1}{\tau_\min}\left(\frac{1}{S_n}+\frac{1}{T_n} \right)\cdot\left|1-\frac{T_n}{T'_n} \right| \\
& \qquad - \frac{1}{\tau_\min}\left(\frac{1}{S_n}+\frac{1}{T_n} \right)
\end{aligned}
\end{equation}
provided $\hat\alpha_n',\hat\alpha_n,\hat\beta_n$ are $C^1$-close enough to $\alpha'_n,\alpha_n,\beta_n$. 
%{\color{magenta} If $A_n,A'_n \in (0,+\infty)$ satisfy $A'_n\geq A_n - \delta_n A_n - \tilde\delta_n \ \forall n$ where {\color{red} $\delta_n,\tilde\delta_n \to 0^+$}, then $\liminf_nA_n>0$ implies that $\liminf_nA'_n \geq \liminf_nA_n>0$. This fact combined with}
Estimate~\eqref{long_estimate_link_minus} and an analogous argument interchanging the roles of $p$ and $q$ show the following statement. {If there exists $a>0$ such that for every pair of sequences $T_n$, $S_n$ satisfying $T_n\to+\infty$, $S_n\to+\infty$, $\phi^{T_n}(p) \in \mathcal{W}_p$ and $\phi^{S_n}(q) \in \mathcal{W}_q$, one finds sequences $T_n'$, $S_n'$ satisfying the same properties and, in addition, $\{t \in \R \mid \phi^t(p) \in \mathcal{W}_p \}$ contains the closed interval bounded by $T_n$ and $T'_n$, $\{t \in \R \mid \phi^t(q) \in \mathcal{W}_q \}$ contains the closed interval bounded by $S_n$ and~$S'_n$,
\begin{equation*}
\sup_n \ |T_n-T'_n| + |S_n-S'_n| < +\infty \qquad \liminf_{n\to\infty} \ \frac{\link_-(\phi^{[0,T'_n]}(p),\phi^{[0,S'_n]}(q))}{T'_nS'_n} \geq a
\end{equation*}
then
\begin{equation*}
\ell(p,q) \geq a.
\end{equation*}
It follows from these arguments that $\phi^{T_n}(p),\phi^{S_n}(q)$ can be assumed to belong to $\Psi(\{0\}\times\mathring{\D})$.}
We proceed under this assumption.

%\begin{remark}\label{remark k are good for kappagamma}
%Observe then that the paths $k(T'_n,p)$ and $k(S'_n,q)$ are actually paths $k(T'_n,p;\Psi(\{0\}\times\mathring{\D}))$ and $k(S'_n,q;\Psi(\{0\}\times\mathring{\D}))$ used in the definition of the quantity $\kappa(\gamma_0)$ in \eqref{def_number_kappa}, with respect to the global surface of section $\Psi(\{0\}\times\mathring{\mathbb{D}})$.
%\end{remark}

%It follows that 
%\begin{equation*}
%\begin{aligned}
%&\frac{\link_-(\phi^{[0,T'_n]}(p),\phi^{[0,S_n]}(q))}{T'_nS_n} \\
%&\geq \frac{\link_-(\phi^{[0,T_n]}(p),\phi^{[0,S_n]}(q))}{T_nS_n}
%\end{aligned}
%\end{equation*}
%
%
%
%
%It follows that 
%\begin{equation*}
%\begin{aligned}
%&\frac{\link_-(\phi^{[0,T'_n]}(p),\phi^{[0,S_n]}(q))}{T'_nS_n} \geq \frac{\link_-(\phi^{[0,T_n]}(p),\phi^{[0,S_n]}(q))}{T_nS_n}
%\end{aligned}
%\end{equation*}
%Reversing the roles of $T'_n$ and $T_n$
%\begin{equation*}
%\frac{\link_-(\phi^{[0,T'_n]}(p),\phi^{[0,S_n]}(q))}{T'_nS_n} = \frac{\link_-(\phi^{[0,T_n]}(p),\phi^{[0,S_n]}(q))}{T_nS_n}
%\end{equation*}
%for all sequences $T_n,T'_n,S_n$ satisfying $\phi^{T_n}(p)\to p$, $\phi^{T'_n}(p)\to p$, $\phi^{S_n}(q) \to q$ and $|T'_n-T_n|\to 0$. 

%As a consequence one concludes that in order to prove an estimate $\ell(p,q)>0$ there is no loss of generality to assume that $s_n,\hat s_n<0$, $T_n < t^p_{m_n(p)+1}$, $S_n < t^q_{m_n(q)+1}$, $t^p_{m_n(p)+1} - T_n \to0$ and $t^q_{m_n(q)+1} - S_n\to0$.

%In view of the above we can assume that $\phi^{T_n}(p),\phi^{S_n}(q)$ belong to $\Psi(\{0\}\times\mathring{\D})$. 

Since $\phi^{T_n}(p) \in \mathcal{W}_p \cap \Psi(\{0\}\times\D)$, $\phi^{S_n}(q) \in \mathcal{W}_q \cap \Psi(\{0\}\times\D)$, and the metric has a special form on $\mathcal{W}$, we have that $\alpha_n,\beta_n$ are images under $\Psi$ of straight line segments in $\{0\}\times\mathring{\D}$. Denote by $z_0^p,\dots,z_{m_n(p)}^p$ and by $z_0^q,\dots,z_{m_n(q)}^q \in \mathring{\D}$ the points in $\mathring{\D}$ uniquely determined by
\begin{equation*}
\Psi(0,z^p_i) = \phi^{t^p_i}(p), \qquad \qquad \Psi(0,z^q_j) = \phi^{t^q_j}(q).
\end{equation*}
%Similarly $z_0^q,\dots,z_{m_n(q)}^q \in \mathring{\D}$ are determined by
%\begin{equation*}
%\Psi(0,z^q_j) = \phi^{t^q_j}(q)\, .
%\end{equation*}
It follows that $$ \alpha_n = \Psi(\{0\}\times[z_{m_n(p)}^p,z_0^p]) \qquad\qquad \beta_n = \Psi(\{0\}\times[z_{m_n(q)}^p,z_0^q]) $$ where $[z,w]$ denotes the path $u \in [0,1] \mapsto (1-u)z+u w$. By transversality of the flow to $\Psi(\{0\}\times\mathring{\D})$ we can choose $\hat\alpha_n,\hat\beta_n$ to be {equal to $\Psi(\{0\}\times c_n^p)$ and $\Psi(\{0\}\times c_n^q)$, respectively, for some paths $c_n^p,c_n^q:[0,1]\to\mathring{\D}$, and so }contained in $\Psi(\{0\}\times\mathring{\D})$. %, i.e. there are $C^1$-small perturbations $c^p_n,c^q_n:[0,1]\to\mathring{\D}$ of $[z_{m_n(p)}^p,z_0^p],[z_{m_n(q)}^p,z_0^q]$ respectively, such that $c^p_n$ misses the points $z_0^q,\dots,z_{m_n(q)}^q$ and $c^q_n$ misses the points $z_0^p,\dots,z_{m_n(p)}^p$, and such that we can choose $$ \hat\alpha_n = \Psi(0\times c^p_n), \qquad\qquad  \hat\beta_n = \Psi(0\times c^q_n). $$
 Define paths
\begin{equation*}
k^p_0,\dots,k^p_{m_n(p)-1},k^q_0,\dots,k^q_{m_n(q)-1}:[0,1] \to \mathring{\D}
\end{equation*}
by 
\begin{equation*}
\Psi(s,k^p_i(s)) = \phi^{t(s,z^p_i)}(\Psi(0,z^p_i)) \qquad\qquad \Psi(s,k^q_j(s)) = \phi^{t(s,z^q_j)}(\Psi(0,z^q_j))
\end{equation*}
where $t(s,z)$ is the function given by Lemma~\ref{lemma_flipping_pages_time}. It follows that
\begin{equation*}
\begin{aligned}
& i \in \{0,\dots,m_n(p)-1\} \qquad \Rightarrow \qquad k^p_i(0) = z^p_i, \qquad k^p_i(1) = z^p_{i+1} \\
& j \in \{0,\dots,m_n(q)-1\} \qquad \Rightarrow \qquad k^q_j(0) = z^p_j, \qquad k^q_j(1) = z^q_{j+1}
\end{aligned}
\end{equation*}
Fix $0<\delta <1$. Consider paths
\begin{equation*}
\hat k^p_i,\hat k^q_j:[0,1] \to \mathring{\D} \qquad (i,j) \in \{0,\dots,m_n(p)-1\} \times \{0,\dots,m_n(q)-1\}
\end{equation*}
defined by
\begin{equation}
\begin{aligned}
& \text{If} \ 0 \leq i \leq m_n(p)-2: \ \hat k^p_i(s) = \begin{cases} k^p_i(\frac{s}{1-\delta}) \ & \text{if} \ s\in[0,1-\delta] \\ z^p_{i+1} \ & \text{if} \ s\in[1-\delta,1] \end{cases} \\
& \text{If} \ 0 \leq j \leq m_n(q)-2: \ \hat k^q_j(s) = \begin{cases} k^q_j(\frac{s}{1-\delta}) \ & \text{if} \ s\in[0,1-\delta] \\ z^q_{j+1} \ & \text{if} \ s\in[1-\delta,1] \end{cases}
\end{aligned}
\end{equation}
and 
\begin{equation*}
\begin{aligned}
& \hat k^p_{m_n(p)-1}(s) = \begin{cases} k^p_{m_n(p)-1}(\frac{s}{1-\delta}) & \text{if} \ s\in[0,1-\delta] \\ c^p_n(\frac{s-1+\delta}{\delta}) & \text{if} \ s\in[1-\delta,1] \end{cases} \\
& \hat k^q_{m_n(q)-1}(s) = \begin{cases} k^q_{m_n(q)-1}(\frac{s}{1-\delta}) & \text{if} \ s\in[0,1-\delta] \\ c^q_n(\frac{s-1+\delta}{\delta}) & \text{if} \ s\in[1-\delta,1] \end{cases}
\end{aligned}
\end{equation*}
In the following we extend $k^p_i,\hat k^p_i$ to all $i\in\Z$ in a $m_n(p)$-periodically way. Similarly, we extend $k^q_j,\hat k^q_j$ to all $j\in\Z$ in a $m_n(q)$-periodically way. With this convention we have
\begin{equation*}
\hat k^p_i(1) = \hat k^p_{i+1}(0) \ \forall i \qquad \hat k^q_j(1) = \hat k^q_{j+1}(0) \ \forall j
\end{equation*}
Let $L = {\rm lcm}(m_n(p),m_n(q))$ and $r_n(p),r_n(q) \in \N$ satisfy
\begin{equation*}
L = r_n(p)m_n(p) = r_n(q)m_n(q)
\end{equation*}
Let us denote by $k(T_n,p)^{r_n(p)}, k(S_n,q)^{r_n(q)}$ the $r_n(p)$-fold iteration of $k(T_n,p)$ and the $r_n(q)$-fold iteration of $k(S_n,q)$, respectively. Then, by the transversality of $\phi^t$ to each $\Psi(\{s\}\times\mathring\D)$, it holds
\begin{equation}\label{identity_linking_winding}
\link(k(T_n,p)^{r_n(p)}, k(S_n,q)^{r_n(q)}) = \sum_{i,j=0}^{L-1} \wind_{s\in[0,1]}(\hat k^p_i(s)-\hat k^q_j(s)).
\end{equation}
With $i,j \in \{0,\dots,L-1\}$ consider
\begin{equation}
e_{ij} = \wind_{s\in[1-\delta,1]}(\hat k^p_i(s)-\hat k^q_j(s)).
\end{equation}
For each pair $(i,j)$ there are two cases to be studied separately:
\begin{itemize}
	\item[(i)] 	$m_n(p)$ does not divide $i+1$ and $m_n(q)$ does not divide $j+1$.
	\item[(ii)] $m_n(p)$ divides $i+1$ or $m_n(q)$ does divides $j+1$.
\end{itemize}
In case (i) both paths $\hat k^p_i|_{[1-\delta,1]}$ and $\hat k^q_j|_{[1-\delta,1]}$ are constant paths, hence $e_{ij}=0$. In case (ii) there are two subcases: either one of the paths is a $C^1$-perturbation of a straight line segment in $\mathring{\D}$ and the other is a point in the complement, or both are contained in disjoint open balls (assuming $n$ is large enough). In both subcases we get $|e_{ij}|\leq1$. The number of multiples of $m_n(p)$ in $[0,L-1]$ is $r_n(p)$, and the number of multiples of $m_n(q)$ in $[0,L-1]$ is $r_n(q)$. 
Hence, there at most {$L(r_n(p)+r_n(q))$} pairs $(i,j)$ falling in case (ii), and this yields the estimate
\begin{equation}\label{estimate_e_ij}
\sum_{i,j=0}^{L-1} |e_{ij}| \leq { L(r_n(p)+r_n(q))} \, . %r_n(p)+r_n(q).
\end{equation}
Now note that 
\begin{equation}\label{winding_reparametrization}
%\begin{aligned}
\sum_{i,j=0}^{L-1} \wind_{s\in[0,1]}(\hat k^p_i(s)-\hat k^q_j(s)) %&= \sum_{i,j=0}^{L-1} \wind_{s\in[0,1-\delta]}(\hat k^p_i(s)-\hat k^q_j(s)) + e_{ij} \\
%&
 = \sum_{i,j=0}^{L-1} \wind_{s\in[0,1]}(k^p_i(s)-k^q_j(s)) + e_{ij}\,.
%\end{aligned}
\end{equation}
Combining~\eqref{identity_linking_winding},~\eqref{estimate_e_ij} and~\eqref{winding_reparametrization}
\begin{equation}\label{broken_winding_pre}
\begin{aligned}
& \left| \link(k(T_n,p)^{r_n(p)}, k(S_n,q)^{r_n(q)}) - \sum_{i,j=0}^{L-1} \wind_{s\in[0,1]}(k^p_i(s)-k^q_j(s)) \right| \\
& \qquad \leq { L(r_n(p)+r_n(q))} \, . %r_n(p)+r_n(q).
\end{aligned}
\end{equation}

Now define maps
\begin{equation}
\begin{aligned}
& \Gamma^p:\R/L\Z \to \mathring{\D} \qquad \Gamma^p(s) = k^p_{\lfloor s \rfloor}(s-\lfloor s \rfloor) \\
& \Gamma^q:\R/L\Z \to \mathring{\D} \qquad \Gamma^q(s) = k^q_{\lfloor s \rfloor}(s-\lfloor s \rfloor)
\end{aligned}
\end{equation}
Note that $\Gamma^p$ is discontinuous, unless $p$ lies on a periodic orbit. Analogously, $\Gamma^q$ is discontinuous, unless $q$ lies on a periodic orbit. For each $j\in\Z/L\Z$ define:
\begin{equation}
j*\Gamma^q(s) = \Gamma^q(s+j)\,.
\end{equation}
Define
\begin{eqnarray*}
	& E_p = \{0,m_n(p),\dots,(r_n(p)-1)m_n(p)\}\,, \\
	& E_q = \{0,m_n(q),\dots,(r_n(q)-1)m_n(q)\}\,.
\end{eqnarray*}
We view $E_p,E_q$ as subsets of $\Z/L\Z$. The set of discontinuity points of $\Gamma^p$ is contained in $E_p$ whose cardinality is equal to $r_n(p)$. The set of discontinuity points of $j*\Gamma^q$ is contained in $E_q-j$ whose cardinality is $r_n(q)$. We then write
\begin{equation}\label{broken_winding}
%\begin{aligned}
\sum_{i,j=0}^{L-1} \wind_{s\in[0,1]}(k^p_i(s)-k^q_j(s)) %&= \sum_{i=0}^{L-1} \sum_{j=0}^{L-1} \wind_{s\in[0,1]}(k^p_i(s)-k^q_{i+j}(s)) \\
%&
= \sum_{j=0}^{L-1} \wind_{s\in[0,L]}(\Gamma^p(s)-j*\Gamma^q(s))\,.
%\end{aligned}
\end{equation}
Note that the set $E_p \cup (E_q-j)$ divides the circle $\R/L\Z$ into $N_j$ closed intervals $I^j_1,\dots,I^j_{N_j}$. Note that $N_j \leq r_n(p)+r_n(q)$. We denote the end points of these intervals by $I^j_\lambda = [a^j_\lambda,b^j_\lambda]$ and their lengths by $|I^j_\lambda| = b^j_\lambda - a^j_\lambda \in \N$. The end points of these intervals belong to $\Z/L\Z$. For each $I^j_\lambda=[a^j_\lambda,b^j_\lambda]$ we find points $\zeta_\lambda,\zeta'_\lambda \in \mathring{\D}$ such that
\begin{equation}
h_s(\zeta_\lambda) = \Gamma^p(s+a^j_\lambda) \qquad h_s(\zeta'_\lambda) = j*\Gamma^q(s+a^j_\lambda) \qquad \forall s\in[0,b^j_\lambda-a^j_\lambda]
\end{equation}
where $h_s$ is the unique extension of the isotopy~\eqref{isotopy_h_tube} to $s\in[0,+\infty)$ determined by the identity
\begin{equation}
\label{isotopy_extended}
h_{s+1}=h_s\circ h \qquad\qquad s\in[0,+\infty).
\end{equation}
{Since $\{D_s\}_{s\in\R/\Z}$ is a smooth foliation, the isotopy~\eqref{isotopy_extended} is smooth.} %We know from Lemma~\ref{lemma_smooth_foliation} that the disks $\{D_s\}_{s\in\R/\Z}$ form a smooth foliation of $\R/\Z\times\D$. It follows that~\eqref{isotopy_extended} is a smooth isotopy.
In particular, observe that, for every $j=0,\dots,L-1$, it holds
\begin{equation}\label{eq to explain Njlambda}
\wind_{s\in[0,L]}(\Gamma^p(s)-j*\Gamma^q(s))= \sum_{\lambda=1}^{N_j}\wind_{s\in[0,b^j_\lambda-a^j_\lambda]}(h_s(\zeta_\lambda)-h_s(\zeta'_\lambda)).
\end{equation}

	~\newline
	Recall that, from \eqref{choice_of_eps}, $\kappa(\gamma_0)-\varepsilon>2\pi$. Note that, for $x\in S^3\setminus\gamma_0$,
	$$
	\text{link}(k(T,x;D),\gamma_0)=\#\{ t\in[0,T] \ \vert\  \phi^t(x)\in \Psi(\{0\}\times\mathring{\mathbb{D}}) \}+\ell,
	$$
	for some $\ell\in \{-1,0,1\}$, where $\#\{ t\in[0,T] \ \vert\  \phi^t(x)\in \Psi(\{0\}\times\mathring{\mathbb{D}}) \}$ is the number of times that $\phi^{[0,T]}(x)$ meets $\Psi(\{0\}\times\mathring{\mathbb{D}})$. Observe that
	$$
\kappa(\gamma_0)=\liminf_{T\rightarrow+\infty}\,\inf_{\substack{x\in S^3\setminus\gamma_0 \\ u\in\xi_x\setminus 0}}\,\dfrac{\Delta\tilde{\Theta}_\sigma(T;x,u)}{\#\{ t\in[0,T] \ \vert\  \phi^t(x)\in \Psi(\{0\}\times\mathring{\mathbb{D}}) \}}.
	$$
%	Here we used that for every $x\in S^3\setminus\gamma_0$ we can estimate
%	\[\dfrac{T}{\tau_\max}\leq\#\{ t\in[0,T] \ \vert\  \phi^t(x)\in \Psi(\{0\}\times\mathring{\mathbb{D}})\}\leq \frac{T}{\tau_\min}\,,\]
%	and, as a consequence, $\#\{ t\in[0,T] |\ \phi^t(x)\in\Psi(\{0\}\times\mathring\D) \}$ grows linearly to $+\infty$ as $T\to+\infty$.
	 Thus, there exists $M=M(\varepsilon)>0$ such that for every $x\in S^3\setminus\gamma_0$ and every $u\in \xi_x\setminus 0$,
	\begin{equation}\label{how to use kappagamma}
	T\geq M \qquad \Rightarrow \qquad \dfrac{\Delta\tilde{\Theta}_{\sigma}(T; x, u)}{\#\{ t\in[0,T] \ \vert\  \phi^t(x)\in \Psi(\{0\}\times\mathring{\mathbb{D}}) \}}>\kappa(\gamma_0)-\varepsilon \, .
	\end{equation}
	
	Observe that, by the extension of the function $t$ in~\eqref{extension of that} and by $(ii)$ of Lemma~\ref{lemma_flipping_pages_time}, for every $z\in\mathbb{D}$, if $s\geq \frac{M}{c\tau_\min}$, then $t(s,z)\geq M$. 
	{According to Lemma~\ref{lemma_flipping_pages_time} the constant $c$ can be any number in $(0,1)$ fixed {\it a priori}, independent of $p$ and $q$.}
	Let us define
	\begin{equation}\label{def mathscrM}
	\mathscr{M}:=\sup_{\substack{z\in\mathbb{D} \\ u\in T_z\mathbb{D}\setminus 0}} \ \max_{1\leq i\leq \big\lfloor \frac{M}{c\tau_\min}\big\rfloor +1}\ \vert \wind_{s\in[0,i]}(Dh_s(z)u)\vert<+\infty \, .
	\end{equation}
	
	\begin{remark}\label{rmk 1}
	{	%We deduce that for 
		For every $\lambda$ such that $|I_\lambda^j|<\frac{M}{c\tau_\min}$, by Theorem~\ref{Florio} and by~\eqref{def mathscrM}, we deduce that
		$$
		\wind_{s\in [0,b^j_\lambda-a^j_\lambda]}(h_s(\zeta_\lambda)-h_s(\zeta_\lambda'))\geq -\mathscr{M}\,.
		$$}
	%	\noindent Indeed, by Theorem \ref{Florio} there exists $z_\lambda\in[\zeta'_\lambda,\zeta_\lambda]$ such that
	%	$$
	%	\wind_{s\in [0,b^j_\lambda-a^j_\lambda]}(h_s(\zeta_\lambda)-h_s(\zeta_\lambda'))=\wind_{s\in[0,b^j_\lambda-a^j_\lambda]}(Dh_s(z_\lambda)(\zeta_\lambda-\zeta'_\lambda))
	%	$$
	%	\noindent Thus, by \eqref{def mathscrM}, it holds $\wind_{s\in[0,b^j_\lambda-a^j_\lambda]}(Dh_s(z_\lambda)(\zeta_\lambda-\zeta'_\lambda))\geq -\mathscr{M}$.
	\end{remark}

\begin{lemma}\label{lemma_main_winding_estimate}
There exists a constant $C>0$ independent of $p,q,\{S_n\},\{T_n\}$ such that for each $\lambda \in \{1,\dots,N_j\}$ satisfying $|I^j_{\lambda}| = b_{\lambda}^j-a_{\lambda}^j\geq \frac{M}{c\tau_\min}$, the estimate
\begin{equation*}
\wind_{s\in [0,b^j_\lambda-a^j_\lambda]}(h_s(\zeta_\lambda)-h_s(\zeta'_\lambda)) \geq \left( \kappa(\gamma_0)-\varepsilon-2\pi \right)  \frac{|I^j_\lambda|}{2\pi} \ - \ C
\end{equation*}
holds.
\end{lemma}

\begin{proof}
We start by introducing some notation. Let $\tilde{X}_1,\tilde{X}_2$ be the nowhere vanishing vector fields given by $\sigma \circ \tilde{X}_j\equiv e_j$ where $e_1=(1,0)$, $e_2=(0,1)$. {We choose $\sigma$ so that $\tilde X_1$ is equal to the vector field $Z$ given by (e) in Proposition~\ref{prop_main}.} Observe that $i_{\tilde{X}_j}\lambda\equiv 0$. Denote by $\hat{X}_1,\hat{X}_2$ the pull-back of $\tilde{X}_1|_{S^3\setminus\gamma_0},\tilde{X}_2|_{S^3\setminus\gamma_0}$ respectively under the diffeomorphism $\Psi|_{\R/\Z\times\mathring{\mathbb{D}}}$. For every $s\in\R/\Z$ let $\Pi^s:T(\R/\Z\times\mathbb{D})|_{\{s\}\times\mathbb{D}}\rightarrow T\mathbb{D}$ be the vector bundle map determined by projecting along the direction of $W$, where $\{s\}\times\mathbb{D}$ is identified with $\mathbb{D}$. Then $X_1^s,X_2^s$ denote the images of $\hat{X}_1|_{\{s\}\times\mathring{\mathbb{D}}}$, $\hat{X}_2|_{\{s\}\times\mathring{\mathbb{D}}}$ under $\Pi^s$, respectively. We get two smooth families $X_1^s,X_2^s$ of smooth vector fields on $\mathring{\D}$ parametrized by $s\in\R/\Z$. At times it might be convenient to think of $s$ as variable in $\R$ and the families $X_1^s,X_2^s$ as $1$-periodic in $s$. By Theorem~\ref{Florio}, there exists $z_{\lambda}\in[\zeta'_{\lambda},\zeta_{\lambda}]$ such that
\begin{equation}\label{eq 1}
\text{wind}_{s\in[0,b^j_{\lambda}-a^j_{\lambda}]}(h_s(\zeta_{\lambda})-h_s(\zeta'_{\lambda}))=\text{wind}_{s\in[0,b^j_{\lambda}-a^j_{\lambda}]}(Dh_s(z_{\lambda})(\zeta_{\lambda}-\zeta'_{\lambda})).
\end{equation}
By definition we have
$$
\text{wind}_{s\in[0,b^j_{\lambda}-a^j_{\lambda}]}(Dh_s(z_{\lambda})(\zeta_{\lambda}-\zeta'_{\lambda})) = \dfrac{\theta(b^j_{\lambda}-a^j_{\lambda})-\theta(0)}{2\pi}
$$
where $\theta:[0,b^j_{\lambda}-a^j_{\lambda}]\rightarrow\R$ is a continuous argument of $Dh_s(z_{\lambda})(\zeta_{\lambda}-\zeta'_{\lambda})$. Consider now the smooth path $c:[0,b^j_{\lambda}-a^j_{\lambda}]\rightarrow\C\setminus\{0\}$
$$
s\mapsto c(s):= X^s_1(s,h_s(z_{\lambda}))
$$
and let $\varTheta:[0,b^j_{\lambda}-a^j_{\lambda}]\rightarrow\R$ be a continuous choice of argument of $c(s)$. Thus we can write $\theta(s)=\vartheta(s)+\varTheta(s)$, where $\vartheta:[0,b^j_{\lambda}-a^j_{\lambda}]\rightarrow\R$ is a unique choice of continuous angular coordinate of the vector $Dh_s(z_{\lambda})(\zeta_{\lambda}-\zeta'_{\lambda})$ in the frame $\{c(s),ic(s)\}$ determined by $\theta$ and $\varTheta$. Hence
\begin{equation}\label{eq 2}
\begin{aligned}
\dfrac{\theta(b^j_{\lambda}-a^j_{\lambda})-\theta(0)}{2\pi} &= \dfrac{\varTheta(b^j_{\lambda}-a^j_{\lambda})-\varTheta(0)}{2\pi} +\dfrac{\vartheta(b^j_{\lambda}-a^j_{\lambda})-\vartheta(0)}{2\pi} \\
&= \text{wind}_{s\in[0,b^j_{\lambda}-a^j_{\lambda}]}(c(s))+\dfrac{\vartheta(b^j_{\lambda}-a^j_{\lambda})-\vartheta(0)}{2\pi}
\end{aligned}
\end{equation}

\noindent \textit{(I) Lower bound for $\text{wind}_{s\in[0,b^j_{\lambda}-a^j_{\lambda}]}(c(s))$.} 
Consider a path $\rho:[0,1] \to \mathring{\D}$ of the form $\rho(0) = h_{b^j_{\lambda}-a^j_{\lambda}}(z_{\lambda})$ and $\rho(1) = z_{\lambda}$. 
Build the loop $\Gamma:[0,b^j_{\lambda}-a^j_{\lambda}+1]\rightarrow \R/\Z\times\mathring{\mathbb{D}}$
$$
\Gamma(s):=\begin{cases}
(s,h_s(z_{\lambda})) & s\in[0,b^j_{\lambda}-a^j_{\lambda}] \\
(0,\rho(s-b^j_{\lambda}+a^j_{\lambda})) & s\in[b^j_{\lambda}-a^j_{\lambda},b^j_{\lambda}-a^j_{\lambda}+1]
\end{cases}
$$ 
Observe that $\text{link}(\Psi\circ\Gamma,\gamma_0) = |I^j_{\lambda}| = b^j_{\lambda}-a^j_{\lambda}$. We associate to $\Gamma$ the path $$ D\Gamma:[0,b^j_{\lambda}-a^j_{\lambda}+1] \rightarrow \C\setminus\{0\} $$ defined by 
$$
D\Gamma(s) := 
\begin{cases}
X_1^{s}(\Gamma(s))\quad s\in[0,b^j_{\lambda}-a^j_{\lambda}] \\
X_1^0(\Gamma(s))\quad s\in[b^j_{\lambda}-a^j_{\lambda},b^j_{\lambda}-a^j_{\lambda}+1].
\end{cases}
$$ 
Then it holds that
$$
\text{wind}_{s\in[0,b^j_{\lambda}-a^j_{\lambda}+1]}(D\Gamma(s))=\text{wind}_{s\in[0,b^j_{\lambda}-a^j_{\lambda}]}(c(s))+\text{wind}_{s\in[0,1]}(X_1^0(\rho(s))).
$$
Since $\tilde X_1$ satisfies (e) of Proposition \ref{prop_main} we deduce that there exists a constant $C_1>0$, independent from $p,q,\{T_n\},\{S_n\}$, such that
\begin{equation}\label{first X1s}
| \text{wind}_{s\in[0,1]}(X_1^0(\rho(s))) | \leq C_1.
\end{equation}

Brouwer's translation theorem gives a fixed point $\Psi(0,\bar{z})$ of the first return map to $\Psi(\{0\}\times\mathring{\mathbb{D}})$ corresponding to a periodic Reeb orbit $\bar\gamma \subset S^3\setminus\gamma_0$. We denote by $\bar\gamma^{|I^j_{\lambda}|}$ the $|I^j_{\lambda}|$-fold iteration of $\bar\gamma$. The loop $\beta := \Psi^{-1}(\bar\gamma^{|I^j_{\lambda}|})$ is homotopic to $\Gamma$ in $\R/\Z\times\mathring\D$ since they have the same linking number with $\gamma_0$. The loop $\beta$ can be parametrised by $s\in [0,b^j_{\lambda}-a^j_{\lambda}]$ so that $\beta(s) = (s,\hat\beta(s-\lfloor s\rfloor))$ where $\hat\beta:[0,1] \to \mathring{\D}$ satisfies $\hat\beta(1)=\hat\beta(0)$. Then
		\begin{equation}\label{eq self-link}
		\begin{aligned}
		\text{wind}_{s\in[0,b^j_{\lambda}-a^j_{\lambda}+1]}(D\Gamma(s)) &= \wind_{s\in[0,b^j_{\lambda}-a^j_{\lambda}]}(X_1^{s}(\hat\beta(s-\lfloor s\rfloor))) \\
		&= (b^j_{\lambda}-a^j_{\lambda})\ \text{wind}_{s\in[0,1]}(X_1^s(\hat\beta(s))\,.
		\end{aligned}
		\end{equation}
		The quantity $\wind_{s\in[0,1]}(X_1^s(\hat\beta(s)))$ is equal to the self-linking number of $\bar\gamma$, see Definition~1.5 in~\cite{openbook}. By~\cite[Theorem~1.8]{openbook}, the self-linking number of $\bar\gamma$ is $-1$. Thus, from~\eqref{eq self-link} we get
		\begin{equation}\label{second X1s}
		\text{wind}_{s\in[0,b^j_{\lambda}-a^j_{\lambda}+1]}(D\Gamma(s))=-(b^j_{\lambda}-a^j_{\lambda}).
		\end{equation}
		From \eqref{first X1s} and \eqref{second X1s} we deduce that
		\begin{equation}\label{ineq 2}
		\text{wind}_{s\in[0,b^j_{\lambda}-a^j_{\lambda}]}(c(s)) \geq -(b^j_{\lambda}-a^j_{\lambda})-C_1 = -|I^j_{\lambda}| - C_1
		\end{equation}

		\noindent \textit{(II) Lower bound for $\frac{\vartheta(b^j_{\lambda}-a^j_{\lambda})-\vartheta(0)}{2\pi}$.} With respect to the frame $\{c(s),ic(s)\}$ the angle coordinate of the vector $Dh_s(z_{\lambda})(\zeta_{\lambda}-\zeta'_{\lambda})$ is $\vartheta(s)$. We are interested in the variation of the angle coordinate with respect to the frame
		\begin{equation}\label{frame_on_disk}
		\{c(s) = X_1^{s}(s,h_s(z_{\lambda})),X_2^{s}(s,h_s(z_{\lambda}))\}
		\end{equation}
		obtained from $\sigma$. 
		%A priori, it determines a different Riemannian metric on the disk. 
		If we denote by $\tilde{\vartheta}(s)$ the angle coordinate of $Dh_s(z_{\lambda})(\zeta_{\lambda}-\zeta'_{\lambda})$ with respect to the frame~\eqref{frame_on_disk}, we can invoke~\cite[Claim~1.1.1]{FloPHD} to conclude that%{\footnote{Comparing $\vartheta$ and $\tilde\vartheta$ corresponds to compare oriented angles with respect to different Riemannian metrics.}}
\begin{equation}
		\dfrac{\vartheta(b^j_{\lambda}-a^j_{\lambda})-\vartheta(0)}{2\pi} > \dfrac{\tilde{\vartheta}(b^j_{\lambda}-a^j_{\lambda})-\tilde{\vartheta}(0)}{2\pi}-1.
		\end{equation}
		Comparing $\vartheta$ and $\tilde\vartheta$ corresponds to compare oriented angles with respect to different Riemannian metrics.
		%Let $D\phi_{\xi}^t$ denote such linearised flow. 
		Let $v\in\xi_{\Psi(0,z_{\lambda})}$ be uniquely characterised by $$ D\Psi(0,z_\lambda) (0,\zeta_{\lambda}-\zeta'_{\lambda}) \in v + \R X\,. $$ 
		
		From Lemma~\ref{lemma_flipping_pages_time}, the strictly increasing function $ t(\cdot, z_\lambda):[0,b^j_{\lambda}-a^j_{\lambda}]\to[0,+\infty)$ satisfies $$ t(0, z_\lambda)=0 \qquad\qquad \Psi(s,h_s(z_\lambda)) = \phi^{ t(s, z_\lambda)}(\Psi(0,z_\lambda)), $$
		with
		\begin{equation}
		t(s, z_\lambda) = \sum_{k=0}^{\lfloor s \rfloor - 1} t(1,h^k(z_\lambda)) + t(s-\lfloor s \rfloor,h^{\lfloor s \rfloor}(z_\lambda)\,.
		\end{equation}
See also \eqref{extension of that}. The coordinates of the vector $Dh_s(z_{\lambda})(\zeta_{\lambda}-\zeta'_{\lambda})$ in the basis~\eqref{frame_on_disk} are the same of $$ D\phi^{ t(s, z_\lambda)}(\Psi(0,z_{\lambda})) \cdot v $$ in the basis $\{\tilde{X}_1,\tilde{X}_2\}$. Recall 
		%the generating vector field ${\color{red} X_\PP}$ \marginpar{Why do we recall $X_\PP$?} of the flow on $(\xi\setminus0)/\R_+$ induced by the linearised flow of $\phi^t$ on $\xi$, and 
		the $\R/2\pi\Z$ coordinate $\Theta_\sigma$ in $(\xi\setminus0)/\R_+$ induced by the frame $\sigma$.  If $p\in S^3$, $w\in \xi_p\setminus 0$, we denote by $t \mapsto \tilde\Theta_\sigma(t,p,w)$ a continuous lift to~$\R$ of $t\mapsto \Theta_\sigma(D\phi^t(p)w)$. Thus
		\begin{equation}
		\begin{aligned}
		& \tilde{\vartheta}(b^j_{\lambda}-a^j_{\lambda})-\tilde{\vartheta}(0) \\
		& \qquad = \tilde{\Theta}_\sigma(t(b_\lambda^j-a_\lambda^j, z_\lambda), \Psi(0,z_\lambda),v)-\tilde{\Theta}_\sigma(0, \Psi(0,z_\lambda),v) \\
		& \qquad =:\Delta\Theta_\sigma(t(b_\lambda^j-a_\lambda^j, z_\lambda), \Psi(0,z_\lambda),v).
		\end{aligned}
		\end{equation}
		\noindent Since, by assumption, $b_\lambda^j-a_\lambda^j\geq \frac{M}{c\tau_\min}$, then $t(b^j_\lambda-a^j_\lambda, z_\lambda)\geq M$ and, from \eqref{how to use kappagamma}, it holds that
		\begin{equation}
		\begin{aligned}
		& \Delta\Theta_\sigma(t(b_\lambda^j-a_\lambda^j, z_\lambda), \Psi(0,z_\lambda),v) \\
		&>(\kappa(\gamma_0)-\varepsilon)\ \#\{t\in[0,t(b_\lambda^j-a_\lambda^j, z_\lambda)] \ \vert\  \phi^t(\Psi(0,z_\lambda))\in \Psi(\{0\}\times\mathring{\mathbb{D}})\} \\ &=(\kappa(\gamma_0)-\varepsilon)\,(|I_\lambda^j|+1).
		\end{aligned}
		\end{equation}
		%	\begin{equation}\label{main ineq}
		%	\begin{aligned}
		%	\tilde{\vartheta}(b^j_{\lambda}-a^j_{\lambda})-\tilde{\vartheta}(0) 
		%	&= \int_0^{b^j_{\lambda}-a^j_{\lambda}} \frac{d}{ds} \ \Theta_\sigma\left( \R_+ D\phi^{\hat t(s)}(\Psi(0,z_\lambda)) \cdot v \right) \ ds \\ 
		%	&= \int_0^{b^j_{\lambda}-a^j_{\lambda}} \left( \left. \frac{d}{dt} \right|_{t=\hat t(s)}\Theta_\sigma \left( \R_+ D\phi^t(\Psi(0,z_\lambda)) \cdot v \right) \right) \hat t'(s) \ ds \\
		%	&\geq \kappa_{\sigma} \, c\tau_{\min} \, |I^i_{\lambda}|
		%	\end{aligned}
		%	\end{equation}
		Consequently
		\begin{equation}\label{ineq 3}
		\dfrac{\vartheta(b^j_{\lambda}-a^j_{\lambda})-\vartheta(0)}{2\pi}>\dfrac{\left( \kappa(\gamma_0)-\varepsilon\right) \, \left(|I^j_{\lambda}|+1\right)}{2\pi}-1
		\end{equation}
		
		With the help of (I) and (II) we can conclude the proof, since from \eqref{eq 1}, \eqref{eq 2}, \eqref{ineq 2}, \eqref{ineq 3} we have 
		$$
		\text{wind}_{s\in[0,b^j_{\lambda}-a^j_{\lambda}]}(h_s(\zeta_{\lambda})-h_s(\zeta'_{\lambda})) >\Big(\kappa(\gamma_0)-\varepsilon-2\pi\Big)\ \dfrac{|I^j_\lambda|}{2\pi}-C,
		$$
		where $C:=C_1+1-\frac{(\kappa(\gamma_0)-\varepsilon)}{2\pi}$.
	\end{proof}

We now give the final estimate to complete the proof of Proposition~\ref{prop_main_main}.
For every $j=0,\dots,L-1$ denote $$ \mathscr{G}_j=\left\{\lambda\in\{1,\dots ,N_j\} \ \vert\  |I_\lambda^j| \geq \frac{M}{c\tau_\min}\right\}, $$
$$ \mathscr{B}_j=\left\{\lambda\in\{1,\dots ,N_j\} \ \vert\  |I_\lambda^j|< \frac{M}{c\tau_\min}\right\} \, . $$
With Lemma~\ref{lemma_main_winding_estimate} and Remark~\ref{rmk 1} we can estimate~\eqref{broken_winding} as (see \eqref{eq to explain Njlambda})
\begin{equation*}
\sum_{i,j=0}^{L-1} \wind_{s\in[0,1]}(k^p_i(s)-k^q_j(s)) = 
\end{equation*}
\begin{equation*}
=\sum_{j=0}^{L-1}\left( \sum_{\lambda\in\mathscr{G}_j}\wind_{s\in [0,b^j_\lambda-a^j_\lambda]}(h_s(\zeta_\lambda)-h_s(\zeta_\lambda'))+\sum_{\lambda\in\mathscr{B}_j}\wind_{s\in [0,b^j_\lambda-a^j_\lambda]}(h_s(\zeta_\lambda)-h_s(\zeta_\lambda')) \right)
\end{equation*}
\begin{equation*}
> \sum_{j=0}^{L-1} \left( \sum_{\lambda\in\mathscr{G}_j}\left[ \Big(\kappa(\gamma_0)-\varepsilon-2\pi\Big)\dfrac{|I_\lambda^j|}{2\pi}-C \right]-\sum_{\lambda\in\mathscr{B}_j}\mathscr{M}\right) 
\end{equation*}
and, denoting $\tilde{C}:=\max(C,\mathscr{M})$, we have
\begin{equation*}
\sum_{i,j=0}^{L-1} \wind_{s\in[0,1]}(k^p_i(s)-k^q_j(s))\geq\sum_{j=0}^{L-1}\left( \dfrac{\kappa(\gamma_0)-\varepsilon-2\pi}{2\pi}\sum_{\lambda\in\mathscr{G}_j}|I_\lambda^j|-\sum_{\lambda=1}^{N_j}\tilde{C} \right)
\end{equation*}
\begin{equation*}
\geq\sum_{j=0}^{L-1}\left( \dfrac{\kappa(\gamma_0)-\varepsilon-2\pi}{2\pi}\sum_{\lambda\in\mathscr{G}_j}|I_\lambda^j| \right)-\tilde{C}L(r_n(p)+r_n(q))\,.
\end{equation*}
For every $\lambda\in\mathscr{B}_j$ recall that $|I_\lambda^j|<\frac{M}{c\tau_\min}$. Observe now that
	\begin{equation*}
	\begin{aligned}
	0\leq \sum_{\lambda\in\mathscr{B}_j}|I_\lambda^j|<\dfrac{M}{c\tau_\min}N^j\leq \dfrac{M}{c\tau_\min}(r_n(p)+r_n(q))=L\dfrac{M}{c\tau_\min}\,\dfrac{m_n(p)+m_n(q)}{m_n(q)m_n(p)}.
	\end{aligned}
	\end{equation*}
	\noindent Moreover, it holds that $\sum_{\lambda\in\mathscr{G}_j}|I_{\lambda}^j|+\sum_{\lambda\in\mathscr{B}_j}|I_\lambda^j|=L$. Consequently
	\begin{equation}
	\sum_{\lambda\in\mathscr{G}_j}|I_\lambda^j|>\left( 1-\dfrac{M}{c\tau_\min}\dfrac{m_n(p)+m_n(q)}{m_n(q)m_n(p)} \right)L.
	\end{equation}
	\noindent Therefore, we have
	$$
	\begin{aligned}
	& \sum_{i,j=0}^{L-1}\wind_{s\in[0,1]}(k^p_i(s)-k^q_j(s)) \\
	& > \dfrac{\kappa(\gamma_0)-\varepsilon-2\pi}{2\pi}\left( 1-\dfrac{M}{c\tau_\min}\dfrac{m_n(p)+m_n(q)}{m_n(q)m_n(p)} \right)L^2-\tilde{C}L(r_n(p)+r_n(q)),
	\end{aligned}
	$$
	which together with~\eqref{broken_winding_pre} yields
	\begin{equation*}
	\begin{aligned}
	& \link(k(T_n,p)^{r_n(p)}, k(S_n,q)^{r_n(q)}) \\
	& > \frac{L^2}{2\pi} \left( \kappa(\gamma_0)-\varepsilon-2\pi \right)\left( 1-\dfrac{M}{c\tau_\min}\dfrac{m_n(p)+m_n(q)}{m_n(q)m_n(p)} \right) - { L(1+\tilde{C}) (r_n(p)+r_n(q))} \, .
	\end{aligned}
	\end{equation*}
		Combining with~\eqref{estimates_closing_times_p}-\eqref{estimates_closing_times_q} we get
		\begin{equation}\label{crucial_ineq_rotation_numbers}
		\begin{aligned}
		& \frac{\link(k(T_n,p), k(S_n,q))}{T_nS_n} = \frac{\link(k(T_n,p)^{r_n(p)}, k(S_n,q)^{r_n(q)})}{r_n(p)T_n \ r_n(q)S_n} \\
		&>\dfrac{(\kappa(\gamma_0)-\varepsilon-2\pi)}{2\pi}\,\dfrac{m_n(p)m_n(q)}{T_nS_n} \\
		& \qquad -\dfrac{M}{c\tau_\min}\,\dfrac{(\kappa(\gamma_0)-\varepsilon-2\pi)}{2\pi}\,\dfrac{m_n(p)+m_n(q)}{T_nS_n} \\
		& \qquad -{L(1+\tilde C)}\dfrac{r_n(p)+r_n(q)}{L^2}\,\dfrac{m_n(q)m_n(p)}{T_nS_n}\\
		&\geq \dfrac{(\kappa(\gamma_0)-\varepsilon-2\pi)}{2\pi\tau_\max^2}\,\dfrac{m_n(p)m_n(q)}{(m_n(p)+1)(m_n(q)+1)} \\ 
		& \qquad -\dfrac{M}{c\tau_\min}\dfrac{(\kappa(\gamma_0)-\varepsilon-2\pi)}{2\pi\tau_\min^2}\,\Big( \dfrac{1}{m_n(q)}+\dfrac{1}{m_n(p)} \Big) \\
		& \qquad -\dfrac{{L(1+\tilde C)}}{\tau_\min^2}\,\dfrac{1}{m_n(p)m_n(q)}\Big( \dfrac{1}{r_n(q)}+\dfrac{1}{r_n(p)} \Big)
		%&> \dfrac{\frac{L^2}{2\pi}(\kappa(\gamma_0)-\varepsilon-2\pi)-\left(\frac{M}{c\tau_\min}\frac{L^2}{2\pi}\frac{m_n(p)+m_n(q)}{m_n(p)m_n(q)}+(1+\tilde{C}L)(r_n(p)+r_n(q))\right)}{L^2}\dfrac{m_n(p)m_n(q)}{T_nS_n}\\
		%&\geq \dfrac{\frac{1}{2\pi}(K_{\gamma_0}-\varepsilon-2\pi)-(\frac{M}{c\tau_\min}+1+\tilde{C})\left( \frac{1}{m_n(p)}+\frac{1}{m_n(q)} \right)}{\tau_\max^2}\dfrac{m_n(p)m_n(q)}{(m_n(p)+1)(m_n(q)+1)}\\
		%& \geq \dfrac{m_n(p)m_n(q)}{T_nS_n}\left( \dfrac{(K_{\gamma_0}-\varepsilon-2\pi)}{2\pi}\, \dfrac{m_n(q)m_n(p)-\frac{M}{c\tau_\min}(m_n(q)+m_n(p))}{m_n(q)m_n(p)}-\left(\tilde{C}+\frac{1}{L}\right) \dfrac{r_n(q)+r_n(p)}{L} \right) \\
		%& \geq \dfrac{K_{\gamma_0}-\varepsilon-2\pi}{2\pi}\, \dfrac{m_n(q)m_n(p)}{T_nS_n}-\left( \tilde{C}+\frac{1}{L} +\dfrac{M}{2\pi c\tau_\min}(K_{\gamma-0}-\varepsilon-2\pi)\right)\dfrac{m_n(q)+m_n(p)}{T_nS_n} \\
		%& \geq \dfrac{K_{\gamma_0}-\varepsilon-2\pi}{2\pi}\dfrac{\left( 1-\frac{\tau_{\max}}{T_n} \right)\left( 1-\frac{\tau_{\max}}{S_n} \right)}{\tau_{\max}^2}-\left( \tilde{C}+\frac{1}{L} +\dfrac{M}{2\pi c\tau_\min}(K_{\gamma_0}-\varepsilon-2\pi)\right)\left( \dfrac{1}{S_n}+\dfrac{1}{T_n} \right)
\end{aligned}
\end{equation}
from where it follows that
\begin{equation}
\liminf_{n\to\infty} \ \frac{\link(k(T_n,p), k(S_n,q))}{T_nS_n} \geq \frac{\kappa(\gamma_0)-\varepsilon-2\pi}{2\pi\tau_\max^2} > 0
\end{equation}
Here we made use of~\eqref{choice_of_eps}. 
The proof of Proposition~\ref{prop_main_main} is complete.

{ The argument above has a simple consequence.

\begin{lemma}\label{lemma_kappa_Seifert_rotation_numbers}
If $\kappa(\gamma_0)>2\pi$ then every periodic orbit has strictly positive transverse rotation number in a Seifert framing.
\end{lemma}

\begin{proof}
We draw freely from the notation established in the proof of Proposition~\ref{prop_main_main}.
In~\eqref{crucial_ineq_rotation_numbers} consider here the case where the point $p \in \Psi(\{0\}\times\D)$ belongs to a periodic orbit $\gamma \subset S^3 \setminus \gamma_0$ with primitive period $T$.
Let $q$ be a recurrent point in $\Psi(\{0\}\times\D)\setminus\{p\}$. 
If $q$ is closed enough to $p$ then $q \not\in\gamma$.
It follows from~\eqref{crucial_ineq_rotation_numbers} that if $T_n = nT$ and $S_n$ are larger than some constant $\bar T>0$, which is fixed large enough independently of $q$, then $\link(\gamma^n,k(S_n,q)) \geq nTS_n(\kappa(\gamma_0)-\varepsilon-2\pi)/(2\pi\tau_{\max}^2)$.
The existence of~$\bar T$ relies crucially on the fact that the constants $M,c,\tilde C,C,\mathscr{M}$ appearing in the previous proof do not depend on~$p$ and~$q$.
Dividing by $n$ we get $$ \link(\gamma,k(S_n,q)) \geq TS_n \frac{\kappa(\gamma_0)-\varepsilon-2\pi}{2\pi\tau_{\max}^2} > 0 \, . $$
Letting $q$ converge to $p$ arbitrarily along recurrent points in $\Psi(\{0\}\times\D)$, and comparing the trajectory $\phi^{[0,S_n]}(q)$ with the linearised flow along $p$ over the {\bf fixed} interval $[0,S_n]$ independent of $q$, we conclude that the transverse rotation number of $\gamma$ on a Seifert framing is positive.
\end{proof}

}

\bigskip

By Proposition~\ref{prop_main_main} and Lemma~\ref{lemma_kappa_Seifert_rotation_numbers} the assumption $\kappa(\gamma_0) > 2\pi$ implies right-handedness of the Reeb flow. The proof of Theorem~\ref{thm_main_0} is complete.

\bigskip

\subsection{Right-handedness on strictly convex energy levels}

Let the contact form $\lambda$ on $S^3$ be dynamically convex. The contact structure and the Reeb flow are $\xi$ and $\phi^t$, respectively. We begin with studying criteria to estimate the invariant $\kappa$ in~\eqref{def_number_kappa}. Consider
\begin{equation}
K_\sigma:=\inf_{\PP_+\xi} \, i_{ X_\PP}d\Theta_\sigma\,
\end{equation}
where $\sigma$ is the global frame, $\Theta_\sigma : \PP_+\xi \to \R/2\pi\Z$ is the global angle coordinate induced by $\sigma$ as in~\eqref{angular_coord_sigma}, and ${X_\PP}$ is the vector field generating the linearised Reeb flow on $\PP_+\xi$.
Let $\gamma_0$ be an unknotted periodic orbit with self-linking number $-1$, and let $D\subset S^3$ be the disk-like global surface of section given by Theorem 1.7 in \cite{openbook}. Denote $0<\tau_\min(D)\leq\tau_\max(D)<+\infty$ the infimum and the supremum of the return time function on $D$.

\begin{theorem}
\label{cor Ksigma}
If $K_\sigma \, \tau_\min(D)>2\pi$ then the Reeb flow of $\lambda$ is right-handed.
\end{theorem}

\begin{proof}
%TO BE REVISED.
We show that the hypothesis implies $\kappa(\gamma_0)>2\pi$ and then apply Theorem~\ref{thm_main_0}. %Fix $\varepsilon>0$ such that $K_\sigma \, \tau_\min(D)-\varepsilon>2\pi$. 
By the fundamental theorem of calculus, for every $x \in S^3 \setminus \gamma_0$ and $u \in \xi|_x$, $u\neq0$, we have for $T>0$
\begin{equation}
\label{cor eq 1}
\begin{aligned}
\tilde{\Theta}_\sigma(T,u) - \tilde\Theta_\sigma(0,u) &= \int_0^T\dfrac{d}{dt}\Theta_\sigma(\R_+D\phi^t(x) \, u)dt \\ 
&\geq \ T \inf_{\substack{x\in S^3\setminus\gamma_0 \\ u\in \xi_x\setminus 0}}i_{ X_\PP}d\Theta_\sigma(x,u)=TK_\sigma\, .
\end{aligned}
\end{equation}
Moreover
\begin{equation}
\label{cor eq 2}
\text{link}(k(T,x;D),\gamma_0)=  
\#\{ t\in[t_-^D(x),T+t^D_+(\phi^T(x))] \mid  \phi^t(x)\in D \}-1 \,.
\end{equation}
Observe that
\begin{equation}
\label{cor eq 3}
\#\{ t\in[t_-^D(x),T+t^D_+(\phi^T(x))] \mid \phi^t(x)\in D \}-1\leq \dfrac{T+t_+^D(\phi^T(x))-t_-^D(x)}{\tau_\min(D)}\, .
\end{equation}
Hence $$ \dfrac{ \tilde{\Theta}_\sigma(T,u)-\tilde\Theta_\sigma(0,u)}{\text{link}(k(T,x;D),\gamma_0)} \geq \dfrac{TK_\sigma }{\frac{T+t_+^D(\phi^T(x))-t_-^D(x)}{\tau_\min(D)}} = K_\sigma \, \tau_{\min}(D) \, \frac{T}{T+t_+^D(\phi^T(x))-t_-^D(x)} \, . $$ Since $0<t_+^D(\phi^T(x))-t_-^D(x)\leq 2\tau_{\max}(D)$ for every $x \in S^3 \setminus \gamma_0$, we finally get $$ \dfrac{ \tilde{\Theta}_\sigma(T,u)-\tilde\Theta_\sigma(0,u)}{\text{link}(k(T,x;D),\gamma_0)} \geq K_\sigma \, \tau_{\min}(D) \, \frac{T}{T+2\tau_{\max}(D)} \, . $$ It follows immediately by first taking the infimum on $x,u$ and then taking the limit as $T\to+\infty$ that $\kappa(\gamma_0) \geq K_\sigma \, \tau_{\min}(D)>2\pi$. 
%Now apply Theorem~\ref{thm_main_0}.
\end{proof}

Our final task in this section is to prove Theorem~\ref{thm_main_1}. Let $J_0:\R^4\rightarrow\R^4$ be the complex structure defined by the matrix 
\begin{equation}\label{def J0}
J_0=\begin{pmatrix}
0 & 0 & -1 & 0 \\
0 & 0 & 0 & -1 \\
1 & 0 & 0 & 0 \\
0 & 1 & 0 & 0
\end{pmatrix}
\end{equation}
with respect to coordinates $(q_1,q_2,p_1,p_2)$ of $\R^4$. Then $\omega_0(u,v)=\langle u,J_0v\rangle$ holds for every $u,v\in\R^4$, where $\langle\cdot,\cdot\rangle$ is the standard Euclidean inner product. We will denote as $\|\cdot\|$ the inherited norm. As explained in the introduction, the Hamiltonian vector field of $H=\nu^2_C$, defined as $i_{X_H}\omega_0=-dH$,  is the Reeb vector field of the contact form on $\partial C$ induced by $\lambda_0$. The associated contact structure is denoted by $\xi \subset T\partial C$. Denote as $\varphi_H^t$ the Hamiltonian flow on $\partial C$, and ${ X_\PP}$ the vector field on $\PP_+\xi$ generating the linearised flow. Together with the matrix $J_0$, consider the following
$$
J_1=\begin{pmatrix}
0 & -1 & 0 & 0 \\1 & 0 & 0 & 0 \\ 0 & 0 & 0 &1 \\ 0 & 0 &-1 &0
\end{pmatrix}\qquad
J_2=\begin{pmatrix}
0 & 0 & 0 & -1 \\ 0 & 0 & 1 & 0 \\ 0 & -1 & 0 & 0 \\ 1 & 0 & 0 & 0
\end{pmatrix}.
$$
Then $J_i^2=-I$ and $J_i^T=-J_i$ hold for all $i=0,1,2$ where $I$ denotes the identity matrix. Moreover, $J_0J_1=J_2$, $J_1J_2=J_0$ and $J_2J_0=J_1$. On $\partial C$ consider the following vector fields 
\begin{equation}\label{def Salomao}
X_0=\dfrac{\nabla H}{\Vert\nabla H\Vert},\quad X_1=J_2X_0,\quad X_2=J_1X_0,\quad X_3=-J_0X_0.
\end{equation}
\noindent The frame $\{X_1,X_2,X_3\}$ is an orthonormal global frame of $T\partial C$. Observe that the vector field $X_3$ is positively collinear with $X_H$ and so it is tangent to the trajectories of the Reeb flow on $\partial C$. Let $\sigma$ the global symplectic trivialisation of $(\xi,d\lambda_0)$ determined by projecting $X_1,X_2$ onto $\xi$ along the Reeb direction $\R X_3$. As explained in the introduction, it induces a trivialisation $\PP_+\xi \simeq \partial C\times \R/2\pi\Z$ of circle bundles. Denote by $\Theta_\sigma$ the $\R/2\pi\Z$-component of this map. %A crucial ingredient is the following statement.
{Theorem~\ref{thm_main_1} follows immediately from Theorem~\ref{cor Ksigma} and the following statement.}

\begin{lemma}[\cite{ragazzo-salomao}]
\label{Salomao}
The inequality $i_{ X_\PP}d\Theta_\sigma \geq 2K^C_\min$ holds everywhere on $\PP_+\xi$.
\end{lemma}

\section{Pinched two-spheres}\label{sec_proof_pinched}

The aim of this section is to prove Theorem~\ref{thm_pinched}. The strategy is, first, to lift the geodesic flow on the unit tangent bundle of a Riemannian two-sphere $(S^2,g)$ to the Reeb flow of a contact form $f_g\lambda_0$ on $S^3$ for some smooth $f_g:S^3\to(0,+\infty)$, where $\lambda_0$ is the restriction to $S^3\subset \C^2$ of the standard Liouville form $\frac{1}{4i}(\bar zdz-zd\bar z+\bar wdw-wd\bar w)$. Here $(z,w)$ denote the complex coordinates on $\C^2$. The second  step is to use the pinching condition on the curvatures to show that the assumptions of Theorem~\ref{thm_main_0} are fulfilled.

\subsection{Lifting geodesic flows on $S^2$ and Birkhoff annuli}

Here we establish notation, and recall well-known facts about geodesic flows. Let $(S^2,g)$ be an oriented Riemannian two-sphere. The foot-point projection of its tangent bundle is denoted by $\pi:TS^2\to S^2$. Consider the unit tangent bundle $T^1_gS^2$ defined as the set of $(x,v) \in TS^2$ such that $g(x)(v,v)=1$.

There are  two important vector bundle maps covering $\pi$, namely, the differential of the projection $d\pi : TTS^2 \to TS^2$ and the connection map $K : TTS^2 \to TS^2$. %The connection map $K$ can be  written in local coordinates $x=(x^1,x^2)$ as $$ K(x,v,\delta x,\delta v) = (x,\delta v+\Gamma(x)(v,\delta x)) $$ where $\Gamma(x)(u,w) = \sum_{i,j,k} \Gamma^k_{ij}(x)u^iw^j\partial_{x^k}$ and $\Gamma^k_{ij}$ are the Christoffel symbols.
 There is a splitting
\begin{equation*}
TTS^2 = \mathcal{H} \oplus \mathcal{V} \qquad\qquad \mathcal{H} = \ker K, \ \mathcal{V} = \ker d\pi
\end{equation*}
where both $\mathcal{H}, \mathcal{V}$ are fiberwise two-dimensional. {For all $v,w\in T_xS^2$, the horizontal lift $w^\hor\in \mathcal{H}_{(x,v)}$ and the vertical lift $w^\ver\in \mathcal{V}_{(x,v)}$ are both well-defined: we refer to \cite{HP,finsler} for further details.} %It follows that $d\pi|_{\mathcal{H}}$ defines an isomorphism fiberwise. More precisely for all $v,w\in T_xS^2$ there exists a unique $w^\hor \in \mathcal{H}_{(x,v)}$ satisfying $d\pi_{(x,v)} \cdot w^\hor = w$, called the horizontal lift of $w$. The inverse mapping $w \mapsto w^\hor$ defines a linear isomorphism $T_xS^2 \to \mathcal{H}_{(x,v)}$. The vertical lift of $w$ is defined as $w^\ver = \frac{d}{dt}|_{t=0}(v+tw)$. The map $w \mapsto w^\ver$ is a linear isomorphism $T_xS^2 \to \mathcal{V}_{(x,v)}$.
%The following facts are well-known, see~\cite{finsler} for details. 
The so-called Hilbert form $\sum_{ij}g_{ij}(x)v^idx^j$ restricts to a contact form $\lambda_g$ on $T^1_gS^2$ whose Reeb flow coincides with the geodesic flow. For every $(x,v) \in T^1_gS^2$ denote by $v^\bot \in T_xS^2$ the unique vector such that $g(x)(v^\bot,v^\bot)=1$, $g(x)(v^\bot,v)=0$, and $\{v,v^\bot\}$ is a positive basis. It follows that $\{v^\hor,(v^\bot)^\ver,(v^\bot)^\hor\}$ defines a basis of $T_{(x,v)}T^1_gS^2$, where $v^\hor$ is the Reeb vector field at $v$ and $\{(v^\bot)^\ver,(v^\bot)^\hor\}$ is a $d\lambda_g$-symplectic frame of the contact structure $\ker \lambda_g$. In fact, we see from this that $\ker\lambda_g$ is trivial as a (symplectic) vector bundle since it admits a global frame
\begin{equation}\label{global_geodesic_frame}
\sigma_g : v \mapsto \{(v^\bot)^\ver,(v^\bot)^\hor\}
\end{equation}
which we call the \textit{geodesic frame}. {If $h$ is another Riemannian metric on $S^2$, then there exists a contactomorphism between $T^1_gS^2$ and $T^1_hS^2$.}
%The Legendre transform $\mathcal{L}_g:TS^2 \to T^*S^2$ is the vector bundle map defined by $(x,v) \mapsto (x,p = g(x)(v,\cdot))$. Let $\tilde\pi:T^*S^2 \to S^2$ be the foot-point projection. The Hilbert form is pushed forward by $\mathcal{L}_g$ to the tautological $1$-form on $T^*S^2$ defined by $\zeta \in T_{(x,p)}T^*S^2 \mapsto p(d\tilde\pi_{(x,p)} \cdot\zeta) \in\R$. Let $F_g:T^*S^2 \to [0,+\infty)$ be defined by $F_g(x,p) = \sqrt{g(x)(v,v)}$ where $\mathcal{L}_g(x,v)=(x,p)$. Then $F_g$ restricts to a norm on the fibers of $T^*S^2$, and $\mathcal{L}_g$ maps $T^1_gS^2$ to $F_g^{-1}(1)$. Denote by $\Pi_g$ be the projection to $F_g^{-1}(1)$ defined by $p \mapsto p/F_g(x,p)$ on the complement of the zero section of $T^*S^2$. If $h$ is another Riemannian metric on $S^2$ then the map 
%\begin{equation}\label{contactomorphism_between_unit}
%\mathcal{L}_h^{-1} \circ \Pi_h \circ \mathcal{L}_g : T^1_gS^2 \to T^1_hS^2
%\end{equation}
%is diffeomorphism sending $\ker\lambda_g$ to $\ker\lambda_h$, i.e. a contactomorphism. If $g=h$ then this is the identity map.
%We identify $S^3 \simeq SU(2)$ by $$ (z,w) \simeq \begin{pmatrix} z & w \\ -\bar w & \bar z \end{pmatrix}, \qquad\qquad |z|^2+|w|^2=1, $$ and 
%\begin{equation}\label{two_sphere_normal_form}
%S^2 \simeq \left\{ \begin{pmatrix} z & w \\ -\bar w & \bar z \end{pmatrix} \in SU(2), \ \Re(z)=0 \right\} \subset S^3.
%\end{equation}
%Consider $j,k \in SU(2)$ defined by $$ j = \begin{pmatrix} 0 & 1 \\ -1 & 0 \end{pmatrix}, \qquad k = \begin{pmatrix} 0 & i \\ i & 0 \end{pmatrix}. $$ 
Denote by $g_0$ the round metric on $S^2$. %, given by $g_0=dy^2+du^2+dv^2$ along the $2$-sphere~\eqref{two_sphere_normal_form} where $z=x+iy$ and $w=u+iv$. 
There is a double covering map
\begin{equation}
D_0:S^3\to T^1_{g_0}S^2 %\qquad\qquad Z \in SU(2) \simeq S^3 \mapsto (Z^{-1}jZ,-Z^{-1}kZ) \in T^1_{g_0}S^2.
\end{equation}
{ such that, following~\cite{HP}, it holds }% one computes
\begin{equation}\label{factor_4}
D_0^*\lambda_{g_0} = 4 \lambda_0
\end{equation}
where $\lambda_0$ denotes the restriction to $S^3$ of the $1$-form 
\begin{equation}
%\frac{1}{4i} (\bar zdz-zd\bar z + \bar wdw - wd\bar w) = 
\frac{1}{2} ( xdy-ydx + udv-vdu ).
\end{equation}
In other words, the geodesic flow of $g_0$ lifts to the Hopf Reeb flow on $S^3$ up to a constant time reparametrisation; to explain the factor $4$ note that a Hopf fibre has Reeb flow period $\pi$ with respect to $\lambda_0$, and is the lift of a great circle of length $2\pi$ prescribed twice. Given any other metric $g$ on $S^2$, using the contactomorphism between $T^1_{g_0}S^2$ and $T^1_gS^2$, we get a double covering map
\begin{equation}\label{covering_D_g}
D_g %= \mathcal{L}_g^{-1} \circ \Pi_g \circ \mathcal{L}_{g_0} \circ D_0 
: S^3 \to T^1_gS^2
\end{equation}
respecting contact structures, i.e.
\begin{equation}
D_g^*\lambda_g = f_g\lambda_0
\end{equation}
for some smooth $f_g:S^3 \to \R\setminus\{0\}$. The covering~\eqref{covering_D_g} is the universal covering and the group of deck transformations is $\Z_2$ generated by the antipodal map. If $g=g_0$ then  $f_{g_0}\equiv 4$ by~\eqref{factor_4}. The following special case of a result from~\cite{HP} relates Gaussian curvature to dynamical convexity; see also~\cite{finsler} for an alternative proof.

\begin{theorem}[Harris and Paternain]
If $(S^2,g)$ is $\delta$-pinched with $\delta>1/4$, then the Reeb flow of $D_g^*\lambda_g$ is dynamically convex.
\end{theorem}

\begin{remark}
In~\cite{HP} one finds a version of the above theorem for Finsler metrics. Dynamical convexity is ensured once the flag curvatures are pinched by more than $(r/(r+1))^2$ where $r\geq 1$ is the reversibility parameter. In~\cite{finsler} one finds examples showing that this pinching condition is sharp for dynamical convexity.
\end{remark}

From now on we denote
\begin{equation*}
\begin{aligned} \phi_g^t & \qquad \text{the geodesic flow on $T^1_gS^2$} \\ \phi^t & \qquad \text{the lift of $\phi_g^t$ to $S^3$ via $D_g$} \end{aligned}
\end{equation*}
so that the identity
\begin{equation*}
\phi^t_g \circ D_g = D_g \circ \phi^t
\end{equation*}
holds. The generating vector field $X$ of $\phi^t$ is the Reeb vector field of the contact form~$f_g\lambda_0$. Let $\tilde\sigma_g$ be the lift of the frame $\sigma_g$~\eqref{global_geodesic_frame} to $S^3$ by $D_g$. Then $\tilde\sigma_g$ is a $d(f_g\lambda_0)$-symplectic global frame of $\xi = \ker\lambda_0$, and as in~\eqref{angular_coord_sigma} we get a well-defined global circle coordinate $\Theta_{\tilde\sigma_g}:(\xi\setminus0)/\R_+ \to \R/2\pi\Z$. The linearised flow $D\phi^t$ on $\xi$ induces a flow on $(\xi\setminus0)/\R_+$, still denoted $D\phi^t$, whose generating vector field is denoted by~$\hat X$.

The following elementary lemma can be proved with the arguments presented in~\cite[section~2.4.2]{finsler}.

\begin{lemma}\label{lemma_angular_twist_curvature}
We have $\min\{1,K_{\mathrm{min}}\} \leq i_{\hat X}d\Theta_{\tilde\sigma_g} \leq \max\{1,K_{\mathrm{max}}\}$ where we denote by $K_{\mathrm{min}},K_{\max}$ the minimum and the maximum of the Gaussian curvature of $(S^2,g)$.
\end{lemma}

%\begin{proof}
%Let $p\in S^3$ and $v_0\in\xi_p$ be arbitrary. Then $D_g(\phi^t(p)) = (d(t),\dot d(t))$ for some unit speed geodesic $d(t)$ on $(S^2,g)$. With $v(t) = D\phi^t \cdot v_0 \in \xi_{\phi^t(p)}$, we have $dD_g \cdot v(t) \in \ker \lambda_g$ and can use the splitting $TTS^2 \simeq \mathcal{H} \oplus \mathcal{V}$ to decompose $$ dD_g \cdot v(t) = \left( \frac{DJ}{dt} \right)^\ver + \left( J(t) \right)^\hor $$ into vertical and horizontal components, for some Jacobi field $J(t)$ along $d(t)$ everywhere perpendicular to $\dot d(t)$, see~\cite[Lemma~2.3]{finsler}. We find a real-valued function $b(t)$ such that $J(t)=b(t)\dot d(t)^\bot$ and $\frac{DJ}{dt}=b'(t)\dot d(t)^\bot$. The equation for Jacobi fields becomes $b''(t)=-K(t)b(t)$ where $K(t)$ is the Gaussian curvature at $d(t)$. Using the frame $\sigma_g$~\eqref{global_geodesic_frame} we can represent $(dD_g)_{\phi^t(p)} \cdot v(t)$ as $u(t)=(b'(t),b(t)) \in \R^2$ satisfying $$ u'(t) = J_0S(t) u(t) \qquad\qquad S(t) = \begin{pmatrix} 1 & 0 \\ 0 &  K(t) \end{pmatrix} \qquad J_0 = \begin{pmatrix} 0 & -1 \\ 1 & 0 \end{pmatrix}. $$ If we write $u(t)=|u(t)|e^{i\theta(t)}$ then we get $$ \theta'(t) = \left< S(t)e^{i\theta(t)},e^{i\theta(t)} \right>  $$ from where the desired conclusion follows.
%\end{proof}

Let $c:\R/L\Z \to (S^2,g)$ be a unit speed smooth immersion. It induces a smooth immersion $s \in \R/L\Z \mapsto (c(s),\dot c(s)) \in T^1_{g}S^2$ such that $\lambda_g \cdot \frac{d}{ds}(c,\dot c) = g(c)(\dot c,\dot c) = 1$. If $c$ has no positive self-tangencies then $(c,\dot c)$ defines a knot on $T^1_gS^2$. Note that $\pi_1(T^1_{g_0}S^2,\mathrm{pt})$ is isomorphic to $\Z/2\Z$, and a generator can be taken as $s\mapsto (c(s),\dot c(s))$ for an embedded unit speed loop $s\mapsto c(s)$ in $S^2$.

\begin{lemma}[\cite{finsler}, subsection~3.1]\label{lemma_unknots}
Let $c:\R/L\Z \to S^2$ be a smooth unit speed embedding, and denote by $\gamma_c:\R/2L\Z\to S^3$ a lift to $S^3$ by $D_g$ of the double cover of $s\mapsto (c(s),\dot c(s))$. Then $\gamma_c$ is unknotted with self-linking number $-1$.
\end{lemma}

Suppose now that  $c:\R/L\Z \to S^2$ is a unit speed embedding. Choose a lift $\gamma_c:\R/2L\Z\to S^3$ of the double cover of $s\mapsto (c(s),\dot c(s))$ to $S^3$ by $D_g$. The Birkhoff annulus $A_c \subset T_g^1S^2$ associated to $c$ is parametrised by
\begin{equation}
a:\R/L\Z \times[0,\pi] \to T^1_{g}S^2, \qquad a(s,\theta) = (c(s),\cos\theta \ \dot c(s) + \sin\theta \ \dot c(s)^\bot).
\end{equation}
It follows from this formula that $a$ admits a unique double lift to $S^3$
\begin{equation}\label{lifting_a}
\tilde a:\R/2L\Z \times [0,\pi] \to S^3
\end{equation}
fixed by requiring that $D_g \circ\tilde a(s,\theta) = a(s\mod L,\theta)$ holds identically, together with the boundary condition $\tilde a(s,0)=\gamma_c(s)$. The image $\tilde A_c$ of $\tilde a$ is an embedded annulus in $S^3$ and will be referred to as the \textit{lifted Birkhoff annulus} associated to~$c$. If $A_c \subset T_g^1S^2$ is the image of $a$ then as sets we have $\tilde A_c = D_g^{-1}(A_c)$ which holds because $\tilde A_c$ is antipodal symmetric and satisfies $D_g(\tilde A_c) = A_c$. When $\tilde A_c$ is oriented by $ds\wedge d\theta$ the oriented boundary $\partial\tilde A_c$ consists of $\gamma_c$ together with a lift $\hat\gamma_c:\R/2L\Z\to S^3$ of the double cover of $s\mapsto (c(-s),-\dot c(-s))$.

\begin{lemma}\label{lemma_cohom_disk_versus_annulus}
The identity $$ \link(\beta,\gamma_c) + \link(\beta,\hat\gamma_c)= \mathrm{int}(\beta,\tilde A_c) $$ holds for every loop $\beta$ on $S^3\setminus (\gamma_c \cup \hat\gamma_c)$.
\end{lemma}

\begin{proof}
Let $D$ and $\hat D$ be oriented disks spanned by $\gamma_c$ and $\hat\gamma_c$, respectively, in such a way that the boundary orientations coincide with the flow orientation. Then $C = D + \hat D - \tilde A_c$ is a $2$-cycle on $S^3$. Thus $0 = \mathrm{int}(\beta,C)$ holds for every loop $\beta$ on $S^3$. The conclusion follows.
\end{proof}

From now on we make the standing assumption that $c$ is an embedded closed geodesic. We consider covering maps
\begin{equation}\label{covering_maps}
\begin{aligned}
P : \R/2L\Z\times[0,\pi] \to \R/L\Z\times[0,\pi] \qquad (s,\theta) &\mapsto (s\mod L,\theta), \\
P_\infty : \R\times[0,\pi] \to \R/2L\Z\times[0,\pi] \qquad (s,\theta) &\mapsto (s\mod 2L,\theta).
\end{aligned}
\end{equation}
We have an identity $D_g\circ \tilde{a} = a \circ P$. Let $\Psi:A_c\to A_c$ be the return map, which exists by the result of Birkhoff~\cite{birkhoff}. Note that, in principle, this return map would only be defined on the interior of $A_c$ but then it can be extended smoothly up to the boundary by taking second conjugate points. We proceed assuming that the map has been extended in this manner. Then $\tilde A_c$ is a global surface of section for the lifted flow on $S^3$ and the return map $\tilde{\Psi}:\tilde{A}_c\to\tilde{A}_c$ satisfies $D_g\circ\tilde{\Psi}=\Psi\circ D_g$. As above, this map exists and is smooth up on the closed annulus $\tilde A_c$.

Later we will use the geometry to choose an appropriate smooth lift
\begin{equation}\label{eqn_lift}
\bar{\Psi}:\R\times[0,\pi]\to\R\times[0,\pi]
\end{equation}
of $\tilde\Psi$, i.e., a map that satisfies $\tilde{a}\circ P_\infty \circ \bar{\Psi} = \tilde{\Psi} \circ \tilde{a} \circ P_\infty$, and hence makes the following diagram commute
\begin{equation*}
\begin{tikzcd}
\R\times[0,\pi] \arrow[r,"P_\infty"] \arrow[d,"\bar{\Psi}"] & \R/2L\Z\times[0,\pi]\arrow[r,"\tilde{a}"] & \tilde{A}_c\arrow[r,"D_g"] \arrow[d,"\tilde{\Psi}"]& A_c \arrow[d,"\Psi"]\\
\R\times[0,\pi] \arrow[r,"P_\infty"] & \R/2L\Z\times[0,\pi]\arrow[r,"\tilde{a}"] & \tilde{A}_c\arrow[r,"D_g"] & A_c
\end{tikzcd}
\end{equation*}

\subsection{Asymptotic estimates on linking and intersection numbers}\label{ssec_estimates_linking_numbers}

Denote by $\delta>0$ the pinching factor. We may assume that 
\begin{equation}\label{normalisation_curvature}
\delta = K_{\min} \leq K_{\max} = 1
\end{equation}
where $K_{\min},K_{\max}$ denote the minimum and the maximum of the Gaussian curvature. Recall that if $\delta>1/4$, then $\phi_g^t$ is dynamically convex \cite{HP}, hence so is~$\phi^t$. Lemma~\ref{lemma_unknots} and Theorem~\ref{thm_JSG} imply that both $\gamma_c$ and $\hat\gamma_c$ span disk-like global surfaces of section. Denote by $\hat D$ a $\partial$-strong disk-like surface of section spanned by~$\hat\gamma_c$;~see Remark~\ref{rmk_JSG_strong}.

We write $\Theta = \Theta_{\tilde\sigma_g}$ for simplicity. With $u \in \xi$, $u\neq0$, arbitrary we denote $$ \Delta\Theta(T,u) = \tilde{\Theta}(T,u)-\tilde{\Theta}(0,u) $$ where $t\mapsto \tilde\Theta(t,u)$ is a continuous lift of $t\mapsto \Theta(D\phi^t(u))$. It does not depend on the choice of lift.

Let $T\geq 0$, $x \in S^3\setminus \hat\gamma_c$. Recall that $k(T,x;\hat D)$ is a loop obtained by closing the piece of trajectory $\phi^{I(T,x;\hat D)}(x)$ with a path $\alpha$ in $\hat D\setminus\hat\gamma_c$. From now on we shall refer to $\alpha$ as a closing path for $(T,x;\hat D)$. Clearly, $\link(k(T,x;\hat D),\hat\gamma_c)$ does not depend on the choice of $\alpha$. Suppose further that $x \in S^3 \setminus (\gamma_c \cup \hat\gamma_c)$. The number $\text{int}(k(T,x;\hat D),\tilde{A}_c)$ might not be well-defined since $\alpha$ could go through the (unique) intersection point of $\gamma_c$ and $\hat D$. Even if it does not touch this point, $\alpha$ can be chosen in such a way that $\text{int}(k(T,x;\hat D),\tilde{A}_c)$ is any integer. Below we might write $k_\alpha(T,x;\hat D)$ if the dependence on $\alpha$ needs to be made explicit.

\begin{lemma}\label{lemma_application_gen_pos}
There exists $C>0$ with the following property. For every point $x\in S^3\setminus(\gamma_c\cup\hat\gamma_c)$ and every $T>0$ there exists a closing path $\alpha$ for $(T,x;\hat D)$ contained in $\hat D\setminus(\gamma_c \cup \hat \gamma_c)$ such that
\begin{equation}\label{0_est_linking}
| {\rm int}(k_{\alpha}(T,x;\hat D),\tilde A_c) - \#\{t\in[0,T]\mid \phi^t(x) \in \tilde A_c\} | \ \leq \ C.
\end{equation}
\end{lemma}

\begin{proof}
Since $\hat D$ is $\partial$-strong, {it is possible to} %we can invoke Lemma~\ref{lemma_gen_pos_1} to 
slightly $C^\infty$-perturb $\hat D$ to a new $\partial$-strong disk-like global surface of section $\tilde D$ satisfying $\partial\tilde D = \hat\gamma_c$ { and such that $\tilde D\setminus\partial\tilde D$ intersects $\tilde A_c\setminus\partial \tilde A_c$ transversely, and 
%for every connected component $\gamma$ of $\partial\tilde D\cap\partial\tilde A_c$ 
the manifolds $\{\R_+\nu \mid \nu\in T\tilde D\vert_\gamma \text{ is outward pointing} \}$ and $\{\R_+\nu \mid \nu\in T\tilde A_c\vert_\gamma \text{ is outward pointing} \}$ induce loops in the $2$-torus $\mathbb{P}_+\xi|_{\hat\gamma_c}$ that intersect transversely.} %in relative generic position with respect to $\tilde A_c$.
Denote by $K$ the closure of $(\tilde D \setminus \partial \tilde D) \cap (\tilde A_c \setminus \partial \tilde A_c) = (\tilde D \cap \tilde A_c) \setminus (\hat\gamma_c \cup \gamma_c)$. %By Lemma~\ref{lemma_gen_pos_2}, 
{Then, }$K$ is a smooth compact $1$-manifold, $\partial K = K \cap (\hat\gamma_c \cup \gamma_c)$, and at the boundary points of $K$ the $1$-manifolds $\hat\gamma_c \cup \gamma_c$ and $K$ are not tangent.

Recall the interval $I(T,x;\hat D)$ defined in~\eqref{interval_(T,x)}. Denote by $\hat x_0 \in \hat D$ and $\hat x_1 \in \hat D$ the initial and end points of the trajectory $\phi^{I(T,x;\hat D)}(x)$. We consider the small transfer map $\psi : p \in \tilde D \mapsto \phi^{g(p)}(p) \in \hat D$ where $g$ is {a smooth function on $\tilde D$, up to the boundary, $C^\infty$-close to zero, satisfying $\phi^{g(p)}(p)\in\hat D$ for every $p\in\tilde D$}. %the function in (e) Lemma~\ref{lemma_gen_pos_1}, see Remark~\ref{rmk_small_transfer_map}. 
Define $\tilde x_0 = \psi^{-1}(\hat x_0)$, $\tilde x_1 = \psi^{-1}(\hat x_1)$. There is an induced interval $\tilde I$ whose end points are close to those of $I(T,x;\hat D)$ such that $\phi^{\tilde I}(x)$ is a piece of trajectory from $\tilde x_0$ to $\tilde x_1$. The Lebesgue measure of the symmetric difference between $\tilde I$ and $I(T,x;\hat D)$ is not larger than $2\|g\|_{L^\infty} \ll1$.

Assume $\tilde x_0\neq \tilde x_1$. Consider a smooth immersion $\alpha_0 : [0,1] \to \tilde D \setminus (\hat\gamma_c \cup \gamma_c)$ from $\tilde x_1 = \alpha_0(0)$ to $\tilde x_0 = \alpha_0(1)$ that is transverse to $K$. Let $\beta$ be a connected component of $K$, and suppose $t_0<t_1$ satisfy $\alpha_0(t_0),\alpha_0(t_1)\in\beta$. Consider the piecewise smooth arc $\hat \alpha_0$ obtained from $\alpha_0$ by replacing $\alpha_0|_{[t_0,t_1]}$ with the arc in $\beta$ from $\alpha_0(t_0)$ to $\alpha_0(t_1)$. A further small perturbation of $\hat \alpha_0$ creates a smooth immersed arc $\alpha_1:[0,1] \to \tilde D \setminus (\hat\gamma_c \cup \gamma_c)$ from $\tilde x_1$ to $\tilde x_0$, transverse to $K$, such that the number of intersection points of $\alpha_1$ and any component of $K$ is no larger than that of $\alpha_0$, and the number of intersection points of $\alpha_1$ with $\beta$ is strictly less than the number of intersection points of $\alpha_0$ with~$\beta$. Proceeding inductively in this way, after a finite number of steps we end up with an immersed arc $\tilde\alpha$ in $\tilde D \setminus (\hat\gamma_c \cup \gamma_c)$ from $\tilde x_1$ to $\tilde x_0$ which is transverse to $K$ and intersects each connected component of $K$ at most once. Consider the loop $\tilde k$ obtained by concatenating $\phi^{\tilde I}(x)$ with $\tilde \alpha$. Then 
\begin{equation}\label{1_est_linking}
|{\rm int}(\tilde k,\tilde A_c) - \#\{t\in \tilde I \mid \phi^t(x) \in \tilde A_c\}| \ \leq \ N + 2
\end{equation}
where $N$ is the number of connected components of $K$. The image $\alpha = \psi(\tilde \alpha) \subset \hat D$ is an arc from $\hat x_1$ to $\hat x_0$. By construction $k_{\alpha}(T,x;\hat D)$ is homotopic to $\tilde k$ on $S^3\setminus (\gamma_c \cup \hat \gamma_c)$. Hence
\begin{equation}\label{2_est_linking}
{\rm int}(\tilde k,\tilde A_c) = {\rm int}(k_{\alpha}(T,x;\hat D),\tilde A_c)
\end{equation}
Finally, consider $\tau_{\max}(\hat D)$ the supremum of the return time of $\hat D$ and $\tau_{\min}(\tilde A_c)>0$ the infimum of the return time of $\tilde A_c$. Note that the Lebesgue measure of the symmetric difference of $\tilde I$ and $[0,T]$ is not larger than $2(\tau_{\max}(\hat D)+\|g\|_{L^\infty})$. Note also that for any interval $J$ one estimates
\begin{equation*}
\# \{ t\in J \mid \phi^t(x) \in \tilde A_c \} \leq \frac{|J|}{\tau_{\min}(\tilde A_c)} + 2
\end{equation*}
Hence
\begin{equation}\label{3_est_linking}
\begin{aligned}
& | \# \{ t\in \tilde I \mid \phi^t(x) \in \tilde A_c \} - \# \{ t\in [0,T] \mid \phi^t(x) \in \tilde A_c \} | \\
& \qquad \leq \ \frac{2(\tau_{\max}(\hat D)+\|g\|_{L^\infty})}{\tau_{\min}(\tilde A_c)} + 2
\end{aligned}
\end{equation}
Combining~\eqref{1_est_linking},~\eqref{2_est_linking} and~\eqref{3_est_linking} we arrive at~\eqref{0_est_linking} with 
\begin{equation*}
C = N + \frac{2(\tau_{\max}(\hat D)+\|g\|_{L^\infty})}{\tau_{\min}(\tilde A_c)} + 4
\end{equation*}
as desired.

The case $\tilde x_0 = \tilde x_1$ is simpler since there is no need to consider an immersion $\alpha$ to close the path.
\end{proof}

Next we recall some definitions and facts about Riemannian geometry that will be useful all along this section; see~\cite{ABHS}.
	
\begin{definition}
A set $D\subset S^2$ is said to be a geodesic polygon if it is the closure of an open disk bounded by a simple closed unit speed broken geodesic $\gamma:\R/L\Z\to S^2$. A corner of $\gamma:\R/L\Z\to S^2$ is a point $\gamma(t)$ such that $\gamma'_+(t)\notin\R^+\gamma_-'(t)$, where $\gamma'_{\pm}$ denote one-sided derivatives. The corners of $\gamma$ are called vertices of $D$, and a side of $D$ is a smooth geodesic arc contained in $\partial D$ connecting two adjacent vertices.
\end{definition}

For $p\in S^2$ and $u,v\in T_pS^2$ non-collinear vectors, consider the sets
\[
\Delta(u,v):=\{ su+tv |\ s,t\geq 0 \}, \qquad \Delta_r(u,v):=\{ w\in\Delta(u,v) |\ |w|<r \}
\]
for some $r>0$. We denote by ${\rm inj}_p$ the injectivity radius at $p$, and ${\rm inj} = \inf_p {\rm inj}_p$ the injectivity radius of $g$.

\begin{definition}
A geodesic polygon $D\subset S^2$ is convex if for every corner $p=\gamma(t)$ of $\partial D$ we find $0<r<\mathrm{inj}_p$ small enough such that $D\cap B_r(p)=\mathrm{exp}(\Delta_r(-\gamma_-'(t),\gamma_+'(t)))$.
\end{definition}
%\marginpar{Maybe we can erase these pictures, i.e. Figure 1??}

\begin{figure}[h]
	\begin{center}
	\includegraphics[scale=0.07]{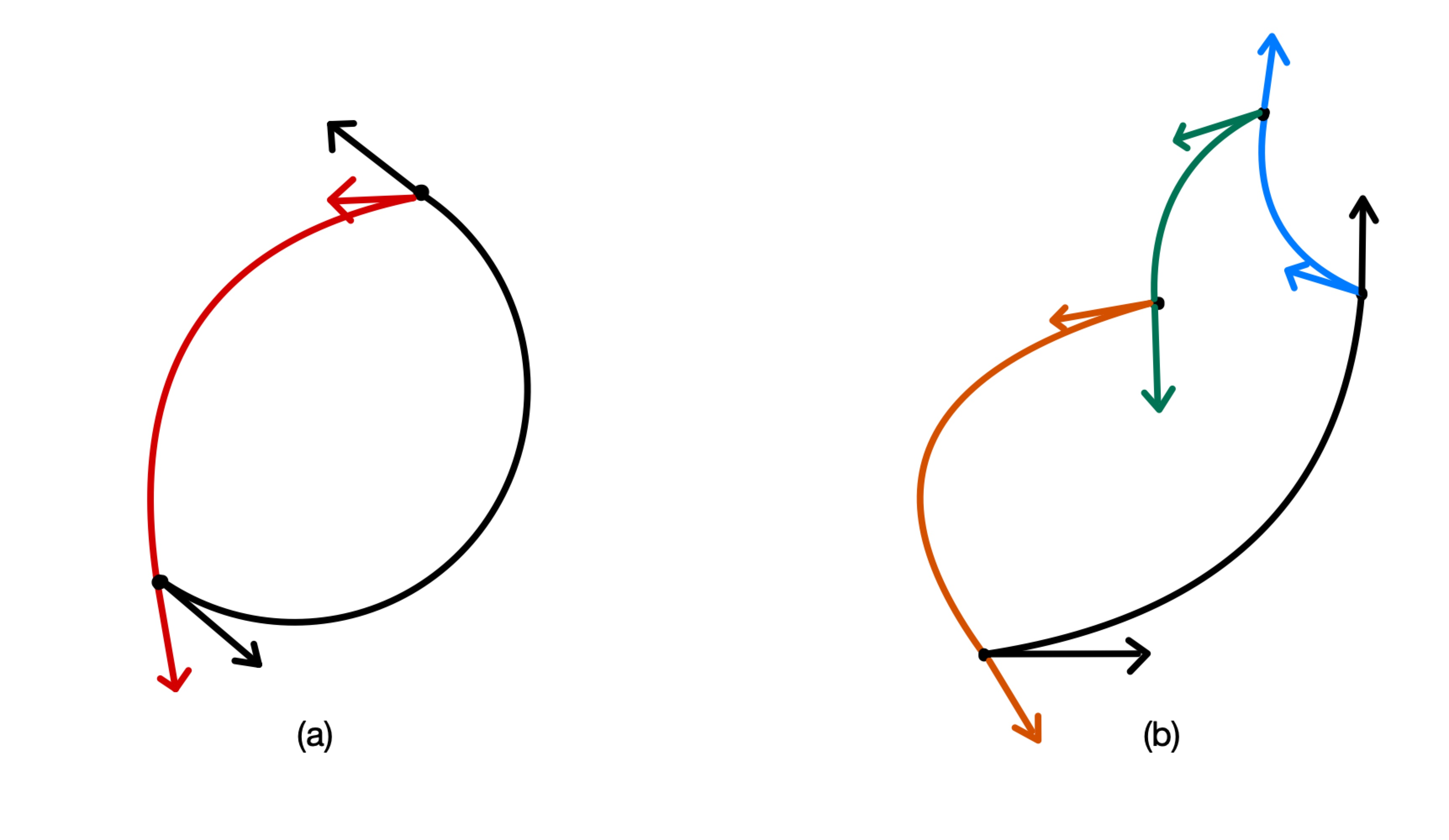}
	\caption{Examples of geodesic polygons, convex $(a)$ and non convex $(b)$.}
	\end{center}
\end{figure}
	
\begin{theorem}[Corollary of Toponogov's theorem~{\cite[Theorem A.12]{ABHS}}]\label{teo A 12}
If $D$ is a convex geodesic polygon in $(S^2,g)$ then $\vert \partial D \vert \leq 2\pi/\sqrt{\delta}$ where $\vert\partial D\vert$ denotes the perimeter of $D$.
\end{theorem}

{For the following theorem, recall that we are assuming $K_{\max}=1$, see~\eqref{normalisation_curvature}.}

\begin{theorem}[{\cite{Kli59},\cite[Theorem 2.6.9]{Kli82}}]\label{Kli}
The injectivity radius satisfies ${\rm inj} \geq \pi$.
\end{theorem}

In particular it follows from theorems~\ref{teo A 12} and~\ref{Kli} that if $D\subset S^2$ is a convex geodesic {bi-gon} for the metric $g$ then
\begin{equation}\label{ineq Kli}
2\pi \leq \vert\partial D\vert \leq \frac{2\pi}{\sqrt{\delta}}.
\end{equation}

\begin{notation}
\label{Not alpha+}
Let $(x,v)\in T^1_gS^2$ be such that $x$ belongs to the embedded closed geodesic $c$. If $v$ is not tangent to $c$ then define	
\begin{equation}
\label{tau_plus}
\tau_+(x,v) := \min \{ t>0 \mid \pi \circ \phi_g^t(x,v) \in c \}.
\end{equation}
That is, $\tau_+(x,v)$ is the first positive time when the geodesic ray with initial conditions $(x,v)$ meets again $c$. Since $\delta>1/4$, by~\cite[Lemma~3.9]{ABHS} the geodesic arc $\{ \pi \circ \phi_g^{t}(x,v) \mid t\in [0,\tau_+(x,v)] \}$ is embedded, in particular $x \neq \pi \circ \phi_g^{\tau_+(x,v)}(x,v)$. If $v$ is tangent to $c$ then define $\tau_+(x,v)$ as the time to the first conjugate point of $(x,v)$. Denote by $\alpha_+(x,v)$ the path contained in $c$ joining $x$ to $\pi \circ \phi_g^{\tau_+(x,v)}(x,v)$ following the orientation of $c$. See Figure~\ref{two-gons fig}.
\begin{figure}[h]
\begin{center}
\includegraphics[scale=0.3]{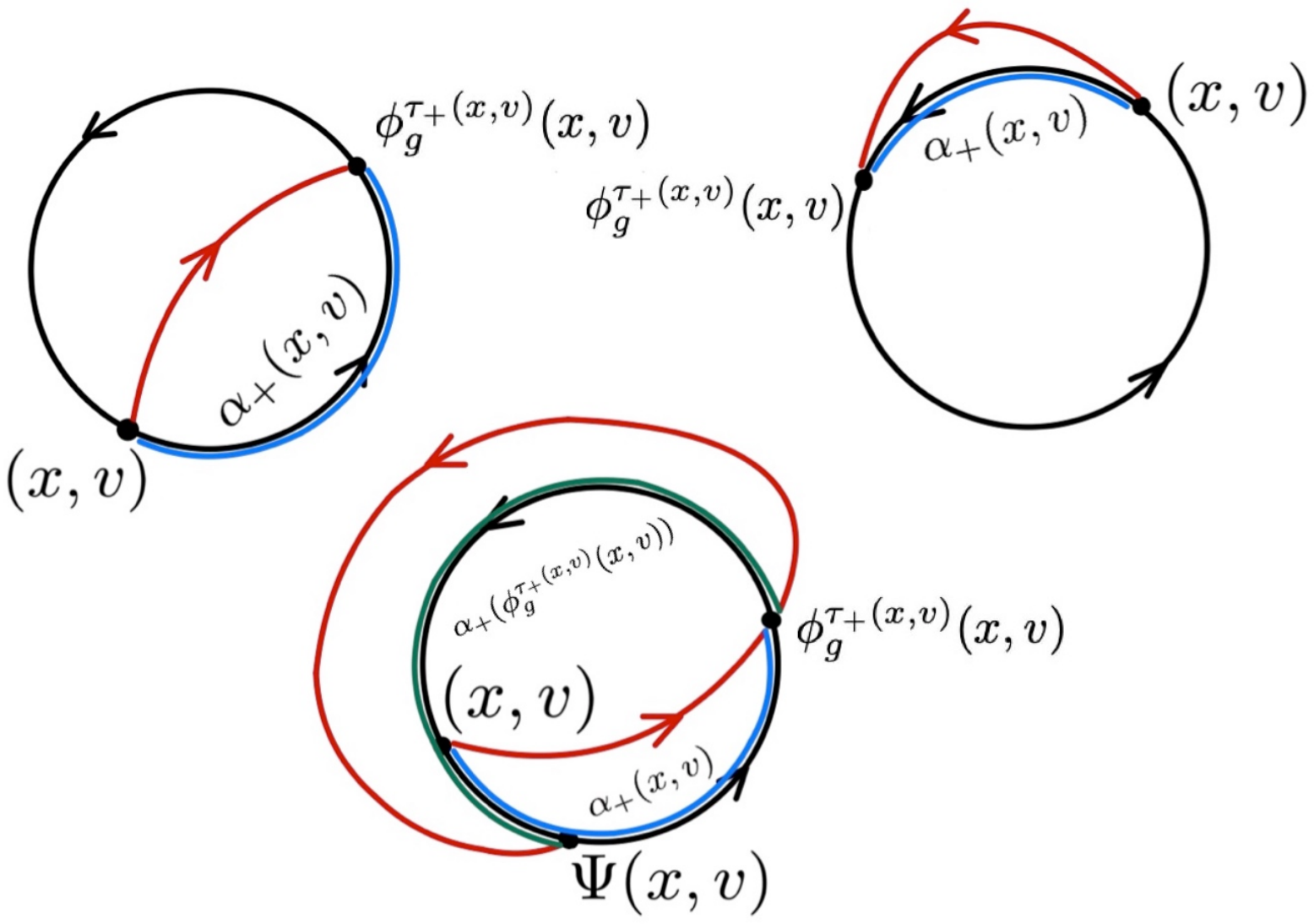}
\caption{Examples of paths $\alpha_+$.}
\label{two-gons fig}
\end{center}
\end{figure}
\end{notation}

{
\begin{lemma}
\label{lemma_new_estimate_pinching}
If $\delta>1/4$ then there exists a unique lift $\bar\Psi:\R\times[0,\pi]\to \R\times[0,\pi]$ such that
\[
L+2\pi\left(1-\frac{1}{\sqrt{\delta}}\right)\leq p_1\circ\bar\Psi(S,\theta)-S\leq L+2\pi\left(\frac{1}{\sqrt{\delta}}-1\right)
\]
holds for every $(S,\theta)\in\R\times[0,\pi]$.
\end{lemma}

\begin{proof}
Any lift of $\tilde a^{-1}\circ\tilde\Psi\circ\tilde a$ preserves the 2-form $\sin\theta \ ds\wedge d\theta$, maps each boundary component into itself, and commutes with the translation $(S,\theta)\mapsto (S+L,\theta)$.
Let $(s,\theta)\in\R/L\Z\times[0,\pi]$. By~\cite[Lemma 3.9]{ABHS} the geodesic arc $\phi_g^{[0,\tau_+(a(s,\theta))]}(a(s,\theta))$ is embedded. 
Denote by $\ell$ its length, and by $\eta=\eta(s,\theta) \in (0,L)$ the length of the arc $\alpha_+(a(s,\theta))$. 
In particular, $\eta$ is a continuous function of $(s,\theta)$.
Consider the convex geodesic polygon bounded by $\phi_g^{[0,\tau_+(a(s,\theta))]}(a(s,\theta))$ and $\alpha_+(a(s,\theta))$. By~\eqref{ineq Kli} we have
\begin{equation}
\label{length alpha+ 1}
2\pi\leq \eta+\ell\leq \dfrac{2\pi}{\sqrt{\delta}}.
\end{equation}
Similarly, we consider the convex geodesic polygon whose boundary consists of the arcs $\phi_g^{[0,\tau_+(a(s,\theta))]}(a(s,\theta))$ and $c\setminus\alpha_+(a(s,\theta))$. 
Observe that the length of $c\setminus\alpha_+(a(s,\theta))$ is $L-\eta\in(0,L)$. Again by~\eqref{ineq Kli}
\begin{equation}
\label{length alpha+ 2}
2\pi\leq L-\eta+\ell\leq\dfrac{2\pi}{\sqrt{\delta}}.
\end{equation}
Thus, considering the difference between~\eqref{length alpha+ 1} and~\eqref{length alpha+ 2}, one gets
\[  
\dfrac L 2 +\pi\left(1-\dfrac{1}{\sqrt{\delta}}\right)\leq \eta \leq \dfrac L 2 +\pi\left( \dfrac{1}{\sqrt{\delta}}-1 \right).
\]
We repeat the same argument, but now using the length $\nu=\nu(s,\theta)$ of the arc $\alpha_+(\phi_g^{\tau_+(a(s,\theta))}(a(s,\theta)))$, and add to obtain	\begin{equation}
\label{bound eta nu}
L  +2\pi\left(1-\dfrac{1}{\sqrt{\delta}}\right)\leq \eta+\nu \leq  L  +2\pi\left( \dfrac{1}{\sqrt{\delta}}-1 \right).
\end{equation}
The endpoints of the arc obtained by concatenating the arc $\alpha_+(a(s,\theta))$ with the arc $\alpha_+(\phi_g^{\tau_+(a(s,\theta))}(a(s,\theta)))$ are $c(s)$ and $c(\hat p_1\circ a^{-1}\circ\Psi\circ a(s,\theta))$, where we denote by $\hat p_1:\R/L\Z\times[0,\pi]\to \R/L\Z$ the projection onto the first coordinate. 
In particular, the $\R$-components of a lift of $\tilde a^{-1}\circ\tilde\Psi\circ\tilde a$ and of the identity map differ by $\eta+\nu+kL$, with some $k\in\Z$. Here we see $\eta$ and $\nu$ as continuous functions of $(S,\theta)$ that are $L$-periodic in~$S$. 
Choose then the lift $\bar\Psi$ such that $k=0$. From~\eqref{bound eta nu} we conclude that
\[
L  +2\pi\left(1-\dfrac{1}{\sqrt{\delta}}\right)\leq p_1\circ\bar\Psi(S,\theta)-S \leq  L  +2\pi\left( \dfrac{1}{\sqrt{\delta}}-1 \right)
\]
holds for every $(S,\theta)\in\R\times[0,\pi]$, as desired.
\end{proof}
}

According to~\cite[Appendix~B]{HSW}, let $(\Pi,\gamma_c\cup\hat{\gamma}_c)$ be the open book decomposition of $S^3$ such that $\tilde{A}_c$ is one of its pages. It holds that $$ H_1(S^3\setminus(\gamma_c\cup\hat{\gamma}_c)) \simeq H_1(\tilde{A}_c)\oplus \left< e \right> \simeq \Z \oplus \Z $$ where $e$ is a loop such that $\Pi_*e$ is the positive generator of $H_1(\R/\Z)$. In particular, we can choose $e$ so that $\link(e,\gamma_c)=0$ and $\link(e,\hat{\gamma}_c)=1$. Choose now a loop $f$ lying in $\tilde{A}_c \setminus (\gamma_c\cup\hat{\gamma}_c)$ which is a generator of $H_1(\tilde{A}_c)$. We can choose $f$ so that $\text{link}(f,\gamma_c)=1$ and $\text{link}(f,\hat{\gamma}_c)=-1$. See Figure \ref{generators}.

\begin{figure}[h]
	\centering
	\includegraphics[scale=0.15]{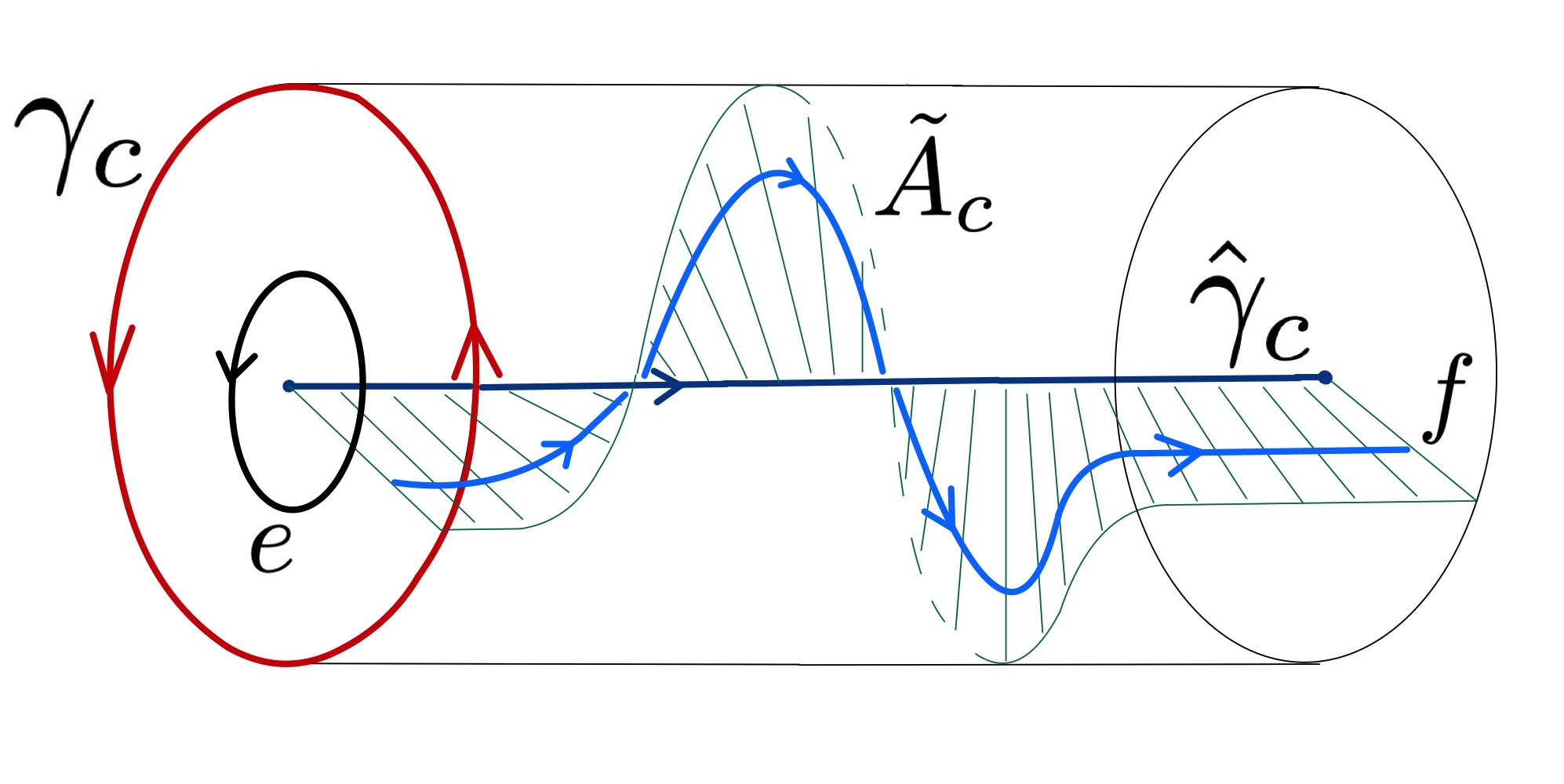}
	\caption{The chosen loops $e$ and $f$.}
	\label{generators}
\end{figure}

\begin{lemma}\label{lemma_f}
The loop $f$ is homotopic to the loop $s \in [0,2L] \mapsto \tilde{a} \circ P_\infty(-s,\theta_*)$ in $S^3 \setminus (\gamma_c\cup\hat{\gamma}_c)$, for any $\theta_* \in (0,\pi)$.
\end{lemma}

\begin{proof}
The free homotopy classes of loops in $S^3 \setminus (\gamma_c\cup\hat{\gamma}_c)$ are determined by the linking numbers with $\gamma_c$ and $\hat{\gamma}_c$. Note that $\tilde{a} \circ P_\infty(-s,0)$ is $-\gamma_c$, and that $\tilde{a} \circ P_\infty(-s,\pi)$ is $\hat{\gamma}_c$, where $s\in[0,2L]$. Hence the loop $\beta : s \in [0,2L] \mapsto \tilde{a} \circ P_\infty(-s,\theta_*)$ is homotopic to $-\gamma_c$ in $S^3\setminus\hat{\gamma}_c$, and is homotopic to $\hat{\gamma}_c$ in $S^3\setminus\gamma_c$. We conclude that $\link(\beta,\hat{\gamma}_c) = \link(-\gamma_c,\hat{\gamma}_c) = -1$ and $\link(\beta,{\gamma}_c) = \link(\hat\gamma_c,{\gamma}_c) = 1$. Hence $\beta$ and $f$ are homotopic in $S^3 \setminus (\gamma_c\cup\hat{\gamma}_c)$. 
\end{proof}

Let us introduce some notation needed for the statement of Lemma~\ref{lemma_horizontal_displacement} below. Let $T>0$ and $x \in S^3\setminus(\gamma_c\cup\hat{\gamma}_c)$, and consider the loop $k(T,x;\hat D)$ in $S^3\setminus(\gamma_c\cup\hat{\gamma}_c)$. There is a choice of closing path in $\hat D \setminus \hat\gamma_c$ which is not made explicit in the notation. Assume that this closing path does not intersect $\gamma_c$, and denote
\begin{equation}\label{notation_M_N}
M(T,x) = \text{link}(k(T,x;\hat D),\hat\gamma_c), \qquad N(T,x) = \text{link}(k(T,x;\hat D),\gamma_c).
\end{equation}
Note that $M$ depends only on $(T,x)$ and that, in contrast, $N$ highly depends on the choice of closing path; again we do not make this dependence explicit in the notation, for simplicity. By Lemma~\ref{lemma_cohom_disk_versus_annulus} we write in  $H_1(S^3\setminus(\gamma_c\cup\hat{\gamma}_c))$
\begin{equation}\label{coordinates ok k(T,x,D)}
k(T,x;\hat D)=(M(T,x)+N(T,x)) e + N(T,x) f.
\end{equation}
Let {$\underline t \geq 0$} and $\underline x \in \tilde{A}_c \setminus (\gamma_c\cup\hat{\gamma}_c)$ be defined by
\begin{equation}\label{special_point_and_time}
\underline t = \min \{ t \geq 0 \mid \phi^{t+t^{\hat D}_-(x)}(x) \in \tilde A_c \}, \qquad \underline x = \phi^{\underline t + t^{\hat D}_-(x)}(x).
\end{equation}
Note that $\underline t$ and $\underline x$ depend only on $x$, {$\tilde A_c$} and $\hat D$.

\begin{lemma}
\label{lemma_horizontal_displacement}
{Suppose that $\delta > 1/4$.} There exists $C\geq0$, {which depends only on $\hat D$ and $\tilde A_c$}, with the following significance. 
Let $T>0$ and $x \in S^3\setminus(\gamma_c\cup\hat{\gamma}_c)$ be arbitrary, and let $(s_0,\theta_0)\in\R\times[0,\pi]$ satisfy $\tilde{a}\circ P_\infty(s_0,\theta_0) = \underline x$. Then there exists {a closing path~$\alpha$ in $\hat D \setminus \hat\gamma_c$} such that $$ \left| \dfrac{p_1\circ\bar{\Psi}^{M(T,x)+N(T,x)}(s_0,\theta_0)-s_0}{2L} - M(T,x) \right| \leq C $$ where $p_1:\R\times[0,\pi]\to\R$ is the projection onto the first coordinate, {and such that the conclusion of Lemma~\ref{lemma_application_gen_pos} holds, i.e., $$ | {\rm int}(k_{\alpha}(T,x;\hat D),\tilde A_c) - \#\{t\in[0,T]\mid \phi^t(x) \in \tilde A_c\} | \ \leq \ C \, . $$} 
%
%{\color{olive}There exists $C\geq0$, {\color{red} which depends only on $\hat D$ and $\tilde A_c$}, 
%	with the following significance. 
%	Let $T>0$ and $x \in S^3\setminus(\gamma_c\cup\hat{\gamma}_c)$ be arbitrary, and let $(s_0,\theta_0)\in\R\times[0,\pi]$ satisfy $\tilde{a}\circ P_\infty(s_0,\theta_0) = \underline x$. Then, for the closing path $\alpha$ chosen in Lemma~\ref{lemma_application_gen_pos}, it holds $$ \left| \dfrac{p_1\circ\bar{\Psi}^{M(T,x)+N(T,x)}(s_0,\theta_0)-s_0}{2L} - M(T,x) \right| \leq C $$ where $p_1:\R\times[0,\pi]\to\R$ is the projection onto the first coordinate.}
\end{lemma}

\begin{proof}
Denote $p = \phi^{t^{\hat D}_-(x)}(x)$. 
Consider $$ m(T,x) = \# \{ t\in[0,T+t^{\hat D}_+(\phi^T(x))-t^{\hat D}_-(x)] \mid \phi^{t}(p) \in \tilde{A}_c \}. $$ For each $i = 0,\dots,m(T,x)-1$ let $t_i$ be defined by $$ \phi^{t_i}(p) \in \tilde{A}_c, \quad 0\leq t_0 = \underline{t} < t_1 < \dots < t_{m(T,x)-1} \leq T + t^{\hat D}_+(\phi^T(x)) - t^{\hat D}_-(x). $$ If $(s_0,\theta_0) \in \R \times [0,\pi]$ satisfies $\tilde a \circ P_\infty(s_0,\theta_0) = \underline{x}$ and $(s_i,\theta_i) := \bar{\Psi}^i(s_0,\theta_0)$ then we have $$ \tilde{a} \circ P_\infty (s_i,\theta_i) = \phi^{t_i}(p) $$ { and by Lemma~\ref{lemma_new_estimate_pinching} $$ s_{i+1} - s_i > 0 $$} for every $i = 0,\dots,m(T,x)-1$. Let $k \in \Z$ be such that
\begin{equation}\label{defi p}
k-1 < \dfrac{s_{m(T,x)-1}-s_0}{2L} \leq k.
\end{equation}

We introduce a path $\varepsilon:[0,1]\to\R\times[0,\pi]$ from $\varepsilon(0) = (s_{m(T,x)-1},\theta_{m(T,x)-1})$ to $\varepsilon(1)=(s_0+2Lk,\theta_0)$ defined in the following way:
\begin{equation*}
\varepsilon(\tau) = 
\begin{cases}
(s_{m(T,x)-1}+2\tau(s_0+2Lk-s_{m(T,x)-1}),\theta_{m(T,x)-1}) & 0 \leq \tau \leq 1/2, \\
(s_0+2Lk,\theta_{m(T,x)-1}+(2\tau-1)(\theta_0-\theta_{m(T,x)-1})) & 1/2 \leq \tau \leq 1.
\end{cases}
\end{equation*}
See Figure~\ref{Fig Eps}.

\begin{figure}[h]
	\begin{center}
\includegraphics[scale=0.5]{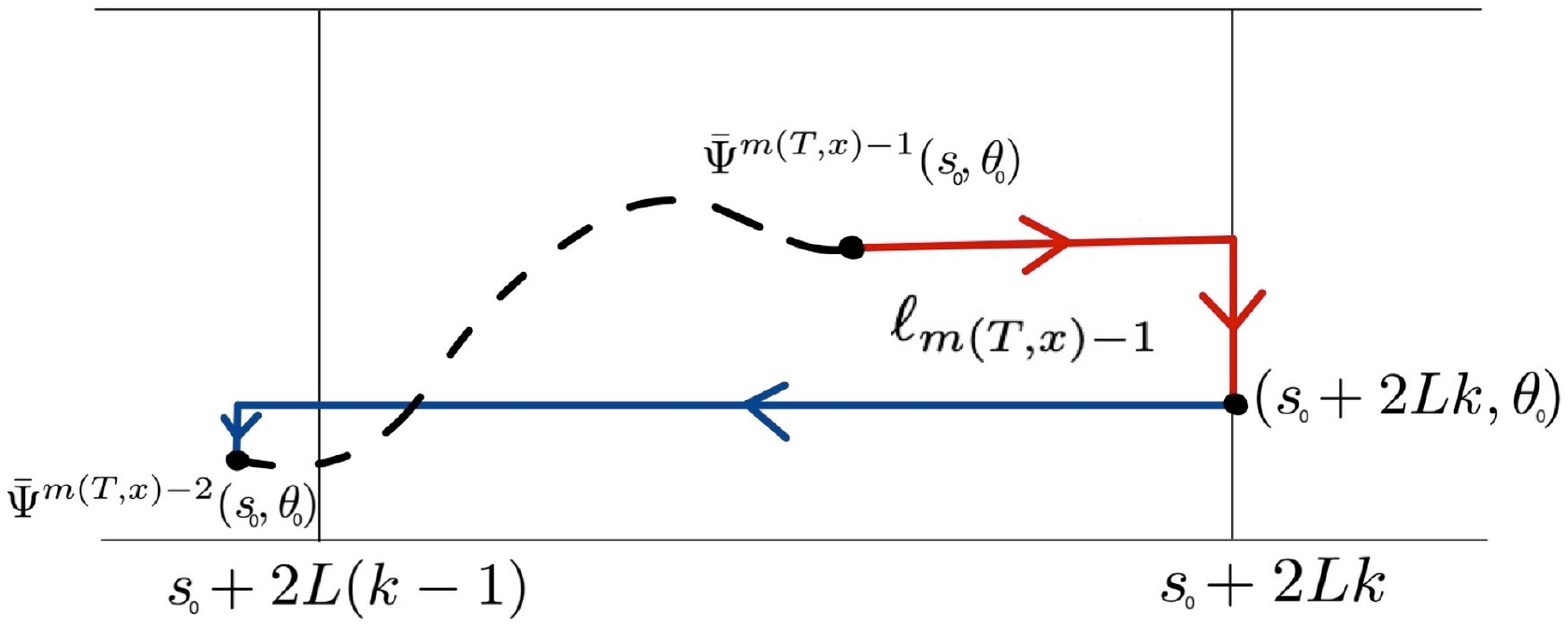}
\caption{The path $\ell_{m(T,x)-1}$.}
\label{Fig Eps}
\end{center}
\end{figure}

For every $i=1,\dots,m(T,x)-1$ construct a path $\ell_i : [0,1] \to \R\times[0,\pi]$ as follows. If $i \leq m(T,x)-2$ then
\begin{equation*}
\ell_i(\tau) = 
\begin{cases}
(s_i+2\tau(s_{i-1}-s_i),\theta_i) & 0 \leq \tau \leq 1/2, \\
(s_{i-1},\theta_i+(2\tau-1)(\theta_{i-1}-\theta_i)) & 1/2 \leq \tau \leq 1.
\end{cases}
\end{equation*} 
The path $\ell_{m(T,x)-1}$ is defined similarly, but following $\varepsilon$ with a path from $\varepsilon(1)=(s_0+2Lk,\theta_0)$ to $(s_{m(T,x)-2},\theta_{m(T,x)-2})$, see Figure~\ref{Fig Eps}. Note that $\ell_i$ connects $(s_i,\theta_i)$ to $(s_{i-1},\theta_{i-1})$.

\begin{figure}[h]
	\centering
\includegraphics[scale=0.7]{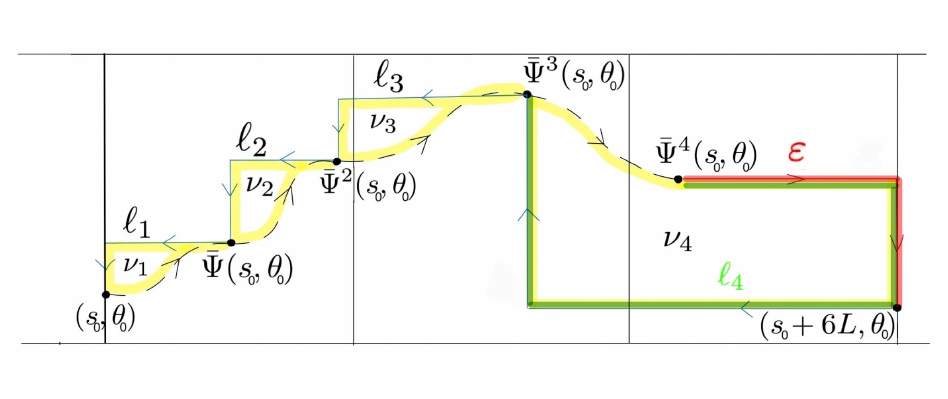}
\caption{An example of the paths $\ell_i,\nu_i$ and $\varepsilon$.}
\label{building paths}
\end{figure}

Denote by $\mathcal{L}$ the concatenation of the path $\tau\mapsto \varepsilon(1-\tau)$ with all the paths $\ell_i$, $i=1,\dots,m(T,x)-1$. By Lemma~\ref{lemma_f}, the path $\tilde{a}\circ P_{\infty}(\mathcal{L})$ is a loop homotopic to $kf$ in $S^3\setminus(\gamma_c\cup\hat{\gamma}_c)$.
{Recall $p = \phi^{t^{\hat D}_-(x)}(x)$. For $i=1,\dots,m(T,x)-1$ denote by $\nu_i$ the loop $$ \nu_i = \phi^{[t_{i-1},t_i]}(p)+\tilde{a}\circ P_{\infty}(\ell_i). $$} See Figure~\ref{building paths}. 
%Consider also\marginpar{Put together?} $$ \nu_{m(T,x)-1} = \phi^{[t_{m(T,x)-2},t_{m(T,x)-1}]}(p) 
%%+ \tilde{a}\circ P_{\infty}(\varepsilon) 
%+ \tilde{a}\circ P_{\infty}(\ell_{m(T,x)-1}). $$ 
In $S^3\setminus(\gamma_c\cup\hat{\gamma}_c)$, since $\tilde{a}\circ P_{\infty}(\mathcal{L})$ is homotopic to $kf$, we have that $$ \phi^{[t_0,t_{m(T,x)-1}]}(p)+\tilde{a}\circ P_{\infty}(\varepsilon)+kf $$ is homologous to $$ \nu_1+\dots+\nu_{m(T,x)-1}. $$ This implies
\begin{equation}\label{calculus_link}
\begin{aligned}
& \link(\phi^{[t_0,t_{m(T,x)-1}]}(p)+\tilde{a}\circ P_{\infty}(\varepsilon),\hat{\gamma}_c) - k \\
& \qquad = \link(\phi^{[t_0,t_{m(T,x)-1}]}(p)+\tilde{a}\circ P_{\infty}(\varepsilon)+kf,\hat{\gamma}_c) \\
& \qquad = \sum_{i=1}^{m(T,x)-1} \link(\nu_i,\hat{\gamma}_c).
\end{aligned}
\end{equation}

\noindent {\sc Claim 1.} For every $i=1,\dots,m(T,x)-1$ it holds $\text{link}(\nu_i,\hat{\gamma}_c)=0$. \\

\noindent {\it Proof of Claim 1.} 
The loop $\nu_i\subset S^3 \setminus (\gamma_c\cup\hat\gamma_c)$ is made up of a trajectory of the lifted geodesic flow $\phi^t$ and a path contained in the annulus $\tilde{A}_c$. Consider the loop $D_g(\nu_i)\subset T^1_gS^2$, see Figure \ref{homotopy}. This loop consists of the velocity vectors of a geodesic starting at (and transversely to) $c$ up to the second hitting point with $c$ (first hit with $A_c$), and then of velocity vectors obtained by parallel transport back the initial point, and then of a vertical deformation. In particular, $D_g(\nu_i)$ does not intersect the set $\dot c \cup (-\dot c) = D_g(\gamma_c\cup\hat\gamma_c)$. The path $D_g(\nu_i)$ can be deformed in a continuous way into a loop $\Gamma$ contained in~$\dot c$ as depicted in Figure \ref{homotopy}: the initial velocity vector of $\pi \circ D_g(\nu_i)$ becomes more and more positively tangent to $c$, the second hitting point converges to second conjugate point, and one parallel transport back by vectors positively tangent to $c$. Such a deformation can be realised within $T^1_gS^2\setminus(-\dot c)$. Moreover, observe that the loop $\Gamma \subset \dot c$ is homotopic to a point in $\dot c$. Thus, $D_g(\nu_i)$ is contractible to a point in $T^1_gS^2\setminus(-\dot c)$. By the homotopy lifting property, also the loop $\nu_i$ is contractible to a point in $S^3\setminus\hat\gamma_c$.
Hence $\link(\nu_i,\hat\gamma_c)=0$ and Claim~1 is proved. \\

\begin{figure}[h]
\begin{center}
\includegraphics[scale=0.15]{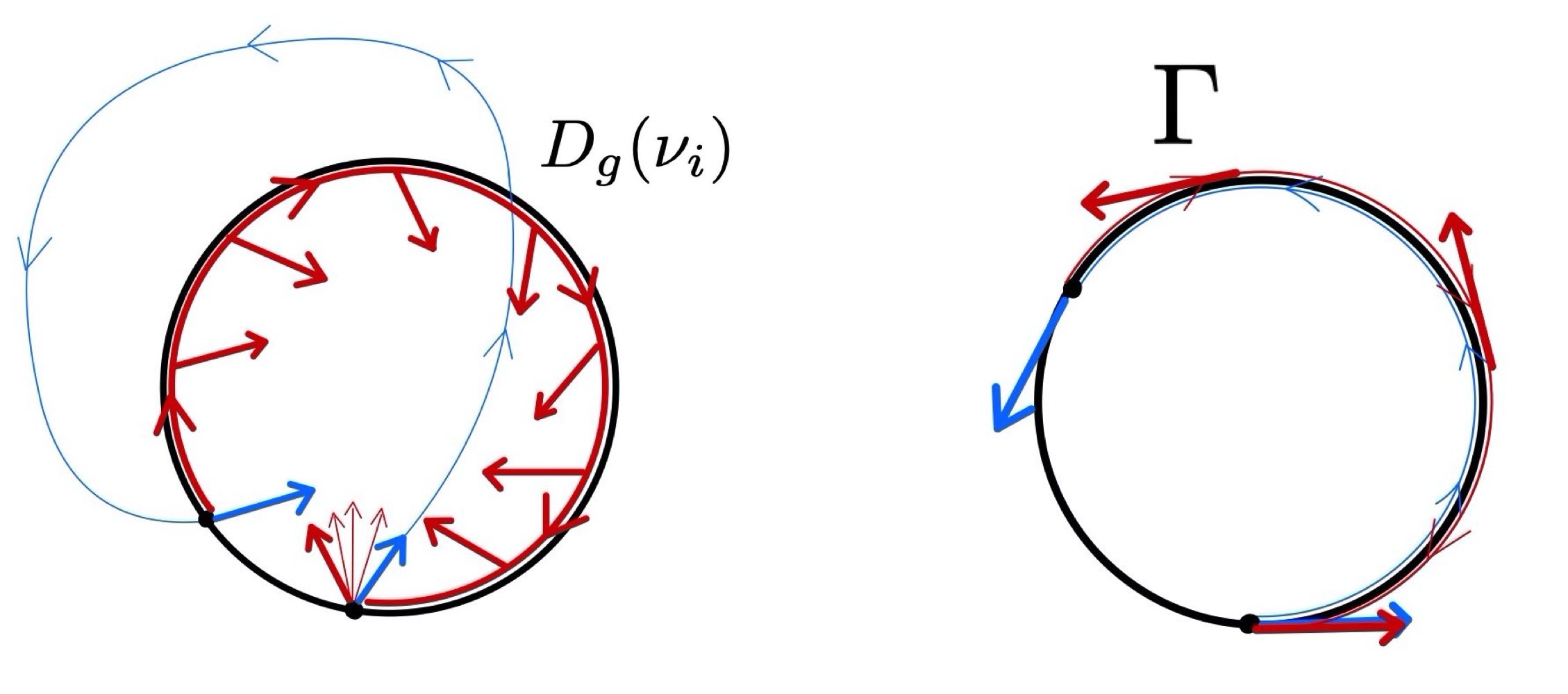}
\caption{The projection through $D_g$ of the loop $\nu_i$ can be deformed into a loop contained in $\dot c$.}
\label{homotopy}
\end{center}
\end{figure}

From~\eqref{calculus_link} and Claim~1 we deduce that
\begin{equation}\label{key relation}
\text{link}(\phi^{[t_0,t_{m(T,x)-1}]}(p)+\tilde{a}\circ P_{\infty}(\varepsilon),\hat{\gamma}_c)=k.
\end{equation}

\noindent {\sc Claim 2.} $\link(\phi^{[t_0,t_{m(T,x)-1}]}(p)+\tilde{a}\circ P_{\infty}(\varepsilon),\hat{\gamma}_c) = M(T,x) + O(1)$. \\

\noindent {\it Proof of Claim 2.} 
Recall that $k(T,x;\hat D)$ was defined as the concatenation of the piece of trajectory $\phi^{[0,T+t^{\hat D}_+(\phi^T(x))-t^{\hat D}_-(x)]}(p)$ with a path $\alpha\subset \hat D$, { called a closing path}. Hence
\begin{equation*}
\begin{aligned}
M(T,x) &= \link(k(T,x;\hat D),\hat{\gamma}_c) \\
&= \# \{ t\in [0,T+t^{\hat D}_+(\phi^T(x))-t^{\hat D}_-(x)] \mid \phi^t(p) \in \hat D \} - 1.
\end{aligned}
\end{equation*}
Now note that 
\begin{equation*}
\begin{aligned}
& \link(\phi^{[t_0,t_{m(T,x)-1}]}(p)+\tilde{a}\circ P_{\infty}(\varepsilon),\hat{\gamma}_c) \\
& \qquad = \# \{ t\in [t_0,t_{m(T,x)-1}] \mid \phi^t(p) \in \hat D \} + Q
\end{aligned}
\end{equation*}
where $Q\in\Z$ satisfies
\begin{equation*}
|Q| \leq \limsup_{\beta \to \tilde{a} \circ P_\infty(\varepsilon)} {\rm int}(\beta,\hat D) + 2
\end{equation*}
where the $\limsup$ is taken with respect to smooth and $C^0$-small perturbations $\beta$ of $\tilde{a} \circ P_\infty(\varepsilon)$, with endpoints fixed, that are transverse to $\hat D$ (up to end points), as $\beta$ $C^0$-converges to $\tilde{a} \circ P_\infty(\varepsilon)$.
Clearly this $\limsup$ exists. Moreover, it is bounded from above by a constant independent of $(T,x)$ since the horizontal translation of $\varepsilon$ in $\R\times[0,\pi]$ is bounded by the constant $2L$, which is independent of $(T,x)$. It follows from all of the above that 
\begin{equation}\label{general control}
\begin{aligned}
& \vert \link(\phi^{[t_0,t_{m(T,x)-1}]}(p) + \tilde{a}\circ P_{\infty}(\varepsilon),\hat{\gamma}_c) - \link(k(T,x;\hat D),\hat{\gamma}_c) \vert \\
& \qquad \leq \# \{ t\in [0,t_0) \mid \phi^t(p) \in \hat D \} \\
& \qquad \qquad + \# \{ t\in (t_{m(T,x)-1},T+t^{\hat D}_+(\phi^T(x))-t^{\hat D}_-(x)] \mid \phi^t(p) \in \hat D \} \\
& \qquad \qquad + \limsup_{\beta \to \tilde{a} \circ P_\infty(\varepsilon)} {\rm int}(\beta,\hat D) + { 3}. 
\end{aligned}
\end{equation}
Since $\hat D$ is a strong global surface of section, we get estimates of the form
\begin{equation*}
\# \{ t\in[a,b] \mid \phi^t(y) \in \hat D \} \leq \frac{b-a}{\tau_{\min}} + 1
\end{equation*}
uniformly in $y \in S^3 \setminus \hat{\gamma}_c$ and in $[a,b] \subset \R$. Here $\tau_{\min} > 0$ denotes the infimum of the return time on $\hat D$. Finally note that the lengths of the intervals $[0,t_0]$ and $[t_{m(T,x)-1},T+t^{\hat D}_+(\phi^T(x))-t^{\hat D}_-(x)]$ are bounded from above by the supremum of the return times on $\hat D$ and on $\tilde{A}_c$. Hence~\eqref{general control} is bounded from above by a constant independent of $(T,x)$, and the proof of Claim~2 is complete. \\

\noindent {\sc Claim 3.} There exists $C\geq0$ independent of $(T,x)$ such that the inequality $|m(T,x)-M(T,x)-N(T,x)| \leq C$ holds for {the closing path chosen in  Lemma~\ref{lemma_application_gen_pos}} defining $k(T,x;\hat D)$. \\

\noindent {\it Proof of Claim 3.} 
Note that $\tilde A_c$ is a $\partial$-strong global surface of section. This is true by positivity of the curvature. 
Hence, the return time back to the interior of $\tilde A_c$ is bounded away from zero and bounded from above. 
We find $\rho>0$ depending only on $\tilde A_c$ such that $\# \{ t\in[a,b] \mid \phi^t(q) \in \tilde A_c \} \leq \rho(b-a)$ holds for every $a<b$ and $q \in S^3 \setminus \partial \tilde A_c$. 
Now note that $\|t^{\hat D}_\pm\|_{L^\infty} < \infty$. This is true because $\hat D$ is $\partial$-strong.
With $L = \max \{ \|t^{\hat D}_-\|_{L^\infty} , \|t^{\hat D}_+\|_{L^\infty} \}$ we get 
\begin{equation*}
\begin{aligned}
0 &\leq m(T,x) - \# \{ t \in [0,T] \mid \phi^t(x) \in \tilde A_c \} \\
& \leq \# \{ t \in [t^{\hat D}_-(x),0] \mid \phi^t(x) \in \tilde A_c \} \\
& \qquad + \# \{ t \in [0,t^{\hat D}_+(\phi^T(x))] \mid \phi^{t+T}(x) \in \tilde A_c \} \\
& \leq 2L\rho.
\end{aligned}
\end{equation*}
By Lemma~\ref{lemma_cohom_disk_versus_annulus} we have $$ M(T,x)+N(T,x) = {\rm int}(k(T,x;\hat D),\tilde A_c), $$ and by Lemma~\ref{lemma_application_gen_pos} there exists $C'\geq0$ independent of $(T,x)$ such that $$ |\# \{ t \in [0,T] \mid \phi^t(x) \in \tilde A_c \}-\text{int}(k(T,x;\hat D),\tilde{A}_c)| \leq C' $$ holds for some particular choice of closing path. The proof is complete if we set $C = C' + 2L\rho$. \\

By {\it Claim 2} and~\eqref{key relation} $$ M(T,x) = k + O(1) $$ and~\eqref{defi p} gives $$ k = \frac{p_1\circ\bar\Psi^{m(T,x)-1}(s,\theta)-s}{2L} + O(1). $$ Thus $$ \left|M(T,x) - \frac{p_1\circ\bar\Psi^{m(T,x)-1}(s,\theta)-s}{2L} \right| = O(1) $$ Up to now the constants $O(1)$ are independent of $(T,x)$ and also of choice of closing paths. By {\it Claim 3} we find $$ \left| m(T,x) - M(T,x) - N(T,x) \right| = O(1) $$ where this last constant is still independent of $(T,x)$, although it depends on the particular choice of closing path. Hence $$ \left| \frac{p_1\circ\bar\Psi^{m(T,x)-1}(s,\theta)-s}{2L} - \frac{p_1\circ\bar\Psi^{M(T,x) + N(T,x)}(s,\theta)-s}{2L} \right| = O(1). $$ All of this implies that $$ \left| M(t,x) - \frac{p_1\circ\bar\Psi^{M(T,x)+N(T,x)}(s,\theta)-s}{2L} \right| = O(1) $$ and the proof is complete.
\end{proof}

We now complete the proof of Theorem~\ref{thm_main_0}. 
{With $x \in S^3 \setminus (\gamma_c\cup\hat{\gamma}_c)$ and $T>0$ arbitrary, below we shall write $k(T,x;\hat D)$ to denote a loop obtained with a closing path given by Lemma~\ref{lemma_horizontal_displacement}.}
Combining {Lemma~\ref{lemma_new_estimate_pinching}} %Corollary~\ref{corollary lift length} 
with Lemma~\ref{lemma_horizontal_displacement}
\begin{equation}
\label{2nd_fraction}
\begin{aligned}
\frac{{\rm int}(k(T,x;\hat D),\tilde A_c)}{\link(k(T,x;\hat D),\hat\gamma_c)} &= \frac{M(T,x)+N(T,x)}{M(T,x)} \\
&= \frac{M(T,x)+N(T,x)}{\dfrac{p_1\circ\bar{\Psi}^{M(T,x)+N(T,x)}(s,\theta)-s}{2L} + O(1)} \\
&\geq \dfrac{2L}{L+2\pi\left(\dfrac{1}{\sqrt{\delta}}-1\right)} + O\left(\frac{1}{M(T,x)+N(T,x)}\right) \, .
\end{aligned}
\end{equation} 
On the other hand by Lemma~\ref{lemma_angular_twist_curvature} and Lemma~\ref{lemma_horizontal_displacement}
\begin{equation}
\label{1st_fraction}
\frac{\Delta\Theta(T,u)}{{\rm int}(k(T,x;\hat D),\tilde A_c)} \geq \frac{T\delta}{\frac{T}{\tau_{\min}} + O(1)} = \delta \tau_{\min} + O\left(\frac{1}{T}\right)
\end{equation}
where $\tau_{\min}$ denotes the infimum of the return time function of $\tilde A_c$. 
By Lemma~\ref{lemma_horizontal_displacement} there is $\chi>0$ independent of $x$ such that $\chi^{-1}T \geq M(T,x)+N(T,x) \geq \chi T$.
Hence, $$ O\left(\frac{1}{M(T,x)+N(T,x)}\right) = O\left(\frac{1}{T}\right) \, . 
%\qquad \text{uniformly in $x \in S^3 \setminus (\gamma_c\cup\hat\gamma_c)$}. 
$$ 
Plugging this together with~\eqref{2nd_fraction} and~\eqref{1st_fraction} we finally arrive at 
\begin{equation}
\label{ineq_wrapping_up}
\frac{\Delta\Theta(T,u)}{\link(k(T,x;\hat D),\hat\gamma_c)} \geq \delta \tau_{\min} \ \dfrac{2L}{L+2\pi\left(\dfrac{1}{\sqrt{\delta}}-1\right)} \ + O \left( \frac{1}{T} \right).
\end{equation}

Let us consider an auxiliary constant $\mu > 1$ and look for $0<\delta<1$ such that
\begin{equation*}
\dfrac{2L}{L+2\pi\left(\dfrac{1}{\sqrt{\delta}}-1\right)} > \frac{\mu}{\delta} \qquad \text{that is}\qquad L(2\delta-\mu) > 2\pi \mu \left(\dfrac{1}{\sqrt{\delta}}-1\right).
\end{equation*}
Since by~\eqref{ineq Kli} we have $L\geq 2\pi$, it suffices to ask for
\begin{equation*}
\mu < 2 \delta \sqrt{\delta}
\end{equation*}
Under this condition on $\mu$, by~\eqref{ineq_wrapping_up} we need
\begin{equation*}
\tau_{\min} \mu > 2\pi
\end{equation*}
in order to get $\kappa(\hat\gamma_c)>2\pi$ and complete the proof. Toponogov's theorem yields
\begin{equation*}
\tau_{\min} \geq 2\pi \left( 2-\frac{1}{\sqrt{\delta}} \right)
\end{equation*}
hence we need
\begin{equation*}
\mu \left( 2-\frac{1}{\sqrt{\delta}} \right) > 1.
\end{equation*} 
Combining all of the above, we need to ask
\begin{equation}\label{ineqs_for_mu}
2\delta \sqrt{\delta} > \mu > \frac{1}{2-\frac{1}{\sqrt{\delta}}} = \frac{\sqrt{\delta}}{2\sqrt{\delta}-1}.
\end{equation}
Note that $\frac{\sqrt{\delta}}{2\sqrt{\delta}-1} \geq 1$ as long as $\frac{1}{4}<\delta\leq1$. Under this condition, a choice of $\mu$ as in~\eqref{ineqs_for_mu} is possible if 
\begin{equation}\label{polynomial_ineq_delta}
2 \ \delta\sqrt{\delta} > \frac{\sqrt{\delta}}{2\sqrt{\delta}-1} \qquad \Leftrightarrow \qquad 2\delta(2\sqrt{\delta}-1) > 1.
\end{equation}
Introducing the variable $x=\sqrt{\delta}$ we get the inequality $P(x) := 2x^2(2x-1)-1 > 0$. The polynomial $P(x)$ has a unique real root $x_*$ because it is negative at its critical points $0$ and $\frac{1}{3}$. This root $x_*$ lies in $(\frac{1}{3},0.85)$ because $P(0.85)>0$. Hence we get $\delta_* = x_*^2 < 0.7225$ and~\eqref{polynomial_ineq_delta} holds for all $\delta \in (\delta_*,1]$, as claimed.

\bibliographystyle{alpha}
\bibliography{biblioFH.bib}

{\footnotesize

{\sc Anna Florio}

Université Paris Dauphine -- Place du Mar\'echal de Lattre de Tassigny, 75775 Paris, France

{\url{florio@ceremade.dauphine.fr}\/}

\medskip

{\sc Umberto Hryniewicz}

RWTH Aachen, Pontdriesch 10-12, 52062 Aachen, Germany

{\url{hryniewicz@mathga.rwth-aachen.de}}

}

\end{document}